\documentclass[11pt]{article}
\usepackage[T1]{fontenc}
\usepackage{lmodern}

\usepackage{amsmath,amsthm,amssymb}
\usepackage[dvipsnames]{xcolor}
\usepackage{enumerate}
\usepackage{graphicx}
\usepackage{cite}
\usepackage{oands}
\usepackage{tikz}
\usepackage{changepage}
\usepackage{bbm}
\usepackage{mathtools}
\usepackage[margin=1in]{geometry}
\usepackage[pagewise,mathlines]{lineno}
\usepackage{appendix}
\usepackage{stmaryrd}
\usepackage{multicol}
\usepackage{microtype}
\usepackage[colorinlistoftodos]{todonotes}
\usepackage{dsfont}

\usepackage[pdftitle={Random walk reflected off of infinity},
  pdfauthor={Ewain Gwynne and Jinwoo Sung},
colorlinks=true,linkcolor=NavyBlue,urlcolor=RoyalBlue,citecolor=PineGreen,bookmarks=true,bookmarksopen=true,bookmarksopenlevel=2,unicode=true,linktocpage]{hyperref}

\theoremstyle{plain}
\newtheorem{thm}{Theorem}[section]

\newtheorem{lem}[thm]{Lemma}
\newtheorem{prop}[thm]{Proposition}

\newtheorem{conj}[thm]{Conjecture}

\def\@rst #1 #2other{#1}
\newcommand\MR[1]{\relax\ifhmode\unskip\spacefactor3000 \space\fi
  \MRhref{\expandafter\@rst #1 other}{#1}}
\newcommand{\MRhref}[2]{\href{http://www.ams.org/mathscinet-getitem?mr=#1}{MR#2}}

\theoremstyle{definition}
\newtheorem{defn}[thm]{Definition}
\newtheorem{remark}[thm]{Remark}

\numberwithin{equation}{section}

\newcommand{\dsb}{\begin{adjustwidth}{2.5em}{0pt}
\begin{footnotesize}}
\newcommand{\dse}{\end{footnotesize}
\end{adjustwidth}}

\newcommand{\ssb}{\begin{adjustwidth}{2.5em}{0pt}}
\newcommand{\sse}{\end{adjustwidth}}

\newcommand{\aryb}{\begin{eqnarray*}}
\newcommand{\arye}{\end{eqnarray*}}
\def\alb#1\ale{\begin{align*}#1\end{align*}}
\def\allb#1\alle{\begin{align}#1\end{align}}
\newcommand{\eqb}{\begin{equation}}
\newcommand{\eqe}{\end{equation}}
\newcommand{\eqbn}{\begin{equation*}}
\newcommand{\eqen}{\end{equation*}}

\newcommand{\BB}{\mathbbm}
\newcommand{\ol}{\overline}

\newcommand{\op}{\operatorname}
\newcommand{\la}{\langle}
\newcommand{\ra}{\rangle}

\newcommand{\frk}{\mathfrak}
\newcommand{\eqD}{\overset{d}{=}}
\newcommand{\ep}{\varepsilon}

\newcommand{\wt}{\widetilde}
\newcommand{\wh}{\widehat}
\newcommand{\mcl}{\mathcal}

\newcommand{\bdy}{\partial}
\newcommand{\rng}{\mathring}

\newcommand{\ora}{\overrightarrow}

\newcommand{\cc}{{\mathbf{c}}}

\let\originalleft\left
\let\originalright\right
\renewcommand{\left}{\mathopen{}\mathclose\bgroup\originalleft}
\renewcommand{\right}{\aftergroup\egroup\originalright}

\title{Random walk reflected off of infinity, with applications to uniform spanning forests and supercritical Liouville quantum gravity}
 \date{ }
 \author{
\begin{tabular}{c} Ewain Gwynne\\[-3pt]\small University of Chicago \end{tabular}
\begin{tabular}{c} Jinwoo Sung\\[-3pt]\small University of Washington \end{tabular}  
}

\begin{document}

\maketitle

\begin{abstract}
Let $\mathcal G$ be an infinite graph --- not necessarily one-ended --- on which the simple random walk is transient. We define a variant of the continuous-time random walk on $\mathcal G$ which reaches $\infty$ in finite time and ``reflects off of $\infty$'' infinitely many times.

We show that the Aldous--Broder algorithm for the random walk reflected off of $\infty$ gives the free uniform spanning forest (FUSF) on $\mathcal G$. Furthermore, Wilson's algorithm for the random walk reflected off of $\infty$ gives the FUSF on $\mathcal G$ on the event that the FUSF is connected, but not in general.

We also apply the theory of random walk reflected off of $\infty$ to study random planar maps in the universality class of supercritical Liouville quantum gravity (LQG), equivalently LQG with central charge in $(1,25)$. Such random planar maps are infinite, with uncountably many ends. We define a version of the Tutte embedding for such maps under which they conjecturally converge to LQG. We also make several conjectures regarding the qualitative behavior of stochastic processes on such maps --- including the FUSF and critical percolation.
\end{abstract}

\tableofcontents

\bigskip
\noindent
\textbf{Acknowledgments.} We thank Leo Bonanno, Zhen-Qing Chen, Peter Kosenko, Greg Lawler, Minjae Park, and Scott Sheffield for helpful discussions. E.G.\ was partially supported by NSF grant DMS-2245832. J.S.\ was partially supported by Kwanjeong Educational Foundation.

\section{Introduction}
\label{sec-intro}

\subsection{Overview}
\label{sec-overview}

Throughout this paper, we let $\mcl G = (\mcl V\mcl G, \mcl E\mcl G , \frk c)$ be a countably infinite connected graph equipped with a conductance function $\frk c : \mcl E\mcl G \to (0,\infty)$ on the (unoriented) edges of $\mcl G$. 
For $x,y\in\mcl V\mcl G$, we write $x\sim y$ if $x$ and $y$ are joined by an edge. We define the stationary measure 
\eqb \label{eqn-stat-measure}
\pi(x) := \sum_{y\sim x} \frk c(x,y) ,\quad\forall x \in\mcl V\mcl G .
\eqe 
For many of our results, we do not require $\mcl G$ is locally finite (i.e., some vertices can have infinite degree) but we always assume that
\eqb  \label{eqn-finite-sum}
\pi(x) < \infty ,\quad \forall x\in \mcl V\mcl G .
\eqe 
When we refer to \textbf{random walk} on $\mcl G$, we mean random walk with the conductances $\frk c$, i.e., the discrete time Markov process which moves from $x$ to $y$ with probability $\frk c(x,y) / \pi(x) = \frk c(y,x) / \pi(x)$. 

For a function $\frk w : \mcl V\mcl G \to (0,\infty)$ on the vertices of $\mcl G$, we can define the continuous time random walk on $\mcl G$ with rate function $\frk w$. For each $x\in\mcl V\mcl G$, this process spends a random, $\op{Exponential}(\frk w(x))$-distributed amount of time at $x$ each time it visits $x$. In other words, the continuous time random walk is the continuous time Markov chain on $\mcl G$ with transition rates $\alpha(x,y) = \frk w(x)\frk c(x,y)  / \pi(x)$. 

If the discrete time random walk on $\mcl G$ is transient and $\frk w(x)$ grows sufficiently quickly as $x\to\infty$, then the continuous time random walk a.s.\ reaches $\infty$ in finite time, i.e., it visits infinitely many vertices before time $T$ for some random, finite $T > 0$.  

In this paper, we will define a canonical version of continuous time random walk on $\mcl G$ which reaches $\infty$ in finite time and then ``reflects off of $\infty$''. This process is defined on all of $[0,\infty)$, and it visits $\infty$ and reflects off infinitely many times. See Theorem~\ref{thm-cont-time-walk} for a precise definition. 

Roughly speaking, the manner in which the random walk reflects off of $\infty$ will depend on which point of the ``boundary at $\infty$'' the walk hits. In particular, the random walk reflected off of $\infty$ is distinct from random interlacements~\cite{sznitman-interlacement,teixeira-interlacement,ct-interlacement-book,drs-interlacement-book}, where one views $\infty$ as a single point. See Remark~\ref{remark-interlacement} for further discussion. 

\begin{remark} \label{remark-dirichlet-form}
Dirichlet form theory provides an analytic approach to constructing reflecting versions of continuous time Markov processes \cite{silverstein-reflected, silverstein-extension1, chen-reflected-dirichlet, kuwae-reflected, schmidt-reflected-form}. The random walk on $\mathcal G$ reflected off of $\infty$ considered in this article realizes the active reflected Dirichlet form on $\mathcal G$ --- see Section~\ref{sec-reflected-dirichlet} --- but our construction is very different. Compared to the Dirichlet form method, our construction has an advantage in giving an elementary description of the path properties of the reflected process.
\end{remark}

Once we have defined the random walk reflected off of $\infty$, we will discuss several of its applications, which we briefly summarize here. 
\begin{itemize}
\item In Theorem~\ref{thm-fsf}, we show that if $\mcl G$ is locally finite, then if one performs the \textbf{Aldous--Broder algorithm}~\cite{broder-algorithm,aldous-algorithm} on $\mcl G$ using random walk reflected off of $\infty$, then one obtains the $\frk c$-weighted free spanning forest of $\mcl G$, abbreviated $\frk c$-FSF (for $\frk c\equiv 1$, this is the free uniform spanning forest). Similarly, we show in Theorem~\ref{thm-wilson} that if one performs \textbf{Wilson's algorithm}~\cite{wilson-algorithm} using random walk reflected off of $\infty$, then one obtains the $\frk c$-FSF if we truncate on the event that the forest consists of a single tree.
The $\frk c$-FSF is usually harder to analyze than the $\frk c$-weighted \emph{wired} spanning forest, in large part because the latter admits nice constructions via the Aldous--Broder and Wilson algorithms~\cite{blps-usf,hutchcroft-wusf}. As such, we expect that our results could give new tools for studying the $\frk c$-FSF. See Section~\ref{sec-fsf0} for further discussion. 
\item Again assuming that $\mcl G$ is locally finite, we define the \textbf{Green's function} and the \textbf{discrete Gaussian free field} with free boundary conditions at $\infty$ on general infinite graphs where the random walk is transient (Definitions~\ref{def-green} and~\ref{def-gff}). 
\item We define a notion of the \textbf{Tutte embedding} for infinite random planar maps, possibly with infinitely many ends, where the vertex set of the embedded map can potentially have infinitely many accumulation points (Section~\ref{sec-tutte0}). We use this to give a precise scaling limit conjecture for embedded random planar maps toward \textbf{supercritical Liouville quantum gravity (LQG)}, i.e., LQG with (Liouville) central charge\footnote{
Throughout this paper, $\cc$ denotes the central charge associated with Liouville conformal field theory, i.e., $\cc = 1+6Q^2$ where $Q = 2/\gamma + \gamma/2$ is the background charge. We note that some papers instead use the central charge of the associated matter field, which is equal to $26-\cc$. 
}
$\cc \in (1,25)$ (Conjecture~\ref{conj-scaling-limit}). This answers~\cite[Problem 4.4]{ag-supercritical-cle4}. 
\item We make several conjectures concerning the qualitative behavior of statistical physics models on random planar maps which approximate supercritical LQG. For example, the free uniform spanning forest on such planar maps should be connected when $\cc > 16$, but not when $\cc < 16$ (Conjecture~\ref{conj-fsf}). Moreover, one should have percolation at criticality on such maps when $\cc < 95/4$, but not when $\cc > 95/4$ (Conjecture~\ref{conj-perc}). 
\end{itemize}

All of the proofs in this paper use only elementary probability theory. The only part of the paper which uses any LQG theory is Section~\ref{sec-supercritical-rpm}, which has no proofs, only conjectures. This section includes preliminary subsections which review the background on LQG and random planar maps necessary to understand the conjectures (Sections~\ref{sec-lqg} and~\ref{sec-rpm}).

\subsection{Energy-minimizing discrete harmonic functions}
\label{sec-harmonic0}

Random walk reflected off of $\infty$ is closely related to energy-minimizing discrete harmonic functions. To explain this, we need to recall some standard definitions. 

\begin{defn} \label{def-harmonic}
We say that $h : \mcl V\mcl G \to \BB R$ is \textbf{discrete harmonic at $x$} if\footnote{If $\mcl G$ is not locally finite, the relation~\eqref{eqn-harmonic} implicitly requires that the sum is convergent, i.e., $\sum_{y\sim x} \frk c(x,y) |h(y) - h(x)|< \infty$.}
\eqb  \label{eqn-harmonic} 
\sum_{y\sim x} \frk c(x,y) (h(y) - h(x)) = 0   .
\eqe  
For $f : \mcl V\mcl G\to\BB R$, we define its \textbf{Dirichlet energy} by
\eqb 
\op{Energy}(f) :=  \sum_{(x,y) \in \mcl E \mcl G} \frk c(x,y)  |f(y) - f(x)|^2  .
\eqe 
\end{defn}

Let $\mcl A\subset \mcl V\mcl G$ be a finite set of vertices and let $\phi : \mcl A\to\BB R$. 
If the random walk on $\mcl G$ is transient, there can be many functions on $\mcl V\mcl G$ which are discrete harmonic on $\mcl V\mcl G\setminus\mcl A$ and agree with $\phi$ on $\mcl A$. For example, the functions 
\eqbn
h_0(x) := 0 \quad \text{and} \quad h_1(x) := \BB P[\text{random walk started from $x$ never hits $\mcl A$}] 
\eqen
are both discrete harmonic and vanish on $\mcl A$. However, as we will show, there is a unique such function which has minimal Dirichlet energy. 
 
\begin{prop} \label{prop-min-harmonic}
Let $\mcl G$ be as in Section~\ref{sec-overview}. 
Let $\mcl A\subset \mcl V\mcl G$ be a non-empty finite set of vertices. 
For each $\phi : \mcl A \to \BB R$, there exists a unique function $h^\phi : \mcl V\mcl G\to \BB R$ such that $h^\phi |_{\mcl A} = \phi$ and $\op{Energy}(h^\phi)$ is minimal among all functions $f : \mcl V\mcl G\to\BB R$ with $f|_{\mcl A} = \phi$. This function $h^\phi$ is discrete harmonic on $\mcl V\mcl G\setminus \mcl A$ and has finite Dirichlet energy. 
Furthermore, the mapping $\phi \mapsto h^\phi$ is linear. 
\end{prop}

We will prove Proposition~\ref{prop-min-harmonic} in Section~\ref{sec-min-harmonic} via an elementary orthogonal projection argument. 

\begin{remark} \label{remark-harmonic-analogy}
To see why energy-minimizing discrete harmonic functions should be relevant for the study of random walk reflected off of $\infty$, consider by way of analogy a finite graph $G$ with a distinguished boundary $\bdy G\subset \mcl V G$ (in this analogy, $\bdy G$ plays the role of $\infty$). Let $A\subset \mcl V G$ be a non-empty set which is disjoint from $\bdy G$ and let $\phi : A\to\BB R$. There can be many functions on $\mcl V G$ which are discrete harmonic on $\mcl V G\setminus (A\cup \bdy G)$ and which agree with $\phi$ on $A$ (such functions have different values on $\bdy G$). However, there is one such function which is energy-minimizing, namely, the one which is discrete harmonic on all of $\mcl V G\setminus A$. The value of this discrete harmonic function at a vertex $x\in\mcl V G$ is given by the expected value of $\phi$ at the first point of $A$ hit by a random walk on $G$ started from $x$ which reflects off of $\bdy G$. If we think of the boundary of $\mcl G$ as ``the set of points at $\infty$'', then it is natural to guess that the hitting measure for random walk reflected off of $\infty$ should be related to energy-minimizing discrete harmonic functions.
\end{remark}

Using Proposition~\ref{prop-min-harmonic}, we can define the energy-minimizing harmonic measure on a finite set of vertices of $\mcl G$. This harmonic measure will be the hitting measure for the random walk reflected off of $\infty$.  

\begin{defn} \label{def-hm}
Let $\mcl A\subset\mcl V\mcl G$ be a finite set of vertices. 
For $y \in \mcl A$, let $\BB 1_y : \mcl A \to\BB R$ be the function which equals 1 at $y$ and zero elsewhere. 
By Proposition~\ref{prop-min-harmonic}, there exists a unique energy-minimizing discrete harmonic function $h^y   : \mcl V\mcl G \to \BB R$ with $h^y|_{\mcl A} = \BB 1_y$. For $x\in \mcl V \mcl G$, we define the \textbf{(energy-minimizing) harmonic measure} on $\mcl A$ viewed from $x$ by 
\eqb  \label{eqn-hm-def} 
\op{hm}_{\mcl A}^x(y) := h^y(x) , \quad\forall y \in \mcl A .
\eqe 
\end{defn}

We will often omit the adjective ``energy-minimizing'' and just refer to ``harmonic measure''. It is easy to check that $\op{hm}_{\mcl A}^x$ is a probability measure on $\mcl A$ (see Lemma~\ref{lem-hm-linear}).

\subsection{Existence and uniqueness of random walk reflected off of \texorpdfstring{$\infty$}{infinity}}
\label{sec-cont-time-walk}

The first main theorem of this paper is the existence and uniqueness of the random walk reflected off of $\infty$. 
 
\begin{thm} \label{thm-cont-time-walk} 
There exists a function $\frk w_* : \mcl V\mcl G \to (0,\infty)$ such that for each function $\frk w : \mcl V\mcl G \to (0,\infty)$ with $\frk w(x) \geq \frk w_*(x) ,\, \forall x\in\mcl V\mcl G$ and for each possible starting point $z \in \mcl V\mcl G$, the following is true. There is a unique (in law) continuous-time stochastic process $X : [0,\infty) \to \mcl V\mcl G \cup \{\infty\}$ with $X_0 = z$, called the \textbf{continuous time random walk on $\mcl G$ reflected off of infinity}, which satisfies the following properties.
\begin{enumerate}[($i$)]
\item \label{item-ae} \textit{(Almost everywhere defined).} For each $t\geq0$, a.s.\ $X_t \in \mcl V\mcl G$ and there exists $\ep > 0$ such that $X_s = X_t$ for each $s\in (t-\ep,t+\ep)$. 
\item \label{item-cont} \textit{(Right continuity).} Almost surely, for every $t\geq 0$ such that $X_t\in\mcl V\mcl G$, there exists $\ep > 0$ such that $X_s = X_t$ for each $s\in [t,t+\ep]$. 
\item \label{item-rw} \textit{(Continuous time random walk).} If we let $\tau := \min\{t > 0:  X_t \not= z \}$, then $\tau$ and $X_\tau$ are independent, $\tau$ has the exponential distribution with rate $\frk w(z)$, and $X_\tau$ has the law of a step of the random walk on $\mcl G$ started at $z$ (i.e., for each $x \sim z$, $\BB P[X_\tau = x] = \frk c(z,x) / \pi(z)$). 
\item \label{item-markov} \textit{(Markov property).} For every $t\geq 0$ and $x\in\mcl V\mcl G$, on the event $\{X_t = x\}$, the conditional law of $\{X_{s+t} \}_{s\geq 0}$ given $\{X_s\}_{s\leq t}$ is the same as the law of $\{X_s\}_{s\geq 0}$ started from $X_0=x$.  
\item \label{item-recurrence} \textit{(Recurrence).} Almost surely, there exist arbitrarily large values of $t\geq 0$ for which $X_t = z$. 
\item \label{item-harmonic} \textit{(Relation to harmonic functions).} Let $\mcl A\subset \mcl V\mcl G$ be a non-empty finite set and let $\phi : \mcl A\to\BB R$. Let $h^\phi : \mcl V\mcl G\to\BB R$ be the energy-minimizing discrete harmonic function as in Proposition~\ref{prop-min-harmonic}. If we let $\tau := \min\{t\geq 0 : X_t \in\mcl A\}$, then for each choice of starting point $z\in\mcl V\mcl G$, 
\eqbn
h^\phi(z) = \BB E_z\left[ \phi(X_\tau) \right] .
\eqen
\end{enumerate}
\end{thm}

As discussed in Section~\ref{sec-overview}, the process $X$ reaches $\infty$ in finite time, then ``reflects off of $\infty$'' and hits additional vertices of $\mcl G$. It can potentially have infinitely many excursions away from $\infty$. 

The choice of $\frk w$ only affects the law of $X$ modulo time change, not the order in which $X$ hits vertices of $\mcl G$. Most of the results and conjectures in this paper will depend only on $X$ viewed modulo time change, so will not depend on the choice of $\frk w$, so long as $\frk w \geq \frk w_*$. 

Properties~\eqref{item-ae} through~\eqref{item-markov} are structural assumptions which justify us in calling $X$ a continuous time random walk. 

If the ordinary random walk on $\mcl G$ (with conductances $\frk c$) is recurrent, then Theorem~\ref{thm-cont-time-walk} is true with $X$ equal to the continuous time random walk on $\mcl G$ with any choice of rate function $\frk w : \mcl V\mcl G\to (0,\infty)$. So, the proposition has non-trivial content only when the ordinary random walk on $\mcl G$ is transient. 
In this case, the recurrence property~\eqref{item-recurrence} is crucial.  
If we did not require this property, we could just take $\frk w$ to be a constant function and take $X$ to be the ordinary continuous time random walk on $\mcl G$. 
The constraint that $\frk w \geq \frk w_*$ is needed for $X$ to be recurrent: we need $\frk w$ to be large enough so that $X$ can reach $\infty$ and come back in a finite amount of time. 
 
Property~\eqref{item-harmonic} tells us that the distribution of the first vertex of $\mcl A$ hit by $X$ is given by the energy-minimizing harmonic measure on $\mcl A$ viewed from $X_0$ (Definition~\ref{def-hm}). See Remark~\ref{remark-harmonic-analogy} for an intuitive discussion of why Property~\eqref{item-harmonic} is natural. 

Property~\eqref{item-harmonic} motivates the key idea in the construction of $X$. We consider an increasing family of finite, connected subgraphs $\{\mcl G_n\}_{n\geq 1}$ whose union is all of $\mcl G$. For each $n$, we define a Markov chain on the 1-neighborhood $B_1 \mcl G_n$ which is the union of $\mcl V\mcl G_n$ and the vertices of $\mcl G\setminus\mcl G_n$ which have a neighbor in $\mcl G_n$. This Markov chain evolves as an ordinary random walk on $\mcl G_n$ until it reaches $B_1\mcl G_n\setminus \mcl G_n$. At this time, it jumps to a point of $\mcl G_n$ chosen according to energy-minimizing harmonic measure on $\mcl G_n$ (defined using Proposition~\ref{prop-min-harmonic}). We show that the transition probabilities for these Markov chains for different values of $n$ are consistent, which allows us to take a limit as $n\to\infty$ to define the random walk reflected off of $\infty$. See the beginning of Section~\ref{sec-rw} for a more detailed outline. 

As noted in Remark~\ref{remark-dirichlet-form}, random walk reflected off of $\infty$ can also be constructed using the theory of Dirichlet forms (as explained in Section~\ref{sec-reflected-dirichlet}). In this paper we give an elementary direct construction.

\begin{remark} \label{remark-end}
In Theorem~\ref{thm-cont-time-walk}, we have set $X_t = \infty$ whenever $X_t$ is not at a vertex of $\mcl G$. However, as we noted in Section~\ref{sec-overview}, the behavior of $X$ after it hits $\infty$ depends on where it hits $\infty$ (roughly speaking, Property~\eqref{item-harmonic} ensures that this is the case). In Proposition~\ref{prop-rw-end} below, we show that (at least if $\mcl G$ is locally finite) we can canonically associate each $t$ such that $X_t = \infty$ with a unique \textbf{end} of $\mcl G$ (see Definition~\ref{def-end} for the definition of an end). Hence, we can view the random walk on $\mcl G$ reflected off of $\infty$ as a function 
\eqb \label{eqn-end-function}
X : [0,\infty) \to  \mcl V\mcl G \cup \{\text{ends of $\mcl G$}\} .
\eqe 
It may be possible to make sense of more refined information about where $X$ hits $\infty$. For example, it may be possible to view $X$ as a function from $[0,\infty)$ to the union of $\mcl V\mcl G$ and the Poisson boundary of $\mcl G$ (see, e.g.,~\cite[Section 14.3]{lyons-peres} for a definition of the Poisson boundary). However, we do not explore this here.
\end{remark}

\begin{remark} \label{remark-reflected-bm}
Random walk reflected off of $\infty$ can be thought of as a discrete analog (on arbitrary graphs) of reflected Brownian motion on an open domain in $\BB R^d$, with the ``boundary at $\infty$'' playing the role of the boundary of the domain. Reflected Brownian motion on an arbitrary bounded domain in $\BB R^d$ was constructed using Dirichlet forms in~\cite{fukushima-reflected-bm,fukushima-dirichlet-forms}. See also~\cite{cf-reflecting-bm} or~\cite[Chapter 7]{cf-book} for some results concerning ``Brownian motion reflected off of $\infty$'' on certain unbounded domains in $\BB R^d$, also obtained using Dirichlet form theory.  
\end{remark}

\begin{remark} \label{remark-interlacement}
Continuous time random walk reflected off of $\infty$ has some similarities to the notion of \textbf{random interlacements}, but the objects are not equivalent. A random interlacement on a transient graph (possibly with conductances) is a Poissonian collection of excursions of random walk away from $\infty$. Random interlacements were first introduced by Sznitman~\cite{sznitman-interlacement} and later generalized to general transient graphs by Teixeira~\cite{teixeira-interlacement}. See the books~\cite{ct-interlacement-book,drs-interlacement-book} for an overview of the subject. In contrast to random interlacements, random walk reflected off of $\infty$ is not a Poissonian collection of excursions away from $\infty$ since, intuitively speaking, if $X_t = \infty$, then the behavior of the process after time $t$ depends on which ``point at $\infty$'' $X_t$ is at (c.f.\ Remark~\ref{remark-end}). In other words, for random interlacements, the ``boundary at $\infty$'' for the random walk is identified to a point, whereas for our random walk reflected off of $\infty$, it is not. 
For graphs where the boundary at $\infty$ is trivial (e.g., graphs where all bounded harmonic functions are constant, c.f.~\cite[Proposition 14.25]{lyons-peres}), we expect that the excursions of $X$ away from $\infty$ are equivalent to a random interlacement. \textit{Update: A version of this statement has subsequently been proven by Yao Yu \cite{yu-reflected-interlacement-equiv}.}
\end{remark}

\subsection{Applications to the free uniform spanning forest}
\label{sec-fsf0}

\noindent
\textit{Throughout this subsection, we assume that $\mcl G$ is locally finite, i.e., each vertex has finite degree. }
\medskip

Using random walk reflected off of $\infty$, one can give two new constructions of the free spanning forest on $\mcl G$ with conductances $\frk c$ (abbreviated $\frk c$-FSF, see Definition~\ref{def-fsf}): one using the Aldous--Broder algorithm (Theorem~\ref{thm-fsf}); and one using Wilson's algorithm, which works only on the event that the $\frk c$-FSF is connected  (Theorem~\ref{thm-wilson}). We also mention a few open problems about the $\frk c$-FSF from the literature for which our results might be relevant in Section~\ref{sec-fsf-problems}.

\subsubsection{Background on spanning forests and loop erasures}
\label{sec-fsf-le}

We review the definition of the $\frk c$-free spanning forest on $\mcl G$ (Definition~\ref{def-fsf}) and the loop erasure of a path in $\mcl G$ (Definition~\ref{def-le}). We also explain how to define the loop erasure of a segment of the random walk reflected off of $\infty$ (Definition~\ref{def-le-infty}).

\begin{defn} \label{def-ust}
Let $G$ be a finite graph equipped with a conductance function $\frk c : \mcl E G \to (0,\infty)$. The \textbf{$\frk c$-spanning tree} of $G$ is the random spanning tree $T\subset  G$ where, for each spanning tree $\frk t$ of $G$,
\eqb  \label{eqn-def-spanning-tree}
\BB P\left[ T = \frk t \right] = \frac{1}{Z} \prod_{e\in \mcl E\frk t} \frk c(e) .
\eqe 
where $Z$ is a proportionality constant chosen to make this a probability measure. 
\end{defn}

For our infinite graph $\mcl G$ (which we recall is equipped with a conductance function $\frk c$), we say that a set $\mcl T\subset \mcl G$ is a \textbf{spanning forest} of $\mcl G$ if $\mcl T$ contains no cycles and $\mcl V\mcl T = \mcl V\mcl G$. 

\begin{defn} \label{def-fsf}
The \textbf{$\frk c$-free spanning forest ($\frk c$-FSF)} of $\mcl G$ is the random spanning forest $\mcl T^{\op{FSF}}$ of $\mcl G$ with the following property. Let $\{\mcl G_n\}_{n\geq 1}$ be an increasing family of finite connected subgraphs of $\mcl G$ whose union is all of $\mcl G$. Let $\mcl T_n$ be the $\frk c$-spanning tree of $\mcl G_n$. Then for any fixed finite set of edges $E\subset\mcl G$, we have 
\eqb \label{eqn-fusf-def}
E \cap \mcl T_n \to E \cap \mcl T^{\op{FSF}} 
\eqe 
in the total variation sense.
\end{defn}

See, e.g.,~\cite[Section 10.1]{lyons-peres} for a proof of the existence and uniqueness of the $\frk c$-FSF. 
We note that if $\frk c(e) = 1$ for each $e\in\mcl E\mcl G$, then $\mcl T^{\op{FSF}}$ is the free uniform spanning forest of $\mcl G$ as defined in~\cite{pemantle-ust}. 
The $\frk c$-free spanning forest is not necessarily connected. Roughly speaking, the reason for this is that for some pairs of vertices $x,y\in\mcl V\mcl G$, the length of the path between $x$ and $y$ in the spanning tree $\mcl T_n$ can go to $\infty$ in distribution as $n\to\infty$.
   
The $\frk c$-FSF is distinct from the \textbf{$\frk c$-wired spanning forest}, which is defined in a similar manner except that we replace the $\frk c$-spanning tree of $\mcl G_n$ by the $\frk c$-spanning tree of the graph $\wh{\mcl G}_n$ obtained by identifying all of the vertices of $\mcl G\setminus \mcl G_n$ to a point. 
The $\frk c$-wired spanning forest is usually more tractable than the $\frk c$-FSF since the $\frk c$-wired spanning forest can be constructed using Wilson's algorithm rooted at $\infty$. See~\cite{blps-usf} for more on free and wired spanning forests. 

We next need to discuss loop erasures, which appear in Wilson's algorithm. 

\begin{defn} \label{def-le}
Let $a,b\in\BB Z$ with $a < b$ and let $Z : [a,b] \cap \BB Z \to \mcl V\mcl G$ be a finite path in $\mcl G$. 
The \textbf{loop erasure} of $Z$ is the path $\op{LE}  : \BB N_0 \to \mcl V\mcl G$ defined inductively as follows. 
Let $\op{LE}(0) = Z_a$. Inductively, suppose $k\geq 1$ and $\op{LE} (k-1)$ has been defined.  
If $\op{LE} (k-1) = Z_b$, we set $\op{LE} (k) = Z_b$. 
Otherwise, let $j_k$ be the last time at which $Z$ visits $\op{LE}(k-1)$ and let $\op{LE}(k) := Z_{j_k + 1}$. 
\end{defn}

Loop erasures still make sense for increments of the random walk reflected off of $\infty$, even if the increment visits $\infty$. 

\begin{defn} \label{def-le-infty}
Let $X$ be the (continuous time) random walk reflected off of $\infty$ on $\mcl G$. Let $a , b \geq 0$ with $a < b$ and $X_a\not=\infty$. 
The \textbf{loop erasure} of $X|_{[a,b]}$ is the path $\op{LE}  : \BB N_0 \to \mcl V\mcl G$ defined inductively as follows. 
Let $\op{LE}(0) = X_a$. Inductively, suppose $k\geq 1$ and $\op{LE} (k-1)$ has been defined.  
If $\op{LE} (k-1) = X_b$, we set $\op{LE} (k) = X_b$. 
Otherwise, let $t_k$ be the supremum of the times $t \in [a,b]$ for which $X_t = \op{LE}(k-1)$ (note that this set is a finite union of left-closed, right-open intervals by Lemma~\ref{lem-to-infty}). Let $\op{LE}(k) :=  X_{t_k}$ be the vertex which $X$ jumps to at time $t_k$.
\end{defn}

In contrast to the loop erasure of a finite path in $\mcl G$, the loop erasure of $X|_{[a,b]}$ might never reach $X_b$. Rather, the loop erasure could instead be an infinite path started from $X_a$. However, it is also possible that the loop erasure of $X|_{[a,b]}$ reaches $b$ in finitely many steps, even if $X$ visits infinitely many vertices during $[a,b]$: indeed, it could be that all of the loops of $X|_{[a,b]}$ which pass through $\infty$ get erased.

\subsubsection{Aldous--Broder algorithm} 
\label{sec-ab0}

The Aldous--Broder algorithm~\cite{broder-algorithm,aldous-algorithm} (see, e.g.,~\cite[Corollary 4.9]{lyons-peres}) gives a way of constructing the $\frk c$-spanning tree of a finite graph using random walk run until its cover time.  
As we now explain, applying the Aldous--Broder algorithm to random walk reflected off of $\infty$ gives the $\frk c$-free spanning forest of $\mcl G$. Let us begin by explaining the construction.

\begin{defn} \label{def-rw-forest}
Fix $z\in\mcl V\mcl G$ and let $X$ be the continuous time random walk on $\mcl G$ reflected off of $\infty$ started from $z$ (Theorem~\ref{thm-cont-time-walk}). For $x\in \mcl V\mcl G$, let $\tau_x := \min\{t\geq 0 : X_t = x\}$ (which is finite a.s.\ since $X$ is recurrent). For $x\in\mcl V\mcl G\setminus \{z\}$, define the \textbf{parent} $a(x)$ of $x$ to be the vertex of $\mcl V\mcl G$ visited by $X$ immediately before time $\tau_x$, i.e., $a(x) = X_{\tau_x -\ep}$ for each small enough $\ep > 0$ (this is well-defined due to~\ref{lem-parent} below). 
The \textbf{Aldous--Broder spanning forest of $\mcl G$} generated by $X$ is the subgraph $\mcl T^{\op{AB}} \subset \mcl G$ with $\mcl V\mcl T^{\op{AB}} = \mcl V\mcl G$ and 
\eqb  \label{eqn-sf-def}
\mcl E\mcl T^{\op{AB}} := \left\{ \{a(x) , x\} : x \in \mcl V\mcl G\setminus \{z\} \right\} \subset \mcl E\mcl G .
\eqe 
\end{defn}

It is clear from the construction that $\mcl T^{\op{AB}}$ is a spanning forest of $\mcl G$. 

\begin{thm}[Aldous--Broder construction of the $\frk c$-FSF] \label{thm-fsf}
For any choice of starting point $z\in\mcl V\mcl G$, the spanning forest of $\mcl G$ generated by $X$ (Definition~\ref{def-rw-forest}) has the same law as the $\frk c$-free spanning forest of $\mcl G$ (Definition~\ref{def-fsf}). 
\end{thm}

We remark that Hutchcroft~\cite{hutchcroft-wusf} gave a construction of the $\frk c$-\emph{wired} spanning forest which is similar to the one in Theorem~\ref{thm-fsf}, but with random interlacements instead of random walk reflected off of $\infty$ (c.f.\ Remark~\ref{remark-interlacement}). 

We will prove Theorem~\ref{thm-fsf} in Section~\ref{sec-fsf}. The proof is based on a limiting argument for the FSF on finite subgraphs of $\mcl G$.

\subsubsection{Wilson's algorithm} 
\label{sec-wilson0}

Wilson's algorithm~\cite{wilson-algorithm} generates the $\frk c$-spanning tree of a finite graph using loop-erased random walk. It is the most computationally efficient way to exactly generate the $\frk c$-spanning tree, and is an extremely important tool in the study of random spanning trees. See~\cite{blps-usf} for a version of Wilson's algorithm for the $\frk c$-wired spanning forest. 

We will now explain a version of Wilson's algorithm~\cite{wilson-algorithm} for the $\frk c$-FSF using random walk reflected off of $\infty$. 
Let $\{x_j\}_{j \in \BB N_0}$ be an enumeration of the vertices of $\mcl G$. 
We will perform Wilson's algorithm on $\mcl G$ with random walk reflected off of $\infty$ started at the vertices $\{x_j\}_{j\in\BB N_0}$. 
To do this, we define an increasing sequence of acyclic subgraphs $\mcl T^{\op{Wil}}_j \subset \mcl G$. 

Let $\mcl T^{\op{Wil}}_0 $ consist of the single vertex $x_0$. 
Inductively, suppose that $j \geq 1$ and $\mcl T^{\op{Wil}}_{j-1}$ has been defined. Let $X^j$ be random walk reflected off of $\infty$ on $\mcl G$ started from $x_j$. Let 
\eqbn
\tau_j  := \inf\{t \geq 0 : X^j_t  \in \mcl T^{\op{Wil}}_{j-1} \} .
\eqen
Note that $\tau_j < \infty$ a.s.\ since $X^j$ is recurrent. 
It is possible that $x_j \in \mcl T^{\op{Wil}}_{j-1}$, in which case $\tau_j = 0$. 

Let $\op{LE}^j$ be the loop erasure of $X^j|_{[0,\tau_j]}$ (Definition~\ref{def-le-infty}). 
Let $\mcl T^{\op{Wil}}_j$ be the union of $\mcl T^{\op{Wil}}_{j-1}$, the vertices visited by $\op{LE}^j$, and the edges traversed by $\op{LE}^j$. By induction, $\mcl T^{\op{Wil}}_j$ is the union of the acyclic subgraph $\mcl T^{\op{Wil}}_{j-1}$ and a simple path which hits $\mcl T^{\op{Wil}}_{j-1}$ only at its terminal endpoint (or not at all, if $\op{LE}^j$ is infinite), so $\mcl T^{\op{Wil}}_j$ is acyclic. 

We define the increasing union
\eqb  \label{eqn-wilson}
\mcl T^{\op{Wil}} := \bigcup_{j=0}^\infty \mcl T^{\op{Wil}}_j  \subset \mcl G .
\eqe 
Then $\mcl T^{\op{Wil}}$ is a spanning forest of $\mcl G$ (here we use that $\{x_j\}_{j\geq 0}$ includes all of the vertices of $\mcl G$). We call it the \textbf{Wilson spanning forest} of $\mcl G$. We have the following elementary observation (c.f.\ the discussion just after Definition~\ref{def-le-infty}).

\begin{lem} \label{lem-wilson-connected}
The Wilson spanning forest $\mcl T^{\op{Wil}}$ is connected if and only if each of the paths $\op{LE}^j$ hits only finitely many vertices of $\mcl G$. 
\end{lem}
\begin{proof}
If each $\op{LE}^j$ hits only finitely many vertices of $\mcl G$, then by construction each $\mcl T_j^{\op{Wil}}$ is a finite subtree of $\mcl G$. From this, it is immediate that their increasing union $\mcl T^{\op{Wil}}$ is connected.

Conversely, assume that there exists $j \in \BB N$ such that $\op{LE}^j$ hits infinitely many vertices of $\mcl G$. Let $J$ be the smallest such $j$. Then $\op{LE}^J$ does not reach $\mcl T_{J-1}^{\op{Wil}}$ in finite time, so $\mcl T_J^{\op{Wil}}$ has two connected components. For each $j \in \BB N$, the set $\mcl T_j^{\op{Wil}}$ is the union of $\mcl T_{j-1}^{\op{Wil}}$ and a path which intersects $\mcl T_{j-1}^{\op{Wil}}$ at at most one vertex. From this, we see that the number of connected components of $\mcl T_j^{\op{Wil}}$ is increasing in $j$, so every $\mcl T_j^{\op{Wil}}$ for $j\geq J$, and hence also $\mcl T^{\op{Wil}}$, is not connected.
\end{proof}

\begin{thm}[Wilson's algorithm for the $\frk c$-FSF] \label{thm-wilson}
The following two random variables agree in law.
\begin{itemize}
\item The random variable which equals the Wilson spanning forest $\mcl T^{\op{Wil}}$ if $\mcl T^{\op{Wil}}$ is connected, and which equals $\emptyset$ otherwise. 
\item The random variable which equals the $\frk c$-free spanning forest $\mcl T^{\op{FSF}}$ if $\mcl T^{\op{FSF}}$ is connected, and which equals $\emptyset$ otherwise. 
\end{itemize} 
\end{thm}

As in the case of Theorem~\ref{thm-fsf}, the proof of Theorem~\ref{thm-wilson} is based on a limiting argument started from Wilson's algorithm on finite subgraphs of $\mcl G$. See Section~\ref{sec-wilson}. 

\begin{remark}  \label{remark-wilson}
The spanning forests $\mcl T^{\op{Wil}}$ and $\mcl T^{\op{FSF}}$ do not agree in law without the truncation on the event that they are connected. The reason is that $\mcl T^{\op{Wil}}$ can have a finite connected component, whereas $\mcl T^{\op{FSF}}$ cannot. Indeed, as explained in the proof of Lemma~\ref{lem-wilson-connected}, on the event that $\mcl T^{\op{Wil}}$ is not connected, there exists $J \geq 1$ such that $\mcl T_J^{\op{Wil}}$ has both a finite connected component and an infinite connected component. Condition on the event that this is the case for some fixed deterministic $J$, and let $A$ be the finite connected component. Assume also that $A$ is such that $\mcl G\setminus A$ is connected. Since the random walk reflected off of $\infty$ on $\mcl G$ follows any finite path in $\mcl G$ with positive probability, we see that there exists $J' \geq J+1$, depending on $A$, with the following property.
With positive conditional probability given $\mcl T_J^{\op{Wil}}$, every vertex of $\mcl G\setminus A$ which has a neighbor in $A$ is part of an infinite connected component of $\mcl T_{J'}^{\op{Wil}}$ which is disjoint from $A$. In this case, $A$ is also a connected component of $\mcl T^{\op{Wil}}$.  

It may be possible to find a new definition of the loop erasure of random walk reflected off of $\infty$ where the loop erasure is allowed to reflect off of $\infty$, possibly infinitely many times (with our current definition, the loop erasure stops upon reaching $\infty$). We expect that if one finds just the right way to reflect the loop erasure off of $\infty$, then Wilson's algorithm with the extended loop erasure gives exactly the $\frk c$-FSF, even on the event when the $\frk c$-FSF is not connected. We will not explore this further in this paper.
\end{remark}

\subsubsection{Open problems}
\label{sec-fsf-problems}

As noted just after Definition~\ref{def-fsf}, the $\frk c$-FSF is typically harder to analyze than the $\frk c$-wired spanning forest ($\frk c$-WSF), in large part because the latter can be constructed via versions of Wilson's algorithm or the Aldous--Broder algorithm. As such, it is natural to wonder whether our results can be used to answer questions about the $\frk c$-FSF which have previously been solved for the $\frk c$-WSF. Numerous open problems about the $\frk c$-FSF can be found, e.g., in~\cite[Section 10]{lyons-peres} and~\cite[Section 15]{blps-usf}. Here are a few examples of problems for which our results might be relevant.
\begin{itemize}
\item Is it the case that for every infinite graph $\mcl G$ and conductance function $\frk c$, the number of connected components of the $\frk c$-FSF is a.s.\ equal to a deterministic constant~\cite[Question 15.7]{blps-usf}? The analogous statement for the $\frk c$-WSF was proved in~\cite[Theorem 9.4]{blps-usf}. 
\item On a vertex-transitive graph $\mcl G$, the expected degree of any (equivalently, every) vertex in the wired uniform spanning forest is 2. For the free uniform spanning forest, this expected degree is related to the first $\ell^2$ Betti number of $\mcl G$, introduced in~\cite{atiyah-discrete-groups}. See~\cite[Proposition 10.10]{lyons-peres} and the discussion just after. Can our results (see in particular Proposition~\ref{prop-fsf-edge} below) be used to compute the expected degree of a given vertex in the $\frk c$-FSF on any particular graphs $\mcl G$? See also~\cite[Question 10.11]{lyons-peres} for an important related question in the case of Cayley graphs. 
\end{itemize}

\subsection{Green's function and Gaussian free field}
\label{sec-green0}

Let $X$ be the random walk reflected off of $\infty$ on $\mcl G$ and for a non-empty finite set $\mcl A\subset\mcl G$, let
\eqbn
\tau_{\mcl A} := \min\left\{t\geq 0 : X_t\in \mcl A\right\} .
\eqen
For $y \in \mcl V\mcl G$, let $N_{\mcl A}(y)$ be the number of distinct visits of $\{X_t\}_{t\geq 0}$ to $y$ which are strictly before time $\tau_{\mcl A}$, i.e.,
\eqb \label{eqn-number-of-visits}
N_{\mcl A}(y) := \#\left\{ \text{connected components of $X^{-1}(y) \cap [0,\tau_{\mcl A}]$} \right\} . 
\eqe
 
\begin{defn} \label{def-green}
We define the \textbf{Green's function for random walk on $\mcl G$ reflected off of $\infty$ and killed upon hitting $\mcl A$} by 
\eqbn  \label{eqn-green-def}
G_{\mcl A}(x,y) := \BB E_x[N_{\mcl A}(y)]  .
\eqen
\end{defn}

We show in Section~\ref{sec-green} that the Green's function of Definition~\ref{def-green} satisfies analogs of the standard properties of the Green's function on a finite graph, e.g., harmonicity (Lemma~\ref{lem-green-finite}) and symmetry in the two variables, up to factors of $\pi(x)$ (Lemma~\ref{lem-green-reversible}). Using the Aldous--Broder construction of the $\frk c$-FSF (Theorem~\ref{thm-fsf}), we also establish an analog of Kirkhoff's effective resistance formula (see, e.g., \cite[Section 4.2]{lyons-peres}) for the $\frk c$-FSF. 
 
\begin{prop} \label{prop-fsf-edge}
Assume that $\mcl G$ is locally finite. 
Let $e = \{x,y\} \in \mcl E\mcl G$. Let $G_y(\cdot,\cdot)$ be the Green's function for random walk on $\mcl G$ reflected off of $\infty$ and killed upon hitting $y$ (Definition~\ref{def-green}). 
Let $\mcl T^{\op{FSF}}$ be the $\frk c$-free spanning forest on $\mcl G$ (Definition~\ref{def-fsf}). Then 
\eqb \label{eqn-edge-contain}
\BB P\left[ e \in \mcl T^{\op{FSF}} \right] = \frac{\frk c(x,y) G_y(x,x)}{\pi(x)} .
\eqe 
\end{prop}

The Green's function also allows us to define a new version of the discrete Gaussian free field on $\mcl G$. 
 
\begin{defn} \label{def-gff}
Let $\mcl A\subset\mcl V\mcl G$ be a non-empty finite set.
The \textbf{discrete Gaussian free field on $\mcl G$ with zero boundary condition on $\mcl A$ and free boundary condition at $\infty$} is the centered Gaussian random function $\Psi : \mcl V\mcl G\to \BB R$ with covariances 
\eqb \label{eqn-gff-cov}
\op{Cov}(\Psi(x) , \Psi(y)) = \frac{G_{\mcl A}(x,y)}{\pi(y)} 
\eqe 
where $G_{\mcl A}(x,y)$ is as in~\eqref{eqn-green-finite}.
\end{defn}

The right side of~\eqref{eqn-gff-cov} defines a valid covariance kernel by Lemma~\ref{lem-green-reversible} below. 
We expect that for random planar maps in the supercritical LQG universality class, the scaling limit of the process $\Psi$ as in Definition~\ref{def-gff} under the Tutte embedding (Section~\ref{sec-tutte0}) is the continuum Gaussian free field. See Conjecture~\ref{conj-scaling-limit} for a more precise statement.

\subsection{Tutte embedding of multi-ended random planar maps}
\label{sec-tutte0}

Various types of random planar maps are conjectured or proven to converge in the scaling limit to \textbf{Liouville quantum gravity (LQG)}, a one-parameter family of random fractal surfaces. We provide some background and references on LQG in Section~\ref{sec-lqg}. One way to formulate the convergence is to first embed the random planar map in the plane in some canonical way, then ask for convergence of objects associated with the embedded map. Examples of possible embeddings include Tutte embedding, circle packing, Smith embedding, and Cardy embedding. We refer to~\cite{hs-cardy-embedding,gms-tutte,gms-poisson-voronoi,bgs-smith} for rigorous convergence results under various embeddings and, e.g.,~\cite{shef-kpz,dkrv-lqg-sphere,shef-zipper,bp-lqg-notes} for precise conjectures. 

Recent papers~\cite{ghpr-central-charge,ag-supercritical-cle4,bgs-supercritical-crt} have defined random planar maps which are believed to converge in the scaling limit to \textbf{supercritical LQG}, i.e., LQG with the central charge parameter $\cc \in (1,25)$, or equivalently with coupling constant $\gamma \in \BB C$ with $|\gamma|  =2$ (see Definition~\ref{def-supercritical-map} for an explicit example of such a random planar map). These random planar maps are infinite, with infinitely many ends, with high probability. Hence, in order to formulate the conjectured convergence to supercritical LQG rigorously, one needs to define embeddings of infinite-ended random planar maps. Constructing such embeddings is~\cite[Problem 4.4]{ag-supercritical-cle4}. 

Using random walk reflected off of $\infty$, we can define a notion of the Tutte embedding~\cite{tutte-embedding} of general, possibly multi-ended random planar maps. This embedding has the property that the coordinates of the embedding function are energy-minimizing discrete harmonic functions (Proposition~\ref{prop-min-harmonic}). 

Let $\mcl G$ be a (possibly infinite) connected planar map whose vertices and faces all have finite (but not necessarily bounded) degree. Let $\mcl G$ be equipped with a conductance function $\frk c : \mcl E\mcl G\to (0,\infty)$ and a distinguished face, which we call the \textbf{external face}. We assume that $\mcl G$ is drawn in the plane in such a way that the external face is the unique unbounded face. We emphasize that the vertex set of $\mcl G$ is allowed to have accumulation points. 

We will now define the Tutte embedding of $\mcl G$ into the closed unit disk $\ol{\BB D}$.
Let $\bdy\mcl G \subset\mcl V\mcl G$ be the set of vertices on the boundary of the external face.  
Fix marked vertices $\BB x\in\mcl V\mcl G \setminus \bdy\mcl G$ and $\BB y\in \bdy \mcl G$.
We trace the boundary of the external face in counterclockwise order starting and ending at $\BB y$. 
Let $y_1,\dots,y_m$ be the distinct vertices of $\bdy\mcl G$, enumerated by the last time that they are visited by this exploration, so that $y_m = \BB y$.\footnote{It could be that some vertices of $\bdy \mcl G$ are visited more than once when we trace the boundary of the external face in counterclockwise order. More precisely, the number of times that $y$ is visited is equal to the number of prime ends of the external face corresponding to the boundary point $y$. However, we only include each vertex once in the enumeration, so that $y_0,\dots,y_m$ are distinct.}
Let $\op{hm}_{\mcl G}^{\BB x}$ be the harmonic measure on $\bdy\mcl G$ as viewed from $\BB x$, as in Definition~\ref{def-hm}. 
For $k = 1,\dots,m$, we define
\eqb  \label{eqn-tutte-bdy}
H(y_k) := \exp\left( 2 \pi i \sum_{j=1}^k \op{hm}_{\bdy\mcl G}^{\BB x} (y_j ) \right) .
\eqe 
That is, $H$ maps $\bdy\mcl G$ to the unit circle in such a way that harmonic measure on $\bdy\mcl G$ viewed $\BB x$ approximates the uniform measure on the unit circle and $H(\BB y) =1$. 

We then define $H(x)$ for $x\in\mcl V\mcl G\setminus \bdy\mcl G$ in such a way that $H$ is the unique energy-minimizing function on $\mcl V\mcl G$ with the given values on $\bdy\mcl G$, as in Proposition~\ref{prop-min-harmonic} (we apply the proposition to the real and imaginary parts of $H$ separately). By Proposition~\ref{prop-min-harmonic}, $H$ is discrete harmonic on $\mcl V\mcl G\setminus \bdy\mcl G$.
We call $H$ the \textbf{Tutte embedding of $(\mcl G,\BB x,\BB y)$}. 

We establish some basic properties of the Tutte embedding --- including the convexity of the image of the faces and an approximation result for Tutte embeddings of finite submaps --- in Section~\ref{sec-tutte}. We give a scaling limit conjecture for random planar maps to supercritical LQG under the Tutte embedding in Section~\ref{sec-tutte-conj}.

\section{Energy-minimizing discrete harmonic functions}
\label{sec-harmonic}

In this section, we will first prove Proposition~\ref{prop-min-harmonic} (the existence and uniqueness of energy-minimizing discrete harmonic functions) in Section~\ref{sec-min-harmonic}. We will prove some elementary properties of energy-minimizing discrete harmonic functions in Section~\ref{sec-harmonic-property}. In Section~\ref{sec-harmonic-conv}, we will establish a useful approximation result for energy-minimizing discrete harmonic functions on $\mcl G$ by discrete harmonic functions on finite subgraphs of $\mcl G$. 

We first review some fairly standard notation. 
Let $\mcl G = (\mcl V\mcl G, \mcl E\mcl G , \frk c)$ be a countably infinite connected graph with conductances as in Section~\ref{sec-overview}. 
We assume that all subgraphs of $\mcl G$ are equipped with conductances given by the restriction of $\frk c$. 

Let $\ora{\mcl E}\mcl G$ denote the set of oriented edges of $\mcl G$.  
The \textbf{discrete gradient} of a function $f : \mcl V\mcl G\to\BB R$ is the function $\nabla f : \ora{\mcl E}\mcl G \to \BB R$ defined by
\eqb 
\nabla f(x,y) = \nabla_{\mcl G} f(x,y) := \frk c(x,y) ( f(y) - f(x) ) . 
\eqe 
The \textbf{energy} of a function $\theta : \ora{\mcl E}\mcl G \to \BB R$ is 
\eqb  \label{eqn-energy-def}
\op{Energy}(\theta) = \op{Energy}_{\mcl G}(\theta) := \sum_{e \in \ora{\mcl E}\mcl G} \frac{ |\theta(e)|^2 }{\frk c(e)} . 
\eqe 
Then, in the notation of Definition~\ref{def-harmonic}, for $f : \mcl V\mcl G \to \BB R$ we have 
\eqbn
\op{Energy}(f)  = \op{Energy}(\nabla f )    .
\eqen 
 
\subsection{Existence and uniqueness}
\label{sec-min-harmonic}

The proof of Proposition~\ref{prop-min-harmonic} is a short and elementary orthogonal projection argument.   

Recall the definition~\eqref{eqn-energy-def} of the energy of a function $\theta$ on oriented edges of $\mcl G$. 
Let $\Theta \subset \ell^2 (\ora{\mcl E}\mcl G , \frk c^{-1})$ be the set of functions $\theta : \ora{\mcl E}\mcl G \to \BB R$ such that $\theta(u,v) = -\theta(v,u)$ for all edges $(u,v) \in \ora{\mcl E}\mcl G$ and $\op{Energy}(\theta)  < \infty$. Then $\Theta$ is a Hilbert space with the $\ell^2$ inner product
\eqbn
\la \theta, \eta \ra_\Theta = \sum_{e \in \ora{\mcl E}\mcl G} \frac{ \theta(e) \eta(e) }{\frk c(e)}  .
\eqen

The following elementary lemma is a more-or-less standard fact about functions on the edges of graphs (see, e.g.,~\cite[Section 4]{gjn-orthodiagonal} for closely related statements on finite graphs). We give a proof since we are unaware of a reference for the precise statement we need. 

\begin{lem} \label{lem-subspace}
Fix a finite, non-empty set $\mcl A\subset\mcl V\mcl G$. 
Let
\eqb \label{eqn-subspace-def}
  \rng\Theta_{\mcl A}  := \left\{ \theta \in \Theta : \text{$\theta = \nabla f$ for some $f : \mcl V\mcl G \to \BB R$ with $f|_{\mcl A} \equiv 0$} \right\} . 
\eqe 
Then $\rng\Theta_{\mcl A}$ is a closed linear subspace of $\Theta$. The orthogonal complement $\rng\Theta_{\mcl A}^\perp$ is described as follows. We say that $\eta \in \Theta$ is an \textbf{cycle modulo $\mcl A$} if there exists a path $P : [0,N]\cap\BB Z \to\mcl V\mcl G$ such that either $P(0) = P(N)$ or $P(0) , P(N) \in \mcl A$ with the following property. For each $e = (x,y) \in \ora{\mcl E}\mcl G$, $\eta(e)$ is equal to the number of times $i$ such that $(P(i-1) ,P(i)) = (x,y)$ minus the number of times $i$ such that $(P(i-1), P(i)) = (y,x)$. Then $\rng\Theta_{\mcl A}^\perp$ is the closure of the linear span of the set of cycles modulo $\mcl A$. 
\end{lem}
\begin{proof}
It is clear that $\rng\Theta_{\mcl A}$ is a vector subspace of $\Theta$. To see that $\rng\Theta_{\mcl A}$ is closed, consider a sequence of functions $\theta_n = \nabla f_n$ in $\rng\Theta_{\mcl A}$ which converges in $\ell^2$ to a function $\theta\in \Theta$. For $x\in \mcl V\mcl G$, there is a finite path $P :[0,N]\cap\BB Z\to\mcl V\mcl G$ from $\mcl A$ to $x$. We have 
\eqbn
f_n(x) = \sum_{j=1}^N \frac{ \theta_n(P(i-1) ,P(i)) }{\frk c(P(i-1) ,P(i))}
\eqen
From this representation and the $\ell^2$ convergence $\theta_n\to\theta$, we get that $\{f_n(x)\}_{n\in\BB N}$ converges pointwise to a function $f : \mcl V\mcl G\to\BB R$ with $f|_{\mcl A} = 0$ and $\nabla f = \theta$. Thus, $\theta \in \rng\Theta_{\mcl A}$. 

It remains to prove the description of $\rng\Theta_{\mcl A}^\perp$. If $\eta$ is a cycle modulo $\mcl A$, arising from a path $P$, then for every $\theta = \nabla f \in \rng\Theta_{\mcl A}$, we have $\la \theta, \eta \ra_\Theta = f(P(N)) - f(P(0)) = 0$. Hence the closed linear span of the cycles modulo $\mcl A$ is contained in $\rng\Theta_{\mcl A}^\perp$. 

To show that the closed linear span of the cycles modulo $\mcl A$ is equal to $\rng\Theta_{\mcl A}^\perp$, it suffices to show that if $\theta\in\Theta$ is orthogonal to every cycle modulo $\mcl A$, then $\theta \in \rng\Theta_{\mcl A}$.  
For such a $\theta$, define $f :\mcl V\mcl G \to\BB R$ by $f|_{\mcl A} \equiv 0$ and for $x\in\mcl V\mcl G\setminus\mcl A$, 
\eqbn
f(x) :=  \sum_{j=1}^N \frac{ \theta (P(i-1) ,P(i)) }{\frk c(P(i-1) ,P(i))} 
\eqen
where $P : [0,N]\cap\BB Z\to\mcl V\mcl G$ is any path in $\mcl V\mcl G$ from a point of $\mcl A$ to $x$. This is well-defined (independently of the choice of path) since $\theta$ is orthogonal to every cycle modulo $\mcl A$, and we have $\nabla f = \theta$. Hence $\theta \in \rng\Theta_{\mcl A}$. 
\end{proof}

\begin{proof}[Proof of Proposition~\ref{prop-min-harmonic}]
We first prove the existence of $h^\phi$. 
For a function $\phi : \mcl A\to \BB R$, let $\wh\phi : \mcl V\mcl G\to\BB R$ be the function which agrees with $\phi$ on $\mcl A$ and is identically zero on $\mcl V\mcl G\setminus \mcl A$. Since $\mcl A$ is finite and $\pi(x) <\infty$ for all $x\in \mcl A$, we have $\op{Energy}(\wh\phi) < \infty$, i.e., $\nabla\wh\phi \in \Theta$. 
Define $\rng\Theta_{\mcl A}$ as in Lemma~\ref{lem-subspace} and let $\op{proj}_{\rng\Theta_{\mcl A}}$ denote the orthogonal projection onto $\rng\Theta_{\mcl A}$. 

By the definition~\eqref{eqn-subspace-def} of $\rng\Theta_{\mcl A}$, there exists $f^\phi : \mcl V\mcl G\to\BB R$ with $f^\phi|_{\mcl A} \equiv 0$ such that 
\eqb \label{eqn-min-proj}
\nabla f^\phi = \op{proj}_{\rng\Theta_{\mcl A}} \nabla \wh\phi .
\eqe 
Moreover, the function $f^\phi$ is unique since $\mcl G$ is connected, so a function on $\mcl V\mcl G$ is determined by its gradient and its value at any given vertex.

By the defining property of orthogonal projection, 
\eqbn
\op{Energy}(\nabla \wh\phi - \nabla f^\phi)  \leq \op{Energy}(\nabla\wh\phi - \nabla g) 
\eqen
for every $g : \mcl V\mcl G\to\BB R$ with $g|_{\mcl A} \equiv 0$, and moreover $f^\phi$ is the unique function which vanishes on $\mcl A$ with this property. Consequently, if we define 
\eqb  \label{eqn-min-def}
h^\phi := \wh\phi - f^\phi 
\eqe 
then $h^\phi|_{\mcl A} = \wh\phi|_{\mcl A} = \phi$, and $h^\phi$ is the unique function with this property for which $\op{Energy}(h^\phi)$ is minimal. 

To see that $h^\phi$ is harmonic on $\mcl V\mcl G\setminus\mcl A$, let us first observe that by the Cauchy-Schwarz inequality and~\eqref{eqn-finite-sum}, for every $x\in \mcl V\mcl G$, 
\eqbn
\sum_{y\sim x} \frk c(x,y) |h^\phi(y) - h^\phi(x)| 
\leq \sqrt{\pi(x) \op{Energy}(h^\phi)} < \infty. 
\eqen
Now, suppose by way of contradiction that $h^\phi$ fails to be harmonic at some $x \in \mcl V\mcl G \setminus \mcl A$. Let 
\eqbn
\wt h^\phi(x) := \frac{1}{\pi(x)} \sum_{y\sim x} \frk c(x,y) h^\phi(y)
\eqen
and $\wt h^\phi(z) := h^\phi(z)$ for $z \in \mcl V\mcl G \setminus \{x\}$. By assumption, $\wt h^\phi(x) \not= h^\phi(x)$. Hence
\eqbn
\sum_{y\sim x} \frk c(x,y) |h^\phi(y) - \wt h^\phi(x)|^2 < \sum_{y\sim x} \frk c(x,y) |h^\phi(y) - h^\phi(x)|^2, 
\eqen
which can easily be seen by differentiating. 
Since $h^\phi(z) = \wt h^\phi(z)$ for all $z\in\mcl V\mcl G\setminus \{x\}$, it follows that $\op{Energy}(\wt h^\phi) < \op{Energy}(h^\phi)$, which contradicts the minimality of $\op{Energy}(h^\phi)$. 

Finally, we check the linearity of $\phi\mapsto h^\phi$. Let $\phi,\psi : \mcl V\mcl G\to\BB R$ and $a\in\BB R$. Since orthogonal projection is linear, we have by~\eqref{eqn-min-proj} and~\eqref{eqn-min-def} that
\eqbn
\nabla h^{\phi + a \psi} 
 = \nabla \wh\phi  + a \nabla \wh\psi - \nabla f^\phi - a \nabla f^\psi  
= \nabla h^\phi + a \nabla h^\psi .
\eqen
Since $\mcl G$ is connected, this implies that $h^{\phi  +a \psi} = h^\phi + a  h^\psi + C$ for some constant $C$. Since $h^{\phi  +a \psi}$ and $ h^\phi + a  h^\psi $ both agree with $\phi + a \psi$ on $\mcl A$, we must have $C = 0$. 
\end{proof}

\subsection{Basic properties}
\label{sec-harmonic-property}

We record some basic properties of the energy-minimizing harmonic functions if Proposition~\ref{prop-min-harmonic}. 

\begin{lem} \label{lem-harmonic-compat}
Let $\mcl A\subset \mcl B \subset \mcl V\mcl G$ be non-empty finite vertex sets.  
Let $\phi : \mcl A\to\BB R$ and let $h^\phi$ be the energy-minimizing function on $\mcl G$ with $h^\phi|_{\mcl A} = \phi$, as in Proposition~\ref{prop-min-harmonic}. 
Then $h^\phi $ is also the energy-minimizing function on $\mcl G$ whose restriction to $\mcl B$ is given by $h^\phi|_{\mcl B}$.
\end{lem}
\begin{proof}
Let $\wt h^\phi$ be the energy-minimizing function on $\mcl G$ whose restriction to $\mcl B$ is given by $h^\phi|_{\mcl B}$. 
Then $\wt h^\phi|_{\mcl A} = h^\phi|_{\mcl A} = \phi$ and (since $\wt h^\phi$ is energy-minimizing), $\op{Energy}(\wt h^\phi) \leq \op{Energy}(h^\phi)$. Since $h^\phi$ is the \emph{unique} energy-minimizing function which agrees with $\phi$ on $\mcl A$ (Proposition~\ref{prop-min-harmonic}), we get that $\wt h^\phi = h^\phi$. 
\end{proof}

\begin{lem}[Maximum principle] \label{lem-max-principle}
Let $\mcl A\subset  \mcl V\mcl G$ be non-empty and finite. 
Let $\phi : \mcl A\to\BB R$ and let $h^\phi$ be the energy-minimizing function on $\mcl G$ with $h^\phi|_{\mcl A} = \phi$, as in Proposition~\ref{prop-min-harmonic}.  
Then 
\eqb \label{eqn-max-principle}
\max_{x\in\mcl V\mcl G} h^\phi(x) = \max_{y\in\mcl A}  \phi(y)  
\eqe 
and the same holds with ``min'' in place of ``max''. 
\end{lem}

We note that Lemma~\ref{lem-max-principle} is not immediate from the fact that $h^\phi$ is discrete harmonic. Indeed, if we let
\eqbn
\frk h(x) := \BB P\left[ \text{random walk started from $x$ eventually hits $\mcl A$} \right] 
\eqen
then $\frk h$ is discrete harmonic on $\mcl V\mcl G\setminus\mcl A$ and is identically equal to 1 on $\mcl A$. However, if random walk on $\mcl G$ is transient, then $\frk h$ can take values in $(0,1)$.

\begin{proof}[Proof of Lemma~\ref{lem-max-principle}]
Let $M := \max_{x\in \mcl A} \phi(y)$ and define $\wt h^\phi(x) := \min\{M , h^\phi(x)\}$. Then $\wt h^\phi|_{\mcl A} = h^\phi|_{\mcl A} = \phi$. Furthermore, if $e$ is an oriented edge of $\mcl G$, then 
\eqb  \label{eqn-mp-grad}
|\nabla \wt h^\phi(e)| \leq |\nabla h^\phi(e)| , 
\eqe 
with equality if and only if $ h^\phi  \leq M$ at both endpoints of $e$. Since $h^\phi$ minimizes discrete Dirichlet energy among all functions which agree with $\phi$ on $\mcl A$, we get $\op{Energy}(\wt h^\phi) \geq \op{Energy}(h^\phi)$. This is possible only if the inequality in~\eqref{eqn-mp-grad} is always an equality, i.e., $\max_{x\in\mcl V\mcl G} h^\phi(x) \leq M$. The statement for the minimum follows by symmetry.  
\end{proof}

Recall the definition of the (energy-minimizing) harmonic measure $\op{hm}_{\mcl A}^x(\cdot)$ on $\mcl A$ from Definition~\ref{def-hm}. 

\begin{lem} \label{lem-hm-linear}
For each $x\in \mcl V\mcl G$, the harmonic measure $\op{hm}_{\mcl A}^x$ is a probability measure on $\mcl A$. 
Moreover, for each $\phi : \mcl A \to\BB R$, the energy-minimizing discrete harmonic function $h^\phi$ satisfies
\eqb  \label{eqn-hm-linear} 
h^\phi(x)= \sum_{y\in \mcl A} \phi(y) \op{hm}_{\mcl A}^x(y) .
\eqe 
\end{lem}
\begin{proof}
By the maximum principle (Lemma~\ref{lem-max-principle}), $\op{hm}_{\mcl A}^x(y) \in [0,1]$. By the linearity part of Proposition~\ref{prop-min-harmonic}, $\sum_{y \in \mcl A} h^y(x) = 1$ for each $x\in \mcl G$ and~\eqref{eqn-hm-linear} holds. 
\end{proof}

\subsection{Finite approximation}
\label{sec-harmonic-conv}

We show in this subsection that energy-minimizing discrete harmonic functions on $\mcl G$ arise as pointwise limits of discrete harmonic functions on finite subgraphs of $\mcl G$.

\begin{prop} \label{prop-harmonic-conv}
Assume that we are in the setting of Proposition~\ref{prop-min-harmonic}. 
Let $\{\mcl G_n\}_{n\geq 1}$ be an increasing family of finite, connected subgraphs of $\mcl G$ whose union is all of $\mcl G$ such that $\mcl A\subset \mcl V\mcl G_n$ for every $n\geq 1$. 
For $n\geq 1$, let $h_n^\phi : \mcl V\mcl G_n\to\BB R$ be the unique function which agrees with $\phi$ on $\mcl A$ and is $\mcl G_n$-discrete harmonic\footnote{If $x\in\mcl G_n$ has at least one neighboring vertex in $\mcl G$ which is not in $\mcl G_n$, then the condition that $h_n^\phi$ is $\mcl G_n$-discrete harmonic at $x$ is different from the condition that $h_n^\phi$ is $\mcl G$-discrete harmonic at $x$.}
on $\mcl V\mcl G_n\setminus\mcl A$. 
Let $h^\phi : \mcl V\mcl G \to \BB R$ be the energy-minimizing function with $h^\phi|_{\mcl A} \equiv \phi$, as in Proposition~\ref{prop-min-harmonic}. 
Then $\lim_{n\to\infty} h_n^\phi(x) = h^\phi(x)$ for each $x\in\mcl V\mcl G$.
\end{prop}
\begin{proof}
By the maximum principle, the functions $h_n^\phi$ are bounded above in absolute value by $\sup_{x\in\mcl A} |\phi(x)|$. By a compactness argument, for any sequence of positive integers tending to infinity, we can find a subsequence $\mcl N$ and a function $\wt h : \mcl V\mcl G \to\BB R$ with $\wt h|_{\mcl A} = \phi$ such that $h_n^\phi \to \wt h$ pointwise on $\mcl V\mcl G$ as $\mcl N\ni n \to\infty$. We need to show that $\wt h = h^\phi$. 

Let $k \in \BB N$. 
Since $h_n^\phi$ minimizes Dirichlet energy on $\mcl G_n$ among all functions on $\mcl V\mcl G_n$ which agree with $\phi$ on $\mcl A$, if $n\geq k$, then 
\alb
\sum_{e\in\mcl E\mcl G_k} \frac{|\nabla h_n^\phi(e)|^2}{\frk c(e)}
 \leq \op{Energy}_{\mcl G_n}(h_n^\phi) 
 \leq \op{Energy}_{\mcl G_n}(h^\phi|_{\mcl V\mcl G_n} ) 
 \leq \op{Energy}_{\mcl G }(h^\phi  )  .
\ale
Sending $n\to \infty$ along the sequence $\mcl N$ gives
\alb
\sum_{e\in\mcl E\mcl G_k} \frac{|\nabla \wt h(e)|^2}{\frk c(e)} 
 \leq \op{Energy}_{\mcl G }(h^\phi  )  .
\ale
By sending $k\to\infty$ and using the monotone convergence theorem, it follows that $\op{Energy}_{\mcl G }(\wt h  )  \leq \op{Energy}_{\mcl G }(h^\phi  )  $. By the uniqueness of the energy minimizer (Proposition~\ref{prop-min-harmonic}), this implies $\wt h = h^\phi$.  
\end{proof}

\section{Construction of random walk reflected off of \texorpdfstring{$\infty$}{infinity}}
\label{sec-rw}

The goal of this section is to prove Theorem~\ref{thm-cont-time-walk}. We first give an outline of the proof. Consider a family of finite, connected subgraphs $\{\mcl G_n\}_{n\geq 1}$ of $\mcl G$ which increase to all of $\mcl G$. 
To construct $X$, we start in Section~\ref{sec-discrete-markov} by defining a discrete-time Markov chain $Y^n$ on the radius-1 graph distance neighborhood
\eqb \label{eqn-1-nbd}
B_1\mcl G_n := \left\{x \in\mcl V\mcl G : \text{either $x \in\mcl V\mcl G_n$ or $x \sim y$ for some $y\in\mcl V\mcl G_n$} \right\} .
\eqe 
 whose transition probabilities started from $x$ are described as follows. If $x\in\mcl V\mcl G_n$, then $Y^n$ takes a step according to random walk on $\mcl G$. If $x\in B_1\mcl G_n\setminus \mcl G_n$, then instead $Y^n$ jumps to a point of $\mcl V\mcl G_n$ sampled from harmonic measure viewed from $x$. This Markov chain will describe the ordered sequence of vertices of $\mcl G_n$ hit by $X$. These Markov chains for different values of $n$ satisfy a consistency condition (Lemma~\ref{lem-rw-consistent}) which allows us to couple them all together. 
 
Using this coupling, we can define a discrete-time version of $X$, which we call $Y$, which is a function from a totally ordered set $\Xi$ to $\mcl V\mcl G$ (Section~\ref{sec-discrete-reflected}). We will then introduce exponential holding times in Section~\ref{sec-cont-time-def} and use these to define the continuous time random walk $X$. In Section~\ref{sec-cont-time-proof} we will check that $X$ satisfies the properties in Theorem~\ref{thm-cont-time-walk} and that the process satisfying these properties is unique. 

Section~\ref{sec-cont} gives the proofs of some basic properties of $X$ which are intuitively obvious but require a small amount of justification. 
Section~\ref{sec-discrete-reflected} discusses connections to Silverstein extensions of Dirichlet forms.

\subsection{Approximating by a discrete time Markov chain}
\label{sec-discrete-markov}

\begin{figure}[ht!]
\begin{center}
\includegraphics[width=\textwidth]{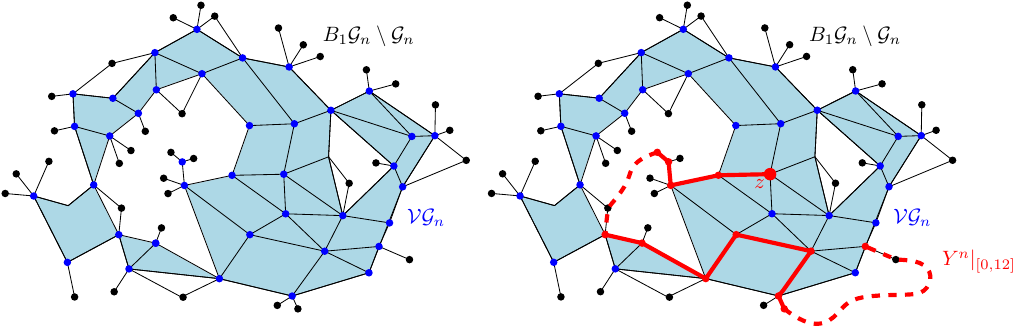}  
\caption{\label{fig-markov-chain-approx} \textbf{Left:} The subgraph $\mcl G_n$ (blue) and the vertex set $B_1\mcl G_n\setminus \mcl G_n$ (black). We have drawn a planar graph for clarity, but of course $\mcl G_n$ is not required to be planar. The vertices of $\mcl G_n \setminus B_1\mcl G_n$ are contained in the three white regions. \textbf{Right:} Twelve steps of the Markov chain $Y^n$ on $B_1\mcl G_n$ defined in~\eqref{eqn-transition-walk} and~\eqref{eqn-transition-comp}. This Markov chain evolves as a random walk on $\mcl G_n$ until it hits a vertex of $B_1\mcl G_n\setminus \mcl G_n$, at which point it jumps to a vertex of $\mcl G_n$ sampled according to harmonic measure. We have drawn a simple trajectory for clarity, but there is nothing stopping $Y^n$ from visiting the same vertex multiple times.
}
\end{center}
\end{figure}

Let $\{\mcl G_n\}_{n\geq 1}$ be an increasing family of finite, connected subgraphs of $\mcl G$ whose union is all of $\mcl G$.  
  
For each $n\geq 1$ and each $x\in\mcl V\mcl G_n$, we define the harmonic measure $\op{hm}_{ \mcl G_n}^x$ on $\mcl V\mcl G_n$ viewed from $x$ as in~\eqref{eqn-hm-def} with $\mcl A = \mcl V\mcl G_n$. 

Fix $n\geq 1$. Let $B_1 \mcl G_n \subset\mcl V\mcl G$ be as in~\eqref{eqn-1-nbd}. 
We define a Markov chain $Y^n = \{Y_j^n\}_{j\geq 0}$ on $B_1\mcl G_n$ with transition probabilities $p_n(x,y)$ specified as follows. If $x \in \mcl V\mcl G_n$, we define 
\eqb  \label{eqn-transition-walk}
p_n(x,y) = \frk c(x,y) / \pi(x) ,\quad \text{if $y\sim x$} \quad \text{and} \quad \text{$p_n(x,y) = 0$ otherwise}. 
\eqe 
If $x\in B_1\mcl G_n \setminus  \mcl G_n$, we define 
\eqb  \label{eqn-transition-comp}
p_n(x,y) = \op{hm}_{\mcl G_n}^x(y) ,\quad \text{if $y \in \mcl V\mcl G_n$} \quad \text{and} \quad \text{$p_n(x,y) = 0$ otherwise}. 
\eqe 
In other words, $Y^n$ moves according to random walk on $\mcl G$ whenever it is in $\mcl G_n$, and whenever it is not in $\mcl G_n$ it jumps to a point of $\mcl G_n$ chosen according to harmonic measure. See Figure~\ref{fig-markov-chain-approx} for an illustration. 

For $x\in\mcl V\mcl G_n$, we write $\BB P_x$ for the probability measure under which $Y^n$ starts from $Y^n_0 = x$, and $\BB E_x$ for the corresponding expectation (we slightly abuse notation by not explicitly including $n$ in this notation). 

\begin{remark} \label{remark-tilde-process}
If $\mcl G$ is locally finite, then $B_1\mcl G_n$ is a finite set of vertices. In general, $B_1\mcl G_n$ could be infinite. In this case, it will occasionally be useful to consider a Markov chain on the finite state space $\mcl V\mcl G_n$. For this purpose, let $\rho_0 = 0$ and let $\rho_k$ for $k\geq 1$ be the $k$th smallest time $j$ for which $Y^n_j \in \mcl V\mcl G_n$.
Then $\wt Y_k^n := Y_{\rho_k}^n$ is a Markov chain on $\mcl V\mcl G_n$, with transition probabilities given by 
\eqb  \label{eqn-tilde-transition}
\wt p_n(x,y) := \frac{\frk c(x,y)}{\pi(x)} + \sum_{\substack{z \in B_1\mcl G_n\setminus \mcl G_n \\ z\sim x}} \frac{\frk c(x,z)}{\pi(x)} \op{hm}_{\mcl G_n}^z(y) .
\eqe  
The process $\wt Y^n$ is an irreducible Markov chain on a finite state space, so is recurrent. It follows that $Y^n$ is also recurrent. 
\end{remark}

We now show that for any finite set $\mcl A\subset\mcl V\mcl G_n$, the distribution of the first hitting time of $\mcl A$ by $Y^n$ is given by harmonic measure on $\mcl A$ (by definition, this is true if $\mcl A = \mcl V\mcl G_n$). 
This is a consequence of the consistency property of energy-minimizing functions on $\mcl G$ (Lemma~\ref{lem-harmonic-compat}). 
  
\begin{lem} \label{lem-rw-harmonic}
Let $\mcl A\subset \mcl V\mcl G$ be a finite set and let $\phi : \mcl A\to\BB R$. Let $h^\phi : \mcl V\mcl G\to\BB R$ be the energy-minimizing function with $h^\phi|_{\mcl A} = \phi$, as in Proposition~\ref{prop-min-harmonic}. Let
\eqbn
\tau_n := \min\{j \geq 0 : Y^n_j \in\mcl A\} . 
\eqen
For each $n\in\BB N$ such that $\mcl A\subset\mcl G_n$ and each choice of starting point $x \in\mcl V\mcl G_n$, 
\eqbn
h^\phi(x) = \BB E_x\left[ \phi(X^n_{\tau_n} ) \right] .
\eqen
\end{lem}
\begin{proof}
Fix $n\in\BB N$ such that $\mcl A\subset\mcl V\mcl G_n$. Define
\eqb  \label{eqn-rw-exit-phi}
f_n^\phi (x) := \BB E_x\left[ \phi(Y_{\tau_n}^n)   \right] \quad \forall x \in   B_1\mcl G_n .
\eqe 
We need to show that
\eqb \label{eqn-rw-agree-phi}
f_n^\phi(x) = h^\phi(x) ,\quad \forall x \in  B_1\mcl G_n . 
\eqe
  
By the Markov property of $\{Y_j^n\}_{j \geq 0}$ and the definition of the transition probabilities~\eqref{eqn-transition-walk},  
\eqb \label{eqn-rw-bulk}
f_n^\phi(x) = \frac{1}{\pi(x)} \sum_{y\sim x} \frk c(x,y) f_n^\phi(y),\quad\forall x  \in \mcl V\mcl G_n \setminus \mcl A .
\eqe
Furthermore, by the Markov property of $\{Y_j^n\}_{j\geq 0}$ and the definition of the transition probabilities~\eqref{eqn-transition-comp}, 
\eqb  \label{eqn-rw-comp}
f_n^\phi(x) = \sum_{y \in \mcl V\mcl G_n}  \op{hm}_{\mcl G_n}^x(y) f_n^\phi(y) , \quad \forall x\in  B_1\mcl G_n \setminus \mcl G_n  . 
\eqe 

Since the function $h^\phi$ is discrete harmonic on $\mcl V\mcl G\setminus \mcl A$, the relation~\eqref{eqn-rw-bulk} holds with $h^\phi$ in place of $f_n^\phi$. By Lemma~\ref{lem-harmonic-compat} (applied with $\mcl B = \mcl V\mcl G_n$), $h^\phi $ is the energy-minimizing function on $\mcl G$ with values on $\mcl V\mcl G_n$ given by $h^\phi|_{\mcl V\mcl G_n}$. By this and Lemma~\ref{lem-hm-linear} (applied with $\mcl V\mcl G_n$ in place of $\mcl A$), the relation~\eqref{eqn-rw-comp} also holds with $h^\phi$ in place of $f_n^\phi$. 

We have $h^\phi|_{\mcl A}=  f_n^\phi|_{\mcl A} = \phi$. By linearity,~\eqref{eqn-rw-bulk} and~\eqref{eqn-rw-comp} hold with 
\eqbn
g := h^\phi - f_n^\phi
\eqen
in place of $f_n^\phi$. Therefore, for the transition probabilities $\wt p_n$ as in Remark~\ref{remark-tilde-process}, we have for each $x \in\mcl V\mcl G_n \setminus \mcl A$,
\alb
\sum_{y \in \mcl V\mcl G_n} \tilde p_n(x,y) g(y) 
&= \sum_{y\in\mcl V\mcl G_n} \frac{\frk c(x,y)}{\pi(x)} g(y) 
+ \sum_{\substack{z \in B_1\mcl G_n\setminus \mcl G_n \\ z\sim x}} \frac{\frk c(x,z)}{\pi(x)} \sum_{y\in\mcl V\mcl G_n} \op{hm}_{\mcl G_n}^z(y) g(y) \\
&= \sum_{y\in\mcl V\mcl G_n} \frac{\frk c(x,y)}{\pi(x)} g(y) 
+ \sum_{\substack{z \in B_1\mcl G_n\setminus \mcl G_n \\ z\sim x}} \frac{\frk c(x,z)}{\pi(x)} g(z) \quad \text{(by~\eqref{eqn-rw-comp} for $g$)} \\
&= g(x) \quad \text{(by~\eqref{eqn-rw-bulk} for $g$)} .
\ale 
From this, we see that $g = h^\phi - f_n^\phi$ must attain its maximum and minimum values on the (finite) set $\mcl V\mcl G_n$ at points of $ \mcl A$. Since $g \equiv 0$ on $\mcl A$, it follows that $g\equiv 0$ on $\mcl V\mcl G_n$. By~\eqref{eqn-rw-comp} for $g$, this implies that $g\equiv 0$ on all of $B_1\mcl G_n$, i.e., $h^\phi \equiv f_n^\phi$ on $ B_1\mcl G_n$.  
\end{proof}

We can now prove the following consistency relation for the Markov chains $Y^n$, which will allow us to couple all of them together in such a way that for $m\leq n$, $Y^m$ and $Y^n$ hit vertices of $\mcl V\mcl G_k$ in the same order (see Lemma~\ref{lem-rw-coupling}).

\begin{lem} \label{lem-rw-consistent}
The Markov chains $\{Y_j^n\}_{n\geq 0}$ defined above satisfy the following consistency condition. 
Let $1 \leq m\leq n $ and assume that $\{Y_j^m\}_{j \geq 0}$ and $\{Y_j^n\}_{j\geq 0}$ both start at the same vertex $z\in \mcl V\mcl G_m$. Let $J_0^{m,n}  = 0$ and inductively let
\eqb \label{eqn-rw-consistent-def}
J_k^{m,n}  := 
\begin{cases}
J_{k-1}^{m,n} + 1 \quad &\text{if $Y_{J_{k-1}^{m,n}}^n \in \mcl V\mcl G_m$} \\
\min\{j >  J_{k-1}^{m,n} : Y^n_j \in \mcl V\mcl G_m \} \quad &\text{if $Y_{J_{k-1}^{m,n}}^n \notin \mcl V\mcl G_m$} .
\end{cases}
\eqe
Then $\{Y^n_{J_k^{m,n}}\}_{k\geq 0} \eqD \{Y^m_k\}_{k\geq 0}$.
\end{lem}
\begin{proof}
By the definition of $J_k^{m,n}$, we always have $Y^n_{J_k^{m,n}} \in B_1\mcl G_m$. 
By the definition of $Y^n$, if $n = m$, then $J_k^{m,m} = k$ for all $k\geq 0$ and the lemma statement is vacuous. So, we can assume without loss of generality that $n > m$. 
 
Let $\tau_n = \tau_n(m)$ be the smallest $ j\geq 0$ for which $Y_j^n \in \mcl V\mcl G_m$.
By Lemma~\ref{lem-rw-harmonic} (applied with $\mcl A = \mcl V\mcl G_m$ and $\phi = \BB 1_y$) together with Definition~\ref{def-hm}, 
\eqb \label{eqn-rw-agree}
\BB P_x \left[ Y_{\tau_n}^n = y   \right]  =\op{hm}_{\mcl G_m}^x(y) ,\quad \forall x \in B_1\mcl G_n ,\quad \forall y \in \mcl V\mcl G_m . 
\eqe
 
To lighten notation, write
\eqbn
J_k = J_k^{m,n} . 
\eqen
Let $k\geq 0$. If $Y_{J_k}^n \in \mcl G_m$, then by definition~\eqref{eqn-rw-consistent-def} we have $J_{k+1} = J_k+1$. By the strong Markov property, for each $x\in \mcl G_m$ and $y\in B_1\mcl G_m$, 
\eqbn
\BB P_z\left[ Y_{J_{k+1}}^n = y \,|\, Y_{J_k}^n = x\right] = \frac{\frk c(x,y)}{\pi(x)} = p_m(x,y) ,
\eqen
where here we recall~\eqref{eqn-transition-walk}. By the strong Markov property and~\eqref{eqn-rw-agree}, for each $x\in   B_1\mcl G_m \setminus \mcl G_m$ and each $y \in \mcl V\mcl G_m$, 
\eqbn
\BB P_z\left[ Y_{J_{k+1}}^n = y \,|\, Y_{J_k}^n = x\right] = \op{hm}_{\mcl G_m}^x(y)  = p_m(x,y)
\eqen
where in the last equality we use the definition~\eqref{eqn-transition-comp}. Thus, $\{Y^n_{ J_k} \}_{k\geq 0}$ has the same transition probabilities as $\{Y_k^m\}_{k\geq 0} $. By the strong Markov property both processes are Markovian, so if we start both processes at $z$, then they agree in law.  
\end{proof}

\subsection{Discrete time random walk reflected off of \texorpdfstring{$\infty$}{infinity}}
\label{sec-discrete-reflected}

Fix $z\in\mcl V\mcl G$ and let 
\eqb  \label{eqn-vertex-in}
n_z := \min\{n\geq 1 : z \in \mcl V\mcl G_n\} .
\eqe 
For $n \geq m \geq n_z$, define the times $\{J_k^{m,n}\}_{k \geq 0}$ as in Lemma~\ref{lem-rw-consistent}. We recall that
\eqbn
J_k^{m,m} = k .
\eqen

\begin{lem} \label{lem-rw-coupling} 
For each $z\in\mcl V\mcl G$, there is a coupling of $\{Y^n\}_{n \geq n_z}$, each started from $Y^n_0 = z$, such that
\eqb  \label{eqn-rw-coupling} 
Y_{J_k^{m,n}}^n = Y_k^m \quad \forall n \geq m \geq n_z, \quad \forall k  \in \BB N_0 
\eqe
and the following Markov property is satisfied. For each $m\geq n_z$ and each $k  \in \BB N_0$, the processes $\{Y_j^n\}_{ j \leq J_k^{m,n}}$ for $n \geq m+1$ are conditionally independent from $Y^m$ given $\{Y^m_j\}_{ j\leq k}$. 
\end{lem}
\begin{proof}
We construct a coupling of $\{Y^n\}_{n_z \leq n \leq N}$ satisfying the conditions in the lemma statement with $m,n$ restricted to lie in $[n_z,N]\cap \BB Z$ by induction on $N$. In the base case $N =n_z$, there is nothing to prove. Inductively, suppose $N\geq n_z+1$ and that we are given a coupling of $\{Y^n\}_{n_z \leq n \leq N-1}$ satisfying the desired properties. By Lemma~\ref{lem-rw-consistent}, we have $\{Y^N_{J_k^{N-1,N}}\}_{k\geq 0} \eqD \{Y^{N-1}_k\}_{k\geq 0}$. So, we can extend our given coupling to a coupling of $\{Y^n\}_{n_z \leq n \leq N}$ as follows. We first set $Y^N_{J_k^{N-1,N}} = Y^{N-1}_k$ for each $k\geq 0$. Then, conditional on $\{Y^n\}_{n_z \leq n \leq N-1}$, we sample the whole process $Y^N$ from its conditional law given  $\{Y^N_{J_k^{N-1,N}}\}_{k\geq 0}$. 

By the definition~\eqref{eqn-rw-consistent-def}, for each $n \geq m' \geq m \geq n_z$, we have that $\{J_k^{m,n}\}_{m\geq 0}$ is given by the formula~\eqref{eqn-rw-consistent-def} with $\{Y^n_{ J_j^{m' ,n}}\}_{j \geq 0 } $ in place of $\{Y^n_j\}_{j\geq 0}$. Using this fact and the definition of our coupling, we can complete the induction and get, for every $N\geq n_z$, a coupling satisfying the conditions in the lemma statement with $m,n$ restricted to lie in $[n_z,N]\cap \BB Z$. Sending $N\to\infty$ concludes the proof. 
\end{proof}  

In the rest of this section, we fix a coupling as in Lemma~\ref{lem-rw-coupling}. We denote the joint law of  $\{Y^n\}_{n \geq n_z}$ under this coupling by $\BB P_z$, and we write $\BB E_z$ for the corresponding expectation. 

We will now define a discrete time random walk on $\mcl G$ reflected off of $\infty$. Since this random walk can reach $\infty$ and come back infinitely many times, it cannot be viewed as a function from $\BB N_0$ to $\mcl V\mcl G$. Instead, we will define this process as a function $Y : \Xi \to \mcl V\mcl G$, where $\Xi$ is a random countable totally ordered set which we now define.  

We define an equivalence relation on pairs of integers $(n,j)$ with $n\geq n_z$ and $j\geq 0$ by setting $(n,j) \sim (m,k)$ if either $m = n$ and $k = j$; $m < n$ and $j = J_k^{m,n}$; or the same is true with the roles of $(n,j)$ and $(m , k)$ reversed. We write $\Xi$ for the set of equivalence classes and we denote the equivalence class of $(n,j)$ by $\xi = [(n,j)]$.

If $\xi$ is an equivalence class, then there exists a unique element $(m,k) \in \xi$ whose first coordinate is minimal; and for each $n\geq m$, $\xi$ contains exactly one element of $\{n\} \times \BB N_0$. 
We define a total ordering on $\Xi$ as follows.  We say that $\xi \leq \wt \xi$ if and only if for each $n \geq n_z$ such that both $\xi$ and $\wt \xi$ contain elements of $\{n\}\times \BB N_0$, say $(n,j)\in\xi$ and $(n,\wt j)\in\xi$, we have $j\leq \wt j$. Since each sequence of times $\{J_k^{m,n}\}_{k\geq 0}$ is monotone increasing, our choice of coupling from Lemma~\ref{lem-rw-coupling} implies that $\xi \leq \wt \xi$ if and only if there exists $n\geq n_z$ and $(n,j) \in \xi$ and $(n,\wt j) \in \wt\xi$ such that $j\leq \wt j$. 

If $\xi\in\Xi$ is an equivalence class, we define 
\eqb  \label{eqn-process-on-xi}
Y_\xi := Y^n_j ,\quad \text{for any $(n,j) \in \xi$} .
\eqe 
This definition does not depend on the choice of equivalence class representation by our choice of coupling in Lemma~\ref{lem-rw-coupling}. We call the process $\{Y_\xi\}_{x\in\Xi}$ the \textbf{discrete time random walk reflected off of $\infty$}.

For future reference, we note that the definition of the times $J_k^{m,n}$ from~\eqref{eqn-rw-consistent-def} together with the definition of $\Xi$ implies that
\eqb \label{eqn-xi-exists}
\text{If $Y_\xi \in \mcl V\mcl G_n$, then $\exists j \in \BB N_0$ such that $(n,j) \in \xi$.}
\eqe 

\subsection{Continuous time random walk reflected off of \texorpdfstring{$\infty$}{infinity}}
\label{sec-cont-time-def}

Fix a function $\frk w : \mcl V\mcl G \to (0,\infty)$. We aim to construct the continuous time random walk on $\mcl G$ reflected off of $\infty$ which spends an exponential$(\frk w(x))$ amount of time at each vertex $x\in\mcl V\mcl G$ whenever it visits $x$.

For $z\in\mcl V\mcl G$, let the Markov chains $\{Y^n\}_{n\geq n_z}$ on $B_1\mcl G_n$ started from $z$ be coupled together as in Lemma~\ref{lem-rw-coupling}. 
Recall the ordered set $\Xi$ used in the definition of the discrete time random walk reflected off of $\infty$, as defined just above~\eqref{eqn-process-on-xi}. 
Conditional on $\{Y^n\}_{n\geq n_z}$, let $\{T_\xi\}_{\xi\in \Xi}$ be conditionally independent random variables indexed by $\Xi$ where the conditional distribution of $T_\xi$ is exponential with parameter $\frk w(Y_\xi)$. 
For $n\geq n_z$, we define the process $X^n : [0,\infty) \to B_1\mcl G_n$ by 
\eqb  \label{eqn-cont-time-def}
X^n_t = Y^n_k  = Y_{[(n,k)]} ,\quad \forall k \in\BB N_0 ,\quad \forall t \in \left[ \sum_{j=1}^{k-1}  T_{[(n,j)]} , \sum_{j=1}^k T_{[(n,j)]} \right) .
\eqe 
Then $X^n$ is a continuous time version of $Y^n$ which spends an exponential$(\frk w(x))$ amount of time at each state $x\in B_1\mcl G_n$. Equivalently, the transition rate for $X^n$ from $x$ to $y$ is $\frk w(x) p_n(x,y)$, where $p_n(x,y)$ is the transition probability from~\eqref{eqn-transition-walk} or~\eqref{eqn-transition-comp}.  

We want to take an appropriate limit of $X^n$ as $n\to\infty$. For this purpose we need to choose the function $\frk w$ in such a way that when $n \geq m$ are large, the total amount of time that $X^n$ spends outside of $\mcl V\mcl G_m$ is small. 
We do this in the following lemma.

\begin{lem} \label{lem-time-finite}
There exists a function $\frk w_* : \mcl V\mcl G\to (0,\infty)$ such that if $\frk w(x) \geq \frk w_*(x)$ for all but finitely many $x \in \mcl V\mcl G$, then the following is true. 
For each choice of starting point $z\in\mcl V\mcl G$, it holds $\BB P_z$-a.s.\ that 
\eqb  \label{eqn-time-finite}
\sum_{\xi \in \Xi} T_\xi= \infty  \quad \text{and} \quad \sum_{\xi \in\Xi , \xi\leq \eta} T_\xi < \infty ,\quad \forall \eta\in\Xi  .
\eqe  
\end{lem} 
\begin{proof}  
\noindent
\textit{Step 0: proof that $\sum_{\xi \in \Xi} T_\xi= \infty$.}
As explained in Remark~\ref{remark-tilde-process}, for each $n\geq n_z$, the discrete time Markov chain $Y^n$ a.s.\ visits each $x\in \mcl V\mcl G_n$ infinitely often. Therefore, $\sum_{j=0}^\infty T_{[(n,j)]}$ stochastically dominates a countable sum of $\op{exponential}(\frk w(x))$ random variables. Hence a.s.\ $\sum_{j=0}^\infty T_{[(n,j)]} = \infty$. 
It remains to prove that the sum up to a fixed $\eta\in\Xi$ is finite.
\medskip

\noindent
\textit{Step 1: reducing to a particular choice of $\eta$.}
For $\ell \in \BB N_0$, let $\eta_\ell$ be the $(\ell+1)$st smallest $\eta \in \Xi$ for which $Y_{\eta_\ell} =z$. Equivalently, for any $n\geq n_z$, we have that $\eta_\ell$ is the equivalence class of $(n , j_\ell )$, where $j_0 = 0$ and $j_\ell$ is the $\ell$th smallest time $j \geq 1$ for which $Y^n_j  = z$. 
As noted just above, Remark~\ref{remark-tilde-process} implies that each $Y^n$ a.s.\ it returns to its starting point infinitely often. Hence the equivalence classes $\{\eta_\ell\}_{\ell\in\BB N_0}$ are well-defined and each $\eta \in\Xi$ satisfies $\eta < \eta_\ell$ for some $\ell \in \BB N_0$.  

It therefore suffices to find $\frk w_* : \mcl V\mcl G\to (0,\infty)$ such that for each $\frk w$ as in the lemma statement and each $\ell \in \BB N$, the relation~\eqref{eqn-time-finite} holds for $\eta = \eta_\ell$. By the strong Markov property (recall our choice of coupling in Lemma~\ref{lem-rw-coupling}), the random variables  
\eqbn
\sum_{\xi \in\Xi , \eta_{\ell-1} \leq \xi < \eta_\ell} T_\xi  
\eqen
are i.i.d.\ Therefore, it suffices to find $\frk w_* : \mcl V\mcl G\to (0,\infty)$ such that for each $\frk w$ as in the lemma statement
\eqb  \label{eqn-time-finite-show} 
\sum_{\xi \in\Xi ,   \xi < \eta_1} T_\xi   < \infty. 
\eqe 
  
\noindent
\textit{Step 2: choosing the rates.}
For $z\in \mcl V\mcl G$ and $n\geq n_z+1$, consider the random variable
\eqb  \label{eqn-out-time}
S_n := \sum \left\{  T_\xi  : \xi \in \Xi , \, \xi < \eta_1 ,\,   Y_\xi \in    \mcl G_n \setminus \mcl G_{n-1} \right\} .
\eqe 
By~\eqref{eqn-cont-time-def} and~\eqref{eqn-xi-exists}, under $\BB P_z$, the random variable $S_n$ is equal to the amount of time that the continuous-time random walk $X^n$ spends in $ \mcl G_n\setminus \mcl G_{n-1}$ before the first time that it returns to its starting point $z$. 
From this description, it is clear that under $\BB P_z$, the random variable $S_n$ tends to zero in distribution if we fix $z$ and $n$ and choose the rates $\frk w(x)$ in such a way that
\eqbn
\min\left\{\frk w(x) : x\in \mcl G_n\setminus \mcl G_{n-1} \right\} \to \infty .
\eqen

Hence, for each $n\geq n_z+1$, there exists $C_n(z) > 0 $ such that if 
\eqb  \label{eqn-rate-condition0}
\frk w(x) \geq C_n(z) ,\quad \forall x \in   \mcl G_n\setminus \mcl G_{n-1} 
\eqe 
then
\eqb  \label{eqn-time-small} 
\BB P_z\left[ S_n \leq 2^{-n}  \right] \geq 1 - 2^{-n}   .
\eqe  
Let $C_n := \max_{z \in \mcl G_n} C_n(z)$, so that if
\eqb  \label{eqn-rate-condition}
\frk w(x) \geq C_n  ,\quad \forall x \in \mcl G_n\setminus \mcl G_{n-1} 
\eqe 
then~\eqref{eqn-time-small} holds for all $z\in \mcl G_n$. 
 
Henceforth assume that the function $\frk w$ has been chosen so that~\eqref{eqn-rate-condition} holds for all but finitely many values of $n$.  
\medskip

\noindent
\textit{Step 3: proof of~\eqref{eqn-time-finite-show}.}
Fix $z\in \mcl V\mcl G$. 
By~\eqref{eqn-time-small} and the Borel-Cantelli lemma, $\BB P_z$-a.s.\ 
\eqb  \label{eqn-use-bc}
\sum_{n = n_z + 1}^\infty  S_n   < \infty. 
\eqe 
By the definition~\eqref{eqn-out-time} of $S_n$,  
\eqbn
\sum_{\xi \in\Xi , \xi <  \eta_1} T_\xi 
= \sum \left\{ T_\xi : \xi \in \Xi ,\,  \xi < \eta_1 , \,  Y_\xi \in \mcl G_{n_z} \right\} + \sum_{n=n_z+1}^\infty S_n .
\eqen
The first sum has only finitely many terms since a.s.\ $Y^{n_z}$ returns to its starting point in finitely many steps. The second term is finite by~\eqref{eqn-use-bc}.  
\end{proof}

We have the following useful corollary of Lemma~\ref{lem-time-finite}.

\begin{lem} \label{lem-time-tail}
Almost surely, for each $\eta\in\Xi$,  
\eqb  \label{eqn-time-tail}
\lim_{n\to\infty} \sum \left\{ T_\xi :  \xi \in\Xi , \,  \xi\leq \eta , \, Y_\xi \notin \mcl G_n \right\}  = 0  .
\eqe  
\end{lem}
\begin{proof}
For each $\xi \in \Xi$, there exists $n \geq n_z$ such that $Y_\xi \in \mcl G_n$. The lemma therefore follows from Lemma~\ref{lem-time-finite} and the dominated convergence theorem.  
\end{proof}
 
We will now define the continuous time random walk reflected off of $\infty$, which we call $X : [0,\infty) \to \mcl V\mcl G\cup\{\infty\}$, by a formula analogous to~\eqref{eqn-cont-time-def}. We first need to define the endpoints of the intervals on which $X$ will be constant. 

Let $\xi \in \Xi$ and let $m \geq n_z$ be chosen so that $Y_\xi \in \mcl G_m$. By~\eqref{eqn-xi-exists}, there exists $k \in \BB N_0$ such that $(m,k) \in \xi$. Furthermore, by~\eqref{eqn-rw-consistent-def}, $J_{k+1}^{m,n} = J_k^{m,n} +1$ for each $n \geq m$. Therefore, there does not exist any $\eta \in \Xi$ with $\xi < \eta < [(m,k+1)]$. We write 
\eqb \label{eqn-successor}
\wh\xi := [(m,k+1)] 
\eqe 
and call it the \textbf{successor} of $\xi$. 

For $\eta \in \Xi$, we define 
\eqb \label{eqn-xi-interval}
\tau_\eta := \sum_{\xi\in \Xi , \, \xi < \eta} T_\xi .
\eqe 
For $t \in [0,\infty)$, we define
\eqb \label{eqn-cont-time-def-infty}
X_t := Y_\eta ,\quad \forall \eta \in \Xi ,\quad\forall t \in \left[\tau_\eta, \tau_{\wh\eta}  \right) .
\eqe 
If $X_t$ does not belong to any interval of the form $[\tau_\eta , \tau_{\wh\eta})$, we instead define $X_t := \infty$.

We now prove some basic lemmas about $X$ in preparation for showing that it satisfies the properties of Theorem~\ref{thm-cont-time-walk}. The first lemma shows that for each $t \geq 0$, a.s.\ $X_t \not= \infty$.

\begin{lem} \label{lem-pt-defined}
For each fixed $t > 0$, a.s.\ there exists $\xi \in \Xi$ such that $t \in (\tau_\xi , \tau_{\wh\xi})$. 
\end{lem}
\begin{proof}
Fix $\eta \in\Xi$.   
If $\xi\in\Xi$ with $Y_\xi \in \mcl V\mcl G_n$, then by~\eqref{eqn-xi-exists}, we have $\xi = [(n,j)]$ for some $j\in\BB N_0$ and $\wh\xi = [(n,j+1)]$. 
By this and Lemma~\ref{lem-time-tail}, the portion of $[0, \tau_\eta]$ covered by open intervals of the form $(\tau_{[(n,j)]} , \tau_{[(n,j+1)]})$ tends to 1 as $n\to\infty$. In particular, Lebesgue-a.e.\ point of $[0,\tau_\eta]$ is contained in an interval of the form $(\tau_\xi,\tau_{\wh\xi})$. Since $\eta$ is arbitrary, we get that Lebesgue-a.e.\ time $t\in [0,\infty)$ is contained in such an interval.

It remains to transfer from a statement for Lebesgue-a.e.\ time $t$ to a statement for a given deterministic time $t$. 
To this end, let $U$ be sampled uniformly from $[0,1]$, independently from everything else. If $T$ is an exponential random variable independent from $U$, then the joint laws of $T$ and $T+U$ are mutually absolutely continuous. Note that the first element of $\Xi$ (with respect to its total order) is the equivalence class $\xi_0$ which contains $ (n,0) $ for each $n\geq n_z$. From the above absolute continuity statement, we get that the laws of $\{T_\xi\}_{\xi\in\Xi}$ and $\{T_{\xi_0} + U\} \cup \{T_\xi\}_{\xi >  \xi_0}$ are mutually absolutely continuous. By the definition~\eqref{eqn-xi-interval} of $\tau_\xi$, it follows that the laws of $\{\tau_\xi \}_{\xi\in\Xi}$ and $\{\tau_\xi + U\}_{\xi\in\Xi}$ are mutually absolutely continuous. 

Since $U$ is independent from $\{\tau_\xi\}_{\xi\in\Xi}$, 
\allb \label{eqn-uniform-leb}
\BB P\left[ t  \in [0,\infty) \setminus \bigcup_{\xi\in\Xi} (\tau_\xi + U , \tau_{\wh\xi} + U) \right] 
&=  \BB P\left[ U \in [0 , t] \setminus \bigcup_{\xi\in\Xi} (t - \tau_{\wh\xi}  , t -  \tau_\xi) \right]  \notag\\
&=  \BB E\left[ \op{Leb}\left(   [0 , 1 \wedge t] \setminus \bigcup_{\xi\in\Xi} (t - \tau_{\wh\xi}  , t -  \tau_\xi) \right)  \right]     
\alle
where $\op{Leb}$ denotes one-dimensional Lebesgue measure. 
The Lebesgue measure of the set inside the expectation of the last line of~\eqref{eqn-uniform-leb} is zero a.s.\ by the conclusion of the first paragraph. So, we get that a.s.\ there exists $\xi\in\Xi$ such that $t\in (\tau_\xi + U , \tau_{\wh\xi} + U)$. By the absolute continuity in the second paragraph, this implies the lemma statement. 
\end{proof}

We also have the following pointwise convergence statement for the continuous time Markov processes $X^n$ of~\eqref{eqn-cont-time-def} to $X$. 

\begin{lem} \label{lem-ptwise-conv} 
Almost surely, the following is true. Let $\eta\in\Xi$, let $\wh\eta$ be its successor, and let $K$ be a compact subset of $[ \tau_\eta, \tau_{\wh\eta})$. 
For each large enough $n \geq n_z$ (how large is random and depends on $\eta$), we have 
\eqb \label{eqn-ptwise-conv} 
X^n_t = X_t ,\quad\forall t \in K . 
\eqe 
In particular, for each fixed $t\geq 0$, a.s.\ $X^n_t = X_t$ for each sufficiently large $n\geq n_z$. 
\end{lem}
\begin{proof}
Let $n\geq n_z$ be large enough so that $Y_\eta \in \mcl G_n$. Then $\eta = [(n,j)]$ for some $j\in\BB N_0$ and $\wh\eta = [(n,j+1)]$.  
By the definitions~\eqref{eqn-cont-time-def} and~\eqref{eqn-cont-time-def-infty}, we have
\eqb  \label{eqn-ptwise-shift}
X_t = X^n_{t - R_n} ,\quad\forall t \in [\tau_\eta, \tau_{\wh\eta}) 
\eqe 
where 
\eqbn
0 \leq R_n \leq \sum \left\{ T_\xi : \xi\in\Xi , \, \xi\leq \eta, \, Y_\xi \notin \mcl G_n \right\}  
\eqen
(the second inequality is not an equality since $X^n$ spends some time in $B_1\mcl G_n\setminus \mcl G_n$).
By Lemma~\ref{lem-time-tail}, we have $R_n\to 0$ as $n\to\infty$.
Furthermore, by the definition of $X$~\eqref{eqn-cont-time-def-infty}, $X$ is constant on $[\tau_\eta,\tau_{\wh\eta})$ for each $\eta\in\Xi$. Hence~\eqref{eqn-ptwise-conv} for large enough $n$ follows from~\eqref{eqn-ptwise-shift}. 
The statement for a fixed $t\geq 0$ follows from~\eqref{eqn-ptwise-conv} and Lemma~\ref{lem-pt-defined}.  
\end{proof}

\subsection{Proof of \texorpdfstring{Theorem~\ref{thm-cont-time-walk}}{Theorem 1.5}}
\label{sec-cont-time-proof}

Define the process $X : [0,\infty) \to \mcl V\mcl G \cup \{\infty\}$ as in~\eqref{eqn-cont-time-def-infty}. 
We first check that $X$ satisfies the properties in the theorem statement, then we prove that $X$ is the unique process satisfying these properties. We will need to talk about the convergence of the functions $X^n$ of~\eqref{eqn-cont-time-def} to $X$. For technical reasons it is convenient to work with a topology (instead of just considering pointwise convergence as in Lemma~\ref{lem-ptwise-conv}). There are many choices of topology which would work equally well. We define one such choice here.

\begin{defn} \label{def-metric} 
Let $S$ be a countable set with the discrete topology (we will eventually take $S = \mcl V\mcl G \cup \{\infty\}$). 
We define $L^1_{\op{loc}}([0,\infty) ,S)$ to be the metric space of Lebesgue measurable functions $f : [0,\infty) \to S$, with functions identified if they agree Lebesgue a.e., with the metric
\eqb  \label{eqn-metric}
\BB d(f,g) := \int_0^\infty e^{-t} \BB 1_{f(t) \not= g(t)} \,dt .
\eqe 
It is easy to check that with this metric, $L^1_{\op{loc}}([0,\infty) ,S)$ is a complete, separable metric space. 
\end{defn}

\begin{proof}[Proof of Theorem~\ref{thm-cont-time-walk}, existence]
Let $\frk w_*$ be as in Lemma~\ref{lem-time-finite} and assume that $\frk w(x) \geq \frk w_*(x)$ for all but finitely many $x\in\mcl V\mcl G$. 
Define $X = \{X_t\}_{t\geq 0} $ as in~\eqref{eqn-cont-time-def-infty}. We will now check that $X$ satisfies the conditions in the theorem statement. 
\medskip

\noindent\textit{Property~\eqref{item-ae}: Almost everywhere defined.} 
This follows from Lemma~\ref{lem-pt-defined} and the fact that by~\eqref{eqn-cont-time-def-infty}, the process $X$ is constant on each interval of the form $[\tau_\xi , \tau_{\wh\xi})$.  
\medskip

\noindent\textit{Property~\eqref{item-cont}: Right continuity.} 
By the definition~\eqref{eqn-cont-time-def-infty} of $X$, if $t\geq 0$ such that $X_t\in \mcl V\mcl G$, then there exists $\xi\in \Xi$ such that $t\in [\tau_\xi , \tau_{\wh\xi})$ and $X_s = X_t$ for all $s\in [\tau_\xi , \tau_{\wh\xi})$. From this, the right continuity is immediate. 
\medskip

\noindent\textit{Property~\eqref{item-rw}: Continuous time random walk.}
Recall that $(n,0)$ for each $n\geq n_z$ belongs to the same equivalence class in $\Xi$, which is the first element of $\Xi$. Call this equivalence class $\xi_0$. By the definitions~\eqref{eqn-cont-time-def} and~\eqref{eqn-cont-time-def-infty} of $X^n$ and $X$, we have $X^n_t = X_t$ for each $n\geq n_z$ and each $t\in [0,T_{\xi_0}]$. Since each $X^n$ is a continuous time random walk, Property~\eqref{item-rw} is immediate.  
\medskip

\noindent\textit{Property~\eqref{item-markov}: Markov property.} 
Each $X^n$ is a continuous time random walk. So, the Markov property of continuous time random walk implies that for each $x\in\mcl V\mcl G$ and each $n\geq n_z$, on the event $\{X^n_t = x\}$, the $\BB P_z$-conditional law of $\{X^n_{s+t}\}_{s\geq 0}$ given $\{X^n_s\}_{s\leq t}$ is the same as the $\BB P_x$-law of $\{X^n_s\}_{s\geq 0}$. 
By Lemma~\ref{lem-ptwise-conv}, we have $X^n \to X$ in law with respect to the metric~\eqref{eqn-metric}. Furthermore, a.s.\ $X_t^n = X_t$ for each sufficiently large $n\in\BB N$, and for each fixed $x\in\BB N$, the $\BB P_x$-law of $\{X^n_s\}_{s\geq 0}$ converges to the $\BB P_x$-law of $\{X_s\}_{s\geq 0}$. Combining these statements with a standard convergence result for conditional laws (see, e.g.,~\cite[Lemma 4.3]{gp-sle-bubbles}) gives Property~\eqref{item-markov}. 
\medskip

\noindent\textit{Property~\eqref{item-recurrence}: Recurrence.} 
By Remark~\ref{remark-tilde-process}, for each $n\in\BB N_0$ the discrete time Markov process $Y^n$ a.s.\ returns to its starting vertex infinitely many times. 
Consequently, if $Y$ is the discrete time random walk reflected off of $\infty$ as in~\eqref{eqn-process-on-xi}, started from $z\in\mcl V\mcl G$, then a.s.\ there exist infinitely many $\xi\in \Xi$ for which $Y_\xi = z$. 
By the definition~\eqref{eqn-cont-time-def-infty} of $X$, this shows that a.s.\ there exist arbitrarily large times $t$ for which $X_t = z$.  
\medskip
 
\noindent\textit{Property~\eqref{item-harmonic}: Relationship to harmonic functions.} 
If $n\geq n_z$ is sufficiently large so that $\mcl A\subset \mcl V\mcl G_n$, then by the definitions~\eqref{eqn-cont-time-def} and~\eqref{eqn-cont-time-def-infty} of $X$ and $X^n$, the first point $X_\tau$ of $\mcl A$ hit by $X$ is the same as the first point of $\mcl A$ hit by $X^n$. Since $X^n$ is a continuous-time version of the discrete-time random walk $Y^n$, the first point of $\mcl A$ hit by $X^n$ is the same as the first point of $\mcl A$ hit by $Y^n$. 
Property~\eqref{item-harmonic} therefore follows from Lemma~\ref{lem-rw-harmonic}.
\end{proof}

We now turn our attention toward proving the uniqueness part of Theorem~\ref{thm-cont-time-walk}. 
We need the following strong Markov property.
We emphasize that the property holds for any process satisfying the properties in Theorem~\ref{thm-cont-time-walk}, not just the one we constructed in~\eqref{eqn-cont-time-def-infty}. 

\begin{lem} \label{lem-wt-strong-markov}
Let $ X : [0,\infty) \to \mcl V\mcl G \cup\{\infty\}$ be a process satisfying the properties in Theorem~\ref{thm-cont-time-walk} and for $z\in\mcl V\mcl G$, write $\BB P_z$ for the law of $X$ started at $z$. 
Let $\tau$ be a stopping time for $\{ X_t\}_{t\geq 0}$ and let $x\in \mcl V\mcl G$. For $x\in\mcl V\mcl G$, on the event $\{\tau < \infty ,  X_\tau = x\}$, the $\BB P_z$-conditional law of $\{  X_{s+\tau}\}_{s\geq 0}$ given $\{  X_s\}_{s\leq \tau}$ is the same as the $\BB P_x$-law of $\{  X_s\}_{s\geq 0}$. 
\end{lem}
\begin{proof}
If $\tau$ takes values in a deterministic countable subset of $[0,\infty)$, the result follows easily from the Markov property of Property~\eqref{item-markov}. In general, let $\tau_k := 2^{-k} \lceil 2^k \tau \rceil$ be the smallest element of $[\tau,\infty)\cap (2^{-k} \BB Z)$. Then $\tau_k$ decreases to $\tau$ a.s.\ on the event $\{\tau < \infty\}$. By right continuity (Property~(\ref{item-cont})), on the event $\{\tau < \infty ,   X_\tau = x\}$, a.s.\ $ X_{\tau_k} =  X_\tau$ for each sufficiently large $k$. The lemma statement now follows from a standard limiting argument. 
\end{proof}

\begin{proof}[Proof of Theorem~\ref{thm-cont-time-walk}, uniqueness]
Let $X$ be the process we constructed in~\eqref{eqn-cont-time-def-infty}, so that (as we showed just above) $X$ satisfies the properties in the theorem statement. 
Let $\wt X : [0,\infty) \to \mcl V\mcl G \cup\{\infty\}$ be another process satisfying the properties in the theorem statement. We write $\BB P_z$ for the probability measure where $\wt X$ starts at $z\in \mcl V\mcl G$. 
We want to show that $\wt X \eqD X$.  

Assume that $\wt X$ starts at $z\in\mcl V\mcl G$ and let $n_z$ be as in~\eqref{eqn-vertex-in}. 
We will construct a family of processes $\{\wt X^n\}_{n\geq n_z}$ coupled with $\wt X$ such that $\wt X^n$ has the same law as the continuous time Markov process $X^n$ from~\eqref{eqn-cont-time-def} used in the construction of $X$. We will then show that $\wt X^n \to \wt X$ in law (with respect to the metric from~\eqref{eqn-metric}) to conclude.
\medskip

\noindent\textit{Step 1: definition of $\wt X^n$.}
For $n\geq n_z$, we define a sequence of stopping times $\{t_j^n\}_{j \geq 0}$ as follows. Let $t_0^n = 0$. Now suppose $j \geq 0$ and $t_j^n$ has been defined. If $\wt X_{t_j^n} \in \mcl G_n$, let $t_{j+1}^n$ be the first time $t\geq t_j^n$ for which $\wt X_t \not= \wt X_{t_j^n}$. By Property~\eqref{item-rw} and the strong Markov property (Lemma~\ref{lem-wt-strong-markov}), necessarily $\wt X_{t_{j+1}^n} \in B_1\mcl G_n$. 
Otherwise, if $\wt X_{t_j^n} \notin \mcl G_n$, let $t_{j+1}^n$ be the smallest $t\geq t_j^n$ for which $\wt X_{t} \in \mcl G_n$. We also define
\eqb  \label{eqn-wt-holding}
T_j^n :=   \min\left\{ t \geq t_j^n : \wt X_t \not= \wt X_{t_j^n}\right\} - t_j^n  .
\eqe 
Note that $T_j^n = t_{j+1}^n - t_j^n$ if $\wt X_{t_j^n} \in\mcl G_n$, but $T_j^n$ could be strictly smaller than $t_{j+1}^n -t_j^n$ if $\wt X_{t_j^n} \in B_1\mcl G_n\setminus \mcl G_n$. 

By the strong Markov property (Lemma~\ref{lem-wt-strong-markov}), for each $ j\in\BB N_0$ and each $x\in B_1\mcl G_n$, on the event $\{\wt X_{t_j^n} =x\}$, the conditional distribution of $\{\wt X_{s+t_j^n}\}_{s\geq 0}$ given $\{\wt X_s\}_{s\leq t_j^n}$ is the same as the law of $\wt X$ under $\BB P_x$. 

By this and Property~\eqref{item-rw}, on the event $\{\wt X_{t_j^n} =x\}$, the random variables $\wt X_{t_{j+1}^n}$ and $T_j^n$ are conditionally independent given $\{\wt X_s\}_{s\leq t_j^n}$, and the conditional law of $T_j^n$ is exponential with parameter $\frk w(x)$. Furthermore, if $x\in \mcl G_n$, then the conditional law of $\wt X_{t_{j+1}^n}$ is given by a step of the random walk on $\mcl G$ started from $x$. If $x\in B_1\mcl G_n\setminus \mcl G_n$, then by Property~\eqref{item-harmonic} (applied with $\mcl A = \mcl V\mcl G_n$), the conditional law of $\wt X_{t_{j+1}^n}$ is instead given by harmonic measure on $\mcl V\mcl G_n$ viewed from $x$ (Definition~\ref{def-hm}). 

By Property~\eqref{item-recurrence}, a.s.\ $\wt X$ returns to its starting point infinitely many times. This combined with the description of the law of $T_j^n$ in the previous paragraph implies that a.s.\ $\lim_{k \to\infty} \sum_{j=1}^k T_j^n  = \infty$. 
Hence we can define $\wt X^n : [0,\infty) \to B_1\mcl G_n$ by 
\eqb  \label{eqn-wt-cont-time}
\wt X^n_t := \wt X_{t_k^n} ,\quad \forall t \in \left[ \sum_{j=0}^{k-1} T_j^n ,  \sum_{j=0}^k T_j^n  \right) , \quad \forall k \geq 1. 
\eqe 
By the preceding paragraph, $\wt X^n$ is the continuous time Markov chain on $B_1\mcl G_n$ with transition rates $\frk w(x) p_n(x,y)$ for all $x,y\in B_1\mcl G_n$, where $p_n(x,y)$ is as in~\eqref{eqn-transition-walk} and~\eqref{eqn-transition-comp}. In other words, $\wt X^n$ has the same law as the process $X^n$ defined in~\eqref{eqn-cont-time-def}. 
\medskip

\noindent\textit{Step 2: convergence of $\wt X^n$ to $\wt X$.}
By the definition of the times $t_j^n$, each time $t \geq 0$ for which $\wt X_t \in \mcl V\mcl G_n$ is contained in $[t_j^n, t_{j+1}^n)$ for some $j\geq 0$. By this and~\eqref{eqn-wt-cont-time}, if $\wt X_t \in \mcl V\mcl G_n$ then
\eqb  \label{eqn-wt-shift} 
\wt X^n_t = \wt X_{t - R_n} 
\eqe 
where 
\eqb \label{eqn-wt-shift-int} 
0 \leq R_n \leq  \int_0^t \BB 1_{(\wt X_s \notin\mcl G_n)}\,ds  .
\eqe 
By Property~\eqref{item-ae}, we have $\wt X_s \in \mcl V\mcl G $ for Lebesgue-a.e.\ $s \geq 0$. 
The dominated convergence theorem therefore implies that for each $t\geq 0$, the right side of~\eqref{eqn-wt-shift-int} goes to zero as $n\to\infty$. 
Hence~\eqref{eqn-wt-shift} together with Property~\eqref{item-ae} implies that for each fixed $t \geq 0$, a.s.\ $\wt X^n_t =\wt X_t$ for each sufficiently large $n\geq n_z$. In particular, $\wt X^n \to \wt X$ in law with respect to the metric~\eqref{eqn-metric}. Since $\wt X^n \eqD X^n$ and $X^n \to X$ in law with respect to the metric~\eqref{eqn-metric} (see the proof of existence), we get that $\wt X \eqD X$. 
\end{proof}

\subsection{Continuity properties}
\label{sec-cont}

Let $X : [0,\infty) \to \mcl V\mcl G\cup \{\infty\}$ be the continuous time random walk on $\mcl G$ reflected off of $\infty$ as in Theorem~\ref{thm-cont-time-walk}. In this subsection, we prove some intuitively obvious statements about the a.s.\ behavior of $X$.

\begin{lem} \label{lem-to-infty}
Almost surely, the following is true for each finite set of vertices $\mcl A\subset \mcl V\mcl G$. 
\begin{enumerate}
\item \label{item-finite-visit} For each $t > 0$, the set $X^{-1}(\mcl A) \cap [0,t)$ is a finite union of left-closed, right-open intervals on which $X$ is constant.
\item \label{item-to-infty} For each $t >0$ such that $X_t = \infty$, there exists $\ep > 0$ such that $X_s \notin \mcl A$ for each $s\in [t-\ep , t ]$.
\end{enumerate} 
\end{lem}
\begin{proof}
For $x\in\mcl V\mcl G$, let $\tau_0(x)$ be the smallest $s \geq 0$ for which $X_s = x$.
Inductively, for $k\geq 0$, let $\sigma_k(x)$ be the smallest $s\geq \tau_k(x)$ for which $X_s \not= x$ and let $\tau_{k+1}(x)$ be the smallest $s\geq \sigma_k(x)$ for which $X_s = x$.  
That is, $\{[\tau_k(x) , \sigma_k(x))\}_{k\geq 0}$ is the ordered sequence of time intervals during which $X$ is at $x$.
 
By the strong Markov property (Lemma~\ref{lem-wt-strong-markov}) and Property~\eqref{item-rw} of $X$, the random variables $\sigma_k(x) - \tau_k(x)$ are i.i.d.\ exponential random variables of parameter $\frk w(x)$. 
Consequently, a.s.\ $\tau_k(x) \to \infty$ as $k\to\infty$. In particular, for any given $t  > 0$, a.s.\ only finitely many of the intervals $[\tau_k(x) , \sigma_k(x))$ intersect $[0,t)$. By applying this for each $x\in\mcl A$, we get Assertion~\ref{item-finite-visit}. 

Again by Property~\eqref{item-rw}, for each time $\sigma_k(x)$, we have that $X_{\sigma_k(x)}$ is a vertex of $\mcl G$ which is joined to $x$ by an edge and there exists a random $\delta = \delta(x,k) > 0$ such that $X_u = X_{\sigma_k(x)}$ for each $u \in [\sigma_k(x) , \sigma_k(x) + \delta]$. 
By combining this with the previous paragraph and taking the minimum of finitely many values of $\delta$, we get that if $t > 0$ such that $X_t = \infty$, then there exists a random $ \ep_x > 0$ such that $X_s \not= x$ for each $s \in [t-\ep_x ,t]$. This implies Assertion~\ref{item-to-infty} in the lemma statement with $\ep = \min_{x\in \mcl A} \ep_x$. 
\end{proof}
 
The following lemma tells us that the visits of $X$ to $\infty$ cannot cause it to ``jump'' over a set of vertices which disconnects $\mcl V\mcl G$. 
 
\begin{lem} \label{lem-no-jump}
Almost surely, the following is true for every $0 \leq s  <  t$ such that $X_s , X_t \in \mcl V\mcl G$.
Let $V$ be a finite set of vertices of $\mcl G$ such that every finite path in $\mcl G$ from $X_s$ to $X_t$ visits a vertex of $V$. Then there exists $u \in [s,t]$ such that $X_u \in V$.
\end{lem}
\begin{proof}
We first prove that the conclusion of the lemma holds a.s.\ for fixed deterministic times $0\leq s <t$. 
Since $\mcl V\mcl G$ is a countable set, it suffices to prove the following. 
Let $x,y\in\mcl V\mcl G$ and let $V$ be a finite set of vertices of $\mcl G$ such that every finite path in $\mcl G$ from $x$ to $y$ visits a vertex in $V$.
Then on the event $\{X_s =x ,\, X_t =y\}$, a.s.\ there exists $u\in [s,t]$ such that $X_u \in V$. 

We can assume without loss of generality that $x,y\notin V$ (otherwise, the desired statement is vacuous). 
Let $\tau$ be the first time $r \geq s$ for which $X_r \in V \cup \{y\}$. 
By the Markov Property~\eqref{item-markov} together with Property~\eqref{item-harmonic}, on the event $\{X_s = x\}$, the conditional distribution of $X_\tau$ given $X|_{[0,s]}$ is given by harmonic measure on $V\cup \{y\}$ viewed from $x$ (as defined in Definition~\ref{def-hm}). Since every path in $\mcl G$ from $x$ to $y$ hits a vertex of $V$, 
the energy-minimizing function $h^y : \mcl V\mcl G\to \BB R$ with $h^y|_{V\cup \{y\}} = \BB 1_y$ is identically equal to zero on the connected component of $\mcl V\mcl G\setminus (V\cup \{y\})$ which contains $x$. Hence, $\BB P[X_\tau = y] = 0$, and hence $\BB P[X_\tau \in V] = 1$. That is, we can take $u = \tau$. 

To transfer to a statement which holds simultaneously a.s.\ for all times $0 \leq s < t$ such that $X_s,X_t\in \mcl V\mcl G$,  we apply Property~\eqref{item-cont} to get that, a.s., for every such $s,t$ there exists $\ep > 0$ such that $X_r = X_s$ for each $r\in [s,s+\ep)$ and $X_u = X_t$ for each $r\in [t,t+\ep)$. We then apply the statement for deterministic times to rational times which are contained in $[s,s+\ep)$ and $[t,t+\ep)$, respectively. 
\end{proof}

The following lemma tells us that if $X_t \in \mcl V\mcl G$, then we can talk about ``the vertex of $\mcl G$ visited by $X$ immediately before time $t$''. 

\begin{lem}  \label{lem-parent}
Assume that $\mcl G$ is locally finite. 
Almost surely, the following is true. 
Let $t > 0$ such that $X_t \in \mcl V\mcl G$. There exists $y\in\mcl V\mcl G$ and $\ep > 0$ such that $X_s = y$ for each $s\in [t-\ep, t)$. Moreover, we have either $y = X_t$ or $y\sim X_t$.
\end{lem}
\begin{proof}
If there exists $\ep > 0$ such that $X_s = X_t$ for each $s\in [t-\ep,t]$, then we are done. So, we can assume without loss of generality that there exists a sequence of positive times $t_n$ increasing to $t$ such that $X_{t_n} \not= X_t$ for each $n \in \BB N$. We first argue that we can choose the times $t_n$ so that $X_{t_n} \in \mcl V\mcl G$ for each $n$. By Assertion~\ref{item-to-infty} of Lemma~\ref{lem-to-infty}, if $X_{t_n} = \infty$, then there exists $\ep > 0$ such that $X_s \not= X_t$ for each $s\in [t_n-\ep , t_n]$. By the almost-everywhere defined property~\eqref{item-cont}, if we choose a rational $t_n' \in [t_n-\ep , t_n]$, then $X_{t_n'} \in \mcl V\mcl G\setminus \{X_t\}$. By replacing $t_n$ with such a $t_n'$, we can choose our sequence so that $X_{t_n} \not=\infty$. 

Now let $V$ be the set of vertices $y\in\mcl V\mcl G$ with $y\sim x$. 
Since we are assuming that $\mcl G$ is locally finite, $V$ is finite. 
By Lemma~\ref{lem-no-jump} with this choice of $V$, for each $n\in\BB N$ such that $X_{t_n} \notin V$, there exists $u_n \in [t_n , t)$ such that $X_{u_n} \in V$. Hence the set $X^{-1}(V) \cap [0,t)$ accumulates at $t$. By Assertion~\ref{item-finite-visit} of Lemma~\ref{lem-to-infty}, the set $X^{-1}(V) \cap [0,t)$ is a finite union of left-closed, right-open intervals on which $X$ is constant. Since this set accumulates at $t$, one of these intervals contains $[t-\ep,t)$ for some $\ep > 0$.
\end{proof}

\begin{lem}  \label{lem-finite-jump}
Assume that $\mcl G$ is locally finite. 
Almost surely, the following is true. 
Let $b > a > 0$ such that $X_t \not=\infty$ for each $t \in [a,b]$.
Then $X$ jumps to a different vertex only finitely many times during the time interval $[a,b]$.
\end{lem}
\begin{proof}
By Lemma~\ref{lem-parent} combined with the right continuity property~\eqref{item-cont}, for each $t\in [a,b]$ there exists $\ep  >0$ such that $X$ changes values at most once during the time interval $ [t-\ep , t+\ep]$. By compactness, finitely many of these intervals cover $[a,b]$. 
\end{proof}

\subsection{Relationship to reflected Dirichlet forms}\label{sec-reflected-dirichlet}

In this subsection, we explain that the continuous time version of the random walk reflected off of $\infty$ in Theorem~\ref{thm-cont-time-walk} is associated with the maximal Silverstein extension for the random walk killed upon reaching $\infty$, which is its active reflected Dirichlet form. This subsection is included only for context and is not used elsewhere in the paper. The theory of boundary extensions of Markov processes was developed using extensions of Dirichlet forms by Silverstein \cite{silverstein-extension-discrete,silverstein-extension1,silverstein-extension2}, and in the case of an infinite graph, \cite{klss-dirichlet-discrete} considered the set of all Dirichlet forms associated to it in terms of its Royden compactification. For a detailed exposition on Dirichlet forms, reflected Dirichlet forms, and Silverstein extensions, we refer the reader to the textbook \cite{cf-book}. 

A symmetric Markov process on the countable state space $\mcl V \mcl G$ is described through the following three components of a Dirichlet form: a $\sigma$-finite measure $m$ on $\mcl V \mcl G$ called the \textbf{speed measure}, a dense linear subspace $\mcl F$ of $L^2(\mcl V \mcl G,m)$ called the \textbf{Dirichlet space}, and a non-negative definite symmetric bilinear form $\mcl E$ defined on $\mcl F \times \mcl F$ called the \textbf{energy form}, where $\mcl F$ is complete with respect to the norm $\|f\| = [\mcl E(f,f) + (f,f)_{L^2(m)}]^{1/2}$. In a nutshell, the energy form determines the path of the process while on the graph, the speed measure determines the time spent by the process at each step, and the Dirichlet space determines the behavior of the process at $\infty$. 

If $p_t: \mcl V\mcl G \times \mcl V \mcl G \to [0,1]$ is the transition kernel of a process on $\mcl G$ and $T_t$ is the corresponding transition semigroup on $L^2(\mcl V \mcl G,m)$ given by $T_tf(x) = \sum_{y\in \mcl V \mcl G} p_t(x,y)f(y)$, then the energy form $\mcl E$ is given by the increasing limit
\[ \mcl E(f,f) = \lim_{t\downarrow 0} \frac{1}{t}(f-T_tf,f)_{L^2(m)}. \] On the other hand, the Dirichlet form $(\mcl E,\mcl F)$ on $L^2(\mcl V\mcl G , m)$ gives the generator $L:D(L)\subset \mcl F \to L^2(\mcl V\mcl G,m)$ of the transition semigroup $T_t$ as
\[ (Lf,g)_{L^2(m)} = -\mcl E(f,g).  \]
Hence, for the continuous time random walk on $\mcl G$ with a bounded rate function $\frk w$, the corresponding Dirichlet form is 
\[ \begin{cases} \mcl E(f,f) = \frac{1}{2}\sum_{x\in \mcl V \mcl G} \sum_{y \sim x} \frk c(x,y)(f(x)-f(y))^2 = \textrm{Energy}(f) , \\
\mcl F = L^2(\mcl V\mcl G,m), \\
m = \sum_{x\in \mcl V \mcl G} \frac{\pi(x)}{\frk w(x)}\delta_x.
\end{cases}\]
See, e.g., \cite[Theorem~2.2.2]{cf-book}. In particular, $\mcl E$ is the $\ell^2$ inner product of the discrete gradient as defined in Section \ref{sec-min-harmonic}, and the rate function $\frk w$ does not appear in the energy form.

In Theorem \ref{thm-cont-time-walk}, we consider unbounded rate functions $\frk w$ so that the associated random walks on $\mcl G$ wander off to $\infty$ in finite time. Even with unbounded $\frk w$, the formulas for $\mcl E$ and $m$ stay the same; this is since, up to the time the process reaches $\infty$, it is merely a time change of the random walk with unit jump rate at each vertex. If the random walk is killed once it reaches $\infty$, the corresponding Dirichlet space is $\mcl F = \mcl F_e \cap L^2(\mcl V \mcl G,m)$
\cite[Equation~2.2.7]{cf-book} where 
\begin{equation}\label{extended-dirichlet-space} \mcl F_e = \{ f :\mcl V \mcl G \to \mathbb R | \mathrm{Energy}(f) < \infty, \mcl E(f,h) = 0 \text{ for any finite-energy harmonic function $h$ on } \mcl V \mcl G\} \end{equation}
is the \textbf{extended Dirichlet space} of $(\mcl E,L^2(\mcl V \mcl G,\pi))$. The extended Dirichlet space of a Dirichlet form $(\mcl E,\mcl F)$ is defined, in general, as the space of functions which are pointwise limits of $\mcl E$-Cauchy sequences in $\mcl F$. Since $\mcl E(f,f) \leq 2\|f\|_{L^2(\pi)}^2$, by a diagonal argument, the approximating functions here can be assumed to be finitely supported. Thus, \eqref{extended-dirichlet-space} is equivalent to the standard definition of $\mcl F_e$ by Royden's decomposition theorem on graphs (see, e.g., \cite[Theorem~3.69]{soardi-network-book}). 

Extending the Dirichlet space $\mcl F = \mcl F_e \cap L^2(\mcl V\mcl G,m)$ gives various ways to continue the random walk after the time it reaches $\infty$. More precisely, if we denote the subspace of bounded elements of $\mcl F$ as $\mcl F_b$, a \textbf{Silverstein extension} of $(\mcl E,\mcl F)$ is a Dirichlet form $(\tilde{\mcl E},\tilde{\mcl F})$ on $L^2(\mcl V\mcl G,m)$ with $\mcl F\subset \tilde{\mcl F}$ such that $\mcl F_b$ is an ideal of the function ring $\tilde {\mcl F}_b$ and $\tilde{\mcl E}|_{\tilde{\mcl F} \times \tilde{\mcl F}}=\mcl E$. Using Gelfand duality, Silverstein showed that this condition is equivalent to the existence of a locally compact Polish space $\ol{\mcl V \mcl G}$ containing $\mcl V \mcl G$ as a dense open subset such that, if $\tilde X$ is the process on $\ol{\mcl V \mcl G}$ corresponding to $(\mcl E,\mcl F)$ on $L^2(\ol{\mcl V \mcl G},m)$ under the natural embedding of $\mcl V\mcl G$ in $\ol{\mcl V\mcl G}$, then the process obtained by killing $\tilde X$ upon exiting $\mcl V\mcl G$ corresponds to $(\mcl E,\mcl F)$ \cite[Theorem~20.1]{silverstein-extension1}. The extended state space $\ol{\mcl V \mcl G}$ is unique up to a quasi-homeomorphism; see \cite[Remark~5.1]{klssw-boundary-extension} for a precise formulation. 

Based on the propertes listed in Theorem \ref{thm-cont-time-walk}, it can be checked that the Markov transition kernel $p_t$ on $\mcl V \mcl G$ corresponding to the random walk reflected off of $\infty$ with rate function $\mathfrak w$ induces a strongly continuous contraction semigroup on $L^2(\mcl V \mcl G,m)$, and the corresponding Dirichlet form $(\tilde{\mcl E}, \tilde{\mcl F})$ is a Silverstein extension of $(\mcl E, \mcl F_e \cap L^2(\mcl V \mcl G, m))$ on $L^2(\mcl V \mcl G,m)$. In fact, it is the \textbf{active reflected Dirichlet form}, described in the following proposition.

\begin{prop}
Let $(\tilde{\mcl E}, \tilde{\mcl F})$ be the Dirichlet form on $L^2(\mcl V\mcl G,m)$ corresponding to the random walk reflected off of $\infty$ on $\mcl G$ with rate function $\mathfrak w$, where $m = \pi/\mathfrak w$. Then, $\tilde{\mcl F} = \mcl F^{\mathrm{ref}}\cap L^2(\mcl V\mcl G,m)$ where \[ \mcl F^{\mathrm{ref}} = \{f : \mcl V \mcl G \to \mathbb R| \mathrm{Energy}(f)<\infty\} = \mcl F_e \oplus \{ \text{finite-energy discrete harmonic functions on } \mcl G  \} \]
is known as the \textbf{reflected Dirichlet space} \cite[Equation~6.5.5]{cf-book}, and $\tilde{\mcl E}(f,f) = \mathrm{Energy}(f)$ for $f\in \tilde{\mcl F}$. 
\end{prop}

\begin{proof}
Let us first show that $\mcl F^{\mathrm{ref}} \cap L^2(\mcl V \mcl G,m) \subseteq \tilde{\mcl F}$. Let $X$ be process defined in Theorem \ref{thm-cont-time-walk}. Dirichlet form theory gives that if $\mcl A \subset \mcl V \mcl G$ is a non-empty finite set and $\tau_{\mcl A}$ is the hitting time of $\mcl A$, then, for any real-valued function $\phi$ on $\mcl A$, the function $x\mapsto h^\phi(x)= \mathbb E_x [ \phi(\tilde X_{\tau_{\mcl A}});\tau_{\mcl A}<\infty]$ on $\mcl V\mcl G$ is in $\tilde{\mcl F}_e$, the extended Dirichlet space of $(\tilde{\mcl E}, \tilde{\mcl F})$ \cite[Theorem~3.4.2]{cf-book}. (Moreover, $\tilde{\mcl E}(f,u)=0$ for all $u \in \tilde{\mcl F}_e$ with $u|_{\mcl A}=0$.) From the transition law of $X$ and the fact that $m$ is supported on $\mcl V \mcl G$, the Dirichlet energy of $h^\phi$ must equal $\tilde{\mcl E}(h^\phi,h^\phi)$. Moreover, Theorem \ref{thm-cont-time-walk} states that $h^\phi$ is the energy-minimizing extension of $\phi$. Thus, the collection $\{ h^\phi| \mcl A \subset \mcl \mcl V \mcl G \text{ finite}, h:\mcl A \to \mathbb R \}\subset \tilde{\mcl F}_e$ is a dense subset of $\mcl F^{\mathrm{ref}}$ simultaneously with respect to Dirichlet energy and pointwise convergence: given a sequence of finite sets $\mcl A_n \uparrow \mcl V \mcl G$, we have $h^{f|_{\mcl A_n}} \to f$ in Dirichlet energy and pointwise for any finite energy function $f$ on $\mcl V \mcl G$. This means $\mcl F^{\mathrm{ref}} \subseteq \tilde{\mcl F}_e$ and $\tilde{\mcl E}(f,f) = \mathrm{Energy}(f)$ for all $f\in {\mcl F}^{\mathrm{ref}}$. The claim follows since $\tilde{\mcl F} = \tilde{\mcl F}_e \cap L^2(\mcl V\mcl G,m)$ \cite[Theorem~1.1.5]{cf-book}.

For the converse direction, it suffices to note that the active reflected Dirichlet form is the maximal Silverstein extension of $(\mcl E, \mcl F_e \cap L^2(\mcl V\mcl G,m))$ on $L^2(\mcl V\mcl G,m)$, with respect to the following partial order: $(\mcl E_1, \mcl F_1) \preceq (\mcl E_2, \mcl F_2)$ if $\mcl F_1 \subseteq \mcl F_2$ and $\mcl E_1(u,u) \geq \mcl E_2(u,u)$ for all $u\in \mcl F_1$. This maximality was proved by Silverstein \cite[Theorem~15.2]{silverstein-extension1}; see also \cite[Theorem~6.6.9]{cf-book}.\footnote{Here we use that the energy form has no killing part on $\mcl G$, a hypothesis whose necessity is clarified in \cite{klss-dirichlet-discrete,schmidt-reflected-form}.}
\end{proof}

Returning to the case of bounded $\frk w$, we have $\mcl F^{\mathrm{ref}} \cap L^2(\mcl V\mcl G,m) = \mcl F_e \cap L^2(\mcl V\mcl G,m) = L^2(\mcl V\mcl G,m)$ since $\mcl E(f,f) \leq 2(\sup \frk w) \sum_{x\in \mcl V\mcl G} \frac{\pi(x)}{\frk w(x)}(f(x))^2$. This depicts the fact that the random walk on $\mcl G$ with rate function $\frk w$ almost surely does not reach $\infty$ within finite time, so there is no reflection to speak of on the time interval $[0,\infty)$.

\section{Applications to the free uniform spanning forest}
\label{sec-fsf}

In Section~\ref{sec-ab}, we prove Theorem~\ref{thm-fsf} (Aldous--Broder algorithm for the $\frk c$-FSF). 
In Section~\ref{sec-wilson}, we prove Theorem~\ref{thm-wilson} (Wilson's algorithm for the $\frk c$-FSF). 
In keeping with Section~\ref{sec-fsf0}, throughout this section, we assume that $\mcl G$ is locally finite.

\subsection{Aldous--Broder algorithm} 
\label{sec-ab}
 
The goal of this section is to prove Theorem~\ref{thm-fsf}. After the proof, we will also record a description of the branches of the $\frk c$-FSF in terms of loop erased random walk reflected off of $\infty$ (Lemma~\ref{lem-fsf-lerw}). 

Let $\{\mcl G_n\}_{n\geq 1}$ be an increasing family of finite connected subgraphs of $\mcl G$ whose union is all of $\mcl G$.
The proof of Theorem~\ref{thm-fsf} is based on a local convergence statement for the ordered sequence of vertices hit by random walk on $\mcl G_n$ to the ordered sequence of vertices hit by $X$. To state this precisely, we introduce some notation. 

For $n\in\BB N$, we write $\{Z^n_j\}_{j\geq 0}$ for the random walk on $\mcl G_n$. 
Fix $k \leq n$ such that the starting point of $Z^n$ is in $\mcl V\mcl G_k$. Let $L_0^n = 0$ and let $L_i^n$ for $i\geq 1$ be the $i$th smallest $j \in \BB N$ for which $Z^n_j \in \mcl V\mcl G_k$. 
We define the \textbf{ordered sequence of vertices of $\mcl G_k$ hit by $Z^n$} to be the following infinite sequence of vertices of $\mcl G_k$: 
\eqb  \label{eqn-ordered-seq-n}
(Z^n_{I_0^n} , Z^n_{I_1^n} , Z^n_{I_2^n} ,  \dots ) 
\eqe 
 
Similarly to the above, if $k \in \BB N$ such that the starting point of $X$ is in $\mcl V\mcl G_k$, we we let $\sigma_0 = 0$ and for $i\geq 0$ we inductively let $\sigma_i$ be the smallest time $t\geq \sigma_{i-1}$ at which $X$ jumps to a vertex of $\mcl V\mcl G_k$. We define the \textbf{ordered sequence of vertices of $\mcl G_k$ hit by $X$} to be the following infinite sequence of vertices of $\mcl G_k$: 
\eqb  \label{eqn-ordered-seq}
(X_{\sigma_0} , X_{\sigma_1} , X_{\sigma_2} , \dots   ) 
\eqe 

\begin{lem} \label{lem-walk-tv}
Assume that we are in the setting described just above. 
Fix $k\in\BB N$ and $z\in \mcl V\mcl G_k$, and assume that $X$ and each $Z^n$ start at $z$.
For each fixed $N\in\BB N$, the total variation distance between the following two random $N$-tuples goes to zero as $n\to\infty$: the first $N$ elements of the ordered sequence of vertices of $\mcl G_k$ hit by $Z^n$; and the first $N$ elements of the ordered sequence of vertices of $\mcl G_k$ hit by $X$.
\end{lem}
\begin{proof} 
For $x \in \mcl V\mcl G$, we write $  \op{hm}_{\mcl G_k}^x$ for the harmonic measure on $\mcl V\mcl G_k$ in $\mcl G$ as viewed from $x$, as in Definition~\ref{def-hm}. For $x\in\mcl V\mcl G_n$, we also write $ \op{hm}_{\mcl G_k}^{x,n}$ for harmonic measure on $\mcl V\mcl G_k$ in $\mcl G_n$ as viewed from $x$, i.e., $\op{hm}_{\mcl G_k}^{x,n}(y)$ is the probability that random walk on $\mcl G_n$ started at $x$ first hits $\mcl V\mcl G_k$ at $y$. The lemma is a straightforward consequence of the convergence $\op{hm}_{\mcl G_k}^{x,n}(y) \to \op{hm}_{\mcl G_k}^x(y)$ for fixed $x \in \mcl V\mcl G $ and $y\in \mcl V\mcl G_k$. Let us now give the details. 

Let $T_j = 0$ and for $j\geq 1$, define $T_j$ inductively as follows. If $X_{T_{j-1}} \in \mcl V\mcl G_k  $, let $T_j$ be the smallest $t\geq T_{j-1}$ for which $X_t \not= X_{T_{j-1}}$. Otherwise, let $T_j$ be the smallest $t\geq T_{j-1}$ for which $X_t\in \mcl V\mcl G_k$. By the strong Markov property (Lemma~\ref{lem-wt-strong-markov}) and Property~\eqref{item-harmonic} from Theorem~\ref{thm-cont-time-walk}, $\{X_{T_j}\}_{j\geq 0}$ is a discrete time Markov process on the 1-neighborhood $B_1\mcl G_k$ with transition probabilities 
\eqb  \label{eqn-tv-transition}
p_k(x,y) = 
\begin{cases} 
\frk c(x,y) / \pi(x) ,\quad &\text{if $x \in \mcl V\mcl G_k  $ and $y\sim x$} \\
\op{hm}_{\mcl G_k}^x(y) , \quad &\text{if $x\in B_1\mcl G_k \setminus \mcl G_k$}   \\
0 , \quad &\text{otherwise}   .
\end{cases}
\eqe  
That is, $\{X_{T_j}\}_{j \geq 0}$ has the same law as the Markov chain $Y^k$ defined in~\eqref{eqn-transition-walk} and~\eqref{eqn-transition-comp}. 

Similarly, for $n\in\BB N$, let $I_0^n = 0$ and for $j \geq 1$, inductively define $I_j^n$ as follows. If $Z^n_{I^n_{j-1}} \in \mcl V\mcl G_k$, let $I_j^n = I_{j-1}^n + 1$. Otherwise, let $I^n_j$ be the smallest $i \geq I^n_{j-1}$ for which $Z^n_i \in \mcl V\mcl G_k$. By the strong Markov property, $\{Z^n_{I_j^n} \}_{j\geq 0}$ is a discrete time Markov process on $B_1\mcl G_k$ with transition probabilities 
\eqb  \label{eqn-tv-transition-n}
p_k^n(x,y) = 
\begin{cases} 
\frk c(x,y) / \pi(x) ,\quad &\text{if $x \in \mcl V\mcl G_k \setminus \mcl A$ and $y\sim x$} \\
\op{hm}_{\mcl G_k}^{x,n}(y) , \quad &\text{if $x\in B_1\mcl G_k \setminus \mcl G_k$} \\
0 , \quad &\text{otherwise}   .
\end{cases}
\eqe  

By Proposition~\ref{prop-harmonic-conv} (applied with $\mcl A = \mcl V\mcl G_k$), for each $x,y\in B_1\mcl G_k$, the transition probability~\eqref{eqn-tv-transition-n} converges to the transition probability~\eqref{eqn-tv-transition} as $n\to\infty$. Since $B_1\mcl G_k$ is a finite set (recall that we are assuming that $\mcl G$ is locally finite throughout this section), the convergence is in fact uniform over all $x,y\in B_1\mcl G_k$. Since both processes are Markovian and start at $z$, for each fixed $M \in\BB N$ and each $\ep > 0$, it holds for sufficiently large $n\in\BB N$ that can couple together $\{Z^n_{I_j^n}\}_{j\geq 0}$ and $\{X_{T_j}\}_{j\geq 0}$ in such a way that
\eqb  \label{eqn-tv-coupling}
\BB P\left[ Z^n_{I_j^n} = X_{T_j} , \,: \forall j = 1,\dots, M\right] \geq 1 - \ep .
\eqe 
Since $\{X_{T_j}\}_{j\geq 0}$ is an irreducible Markov chain on the finite state space $B_1\mcl G_k$, for any given $N\in\BB N$, if we choose $M$ sufficiently large (depending on $N$), then with probability at least $1-\ep$, the process $\{X_{T_j}\}_{j\geq 0}$ spends at least $N$ integer times in $\mcl V\mcl G_k$ before time $M$. Hence, \eqref{eqn-tv-coupling} implies that with probability at least $1-2\ep$, the first $N$ elements of the ordered sequence of vertices of $\mcl G_k$ hit by $\{Z^n_{I_j^n}\}_{j\geq 0}$ are identical to the first $N$ elements of the ordered sequence of vertices of $\mcl G_k$ hit by $\{X_{T_j}\}_{j\geq 0}$.  

By the definition of $\{I_j^n\}_{j\geq 0}$, the first $N$ elements of the ordered sequence of vertices of $\mcl G_k$ hit by $Z^n$ are the same as the first $N$ elements of the ordered sequence of vertices of $\mcl G_k$ hit by $\{Z^n_{I_j^n}\}_{j\geq 0}$. The analogous statement holds for $X$. Since $\ep > 0$ is arbitrary, this concludes the proof. 
\end{proof}

\begin{proof}[Proof of Theorem~\ref{thm-fsf}]
Let $\{\mcl G_n\}_{n\geq 1}$ be an increasing family of finite connected subgraphs of $\mcl G$ whose union is all of $\mcl G$. 
We assume that $z\in \mcl V\mcl G_n$ for every $n\geq 1$. 
 
For $n\geq 1$, let $\mcl T_n$ be the $\frk c$-weighted spanning tree of $\mcl G_n$. 
By the Aldous--Broder algorithm for finite graphs (see, e.g.,~\cite[Corollary 4.9]{lyons-peres}) we can generate $\mcl T_n$ from random walk $Z^n$ on $\mcl G_n$ started from $z$ as follows. 
For $x\in \mcl V\mcl G_n$, let $J_x^n := \min\{j \geq 0 : Z^n_j = x \}$. Then set $\mcl V\mcl T_n = \mcl V\mcl G_n$ and
\eqb  \label{eqn-sf-def-n}
\mcl E\mcl T_n  = \left\{ \{ Z^n_{J_x^n - 1} , x \} : x \in \mcl V\mcl G_n \setminus \{z\} \right\}  .
\eqe 

Let $k \in \BB N$ and let $K \geq k+1$ be such that the 1-neighborhood $B_1\mcl G_k$ is contained in $\mcl G_K$. 
For $n\geq K$, let $\mcl N_K^n$ be the amount of time that $Z^n$ spends in $\mcl V\mcl G_K$ before the first time at which it has hit every vertex of $\mcl G_K$. Similarly, let $\mcl N_K$ be equal to 1 plus the number of times that $X$ jumps to a vertex of $\mcl G_K$ before the first time at which it has visited every vertex of $\mcl G_K$.  

From the definition~\eqref{eqn-sf-def-n}, it is immediate that $\mcl T_n \cap \mcl E\mcl G_k$ is given by a function of the first $\mcl N_K^n$ elements of the ordered sequence of vertices of $\mcl G_K$ hit by $Z^n$ (recall~\eqref{eqn-ordered-seq-n}). Moreover, by~\eqref{eqn-sf-def}, we have that $\mcl T^{\op{AB}} \cap \mcl E\mcl G_k$ is given by the \emph{same} function of the first $\mcl N_K$ elements of the ordered sequence of vertices of $\mcl G_K$ hit by $X$ (recall~\eqref{eqn-ordered-seq}).

By considering the irreducible Markov chain $\{X_{T_j}\}_{j\geq 0}$ on $B_1\mcl G_K$ as defined in~\eqref{eqn-tv-transition} (with $K$ in place of $k$), for every $\ep > 0$, we can find $N\geq 0$ (depending on $K$) such that with probability at least $1-\ep$, we have $\mcl N_K \leq N$. By combining this with the previous paragraph and Lemma~\ref{lem-walk-tv}, we get that for each sufficiently large $n \geq K$, the total variation distance between the laws of $\mcl T_n\cap \mcl E\mcl G_k$ and $\mcl T^{\op{AB}} \cap \mcl E\mcl G_k$ is at most $1-2\ep$. Since $\ep >0$ and $k\in\BB N$ are arbitrary, by the definition of the $\frk c$-FSF on $\mcl G$, this shows that $\mcl T$ is equal to the $\frk c$-FSF. 
\end{proof}

Using Theorem~\ref{thm-fsf}, we can describe the branches of the $\frk c$-FSF in terms of loop-erased random walk. 

Suppose that the continuous time random walk $X$ reflected off of $\infty$ starts at $X_0 = z \in \mcl V\mcl G$. 
For $y \in \mcl V\mcl G$, let $\tau_y$ be the first time that $X$ hits $y$ (as in Definition~\ref{def-rw-forest}) and let $\op{LE}^y : \BB N_0 \to \mcl V\mcl G$ be the loop erasure of the time reversal of $X^y|_{[0,\tau_y]}$, as in Definition~\ref{def-le-infty}. That is, in the notation of Definition~\ref{def-rw-forest}, we have $\op{LE}^y(0) = y$ and
\eqb \label{eqn-le-def}
\op{LE}^y(k) =
\begin{cases}
a(\op{LE}^y(k-1)) ,\quad &\text{if $\op{LE}^y(k-1) \not= z$} \\
z ,\quad &\text{if $\op{LE}^y(k-1) = z$} . 
\end{cases} 
\eqe
From this description and the definition of $\mcl T^{\op{AB}}$ in Definition~\ref{def-rw-forest}, the following lemma is immediate. 

\begin{lem} \label{lem-fsf-lerw}
Suppose that $X$ starts at $z\in\mcl V\mcl G$. Let $\mcl T^{\op{AB}}$ be as in Definition~\ref{def-rw-forest} (see also Theorem~\ref{thm-fsf}). For $y\in\mcl V\mcl G$, we have that $z$ and $y$ are in the same connected component of $\mcl T^{\op{AB}}$ if and only if the loop erasure $\op{LE}^y$ reaches $y$ in finitely many steps. If this is the case, then the unique simple path from $z$ to $y$ in $\mcl T^{\op{AB}}$ is equal to $\op{LE}^y$, stopped when it first reaches $y$. 
\end{lem}

\subsection{Wilson's algorithm}
\label{sec-wilson}

To prove Theorem~\ref{thm-wilson}, we will, roughly speaking, perform Wilson's algorithm on finite subgraphs of $\mcl G$ and take a limit. We need the following lemma, which is a straightforward consequence of Lemma~\ref{lem-walk-tv}. 
 
\begin{lem} \label{lem-le-tv}
Let $\{\mcl G_n\}_{n\geq 1}$ be an increasing family of connected subgraphs of $\mcl G$ whose union is all of $\mcl G$.
Also let $A\subset \mcl V\mcl G$ be a finite set of vertices and let $z \in \mcl V\mcl G$.
\begin{itemize}
\item Let $X$ be the random walk on $\mcl G$ reflected off of $\infty$ started at $z$ and let $\tau$ be the smallest $t$ for which $X_t \in A$. Let $\op{LE}$ be the loop erasure of $X|_{[0,\tau]}$ (Definition~\ref{def-le-infty}). 
\item For $n\in\BB N$ such that $A \cup \{z\} \subset\mcl G_n$, let $Z^n$ be the random walk on $\mcl G_n$ started at $z$ and let $J_n$ be the smallest $j$ such that $Z^n_j \in A$. Let $\op{LE}^n$ be the loop-erasure of $Z^n|_{[0,J^n]}$ (Definition~\ref{def-le}).
\end{itemize}
Fix $k\in \BB N$ such that $A \cup \{z\} \subset\mcl G_k$.
As $n\to\infty$, the total variation distance between the following two random variables goes to zero. The random variable which equals $\op{LE}^n$ on the event $\op{LE}^n \subset \mcl G_k$ and equals $\emptyset$ otherwise; and the random variable which equals $\op{LE}$ on the event that $\op{LE} \subset \mcl G_k$ and equals $\emptyset$ otherwise. 
\end{lem}
\begin{proof} 
Fix $\ep > 0$.  
Since a.s.\ $X$ reaches $A$ in finite time (by recurrence) and $X$ visits each vertex only finitely many times during each finite time interval (Lemma~\ref{lem-to-infty}), we can find $N = N(k) \in\BB N$ such that 
\eqbn
\BB P\left[ \{\op{LE} \subset \mcl G_k\} \cap F^c \right] < \ep 
\eqen 
where $F$ is the event that $X$ jumps between vertices of $\mcl B_1\mcl G_k$ at most $N$ times before reaching $A$. 

The event $F$ is determined by the first $N$ elements of the ordered sequence of vertices of $\mcl G_k$ hit by $X$, as defined in~\eqref{eqn-ordered-seq}. Moreover, by the definition of the loop erasure, on $F$ the random variable which equals $\op{LE}$ on the event that $\op{LE} \subset \mcl G_k$ and equals $\emptyset$ otherwise is also determined by these first $N$ elements.  Since $\ep > 0$ is arbitrary, the lemma statement therefore follows from Lemma~\ref{lem-walk-tv}. 
\end{proof}

\begin{proof}[Proof of Theorem~\ref{thm-wilson}] 
Recall the notation used to define Wilson's algorithm in Section~\ref{sec-wilson0}. 
Let $\{\mcl G_n\}_{n\geq 1}$ be an increasing family of connected subgraphs of $\mcl G$ whose union is all of $\mcl G$.
\medskip

\noindent\textit{Step 1: connectivity events.}
Fix $n\geq 1$.  
Let $J = J_n$ be the largest index $j$ such that $x_j \in \mcl G_n$.   
Let
\eqb  \label{eqn-wilson-event}
E_n := \left\{ \text{$\mcl T^{\op{FSF}}$ contains a path between any two points of $\{x_0,\dots,x_J\}$} \right\} .
\eqe   
Then $\bigcap_{n = 1}^\infty E_n$ is the same as the event that $\mcl T^{\op{FSF}}$ is connected. 
 
For $N \geq n$ such that $\{x_0,\dots,x_J\} \subset \mcl G_N$, we also define 
\eqb  \label{eqn-wilson-event-N} 
F_n^N := E_n \cap \left\{ \text{$\forall i,j \in \{0,\dots,J\}  $, the simple path in $\mcl T^{\op{FSF}}$ from $x_i$ to $x_j$ is contained in $\mcl G_N$} \right\} . 
\eqe 
Then $E_n$ is the increasing union of the events $F_n^N$ over all $N \geq n$. 
\medskip

\noindent\textit{Step 2: Wilson's algorithm on finite subgraphs.}
For $m\in\BB N$, let $\mcl T_m^{\op{FSF}}$ be the $\frk c$-spanning tree of $\mcl G_m$ (Definition~\ref{def-ust}). 
For $m\geq n$, we perform $J$ steps of Wilson's algorithm on $\mcl G_m$, with loop-erased random walks started from the vertices $x_1,\dots,x_J$. For $j=0,\dots,J$, let $\mcl T_{m,j}^{\op{Wil}}$ be the subtree of $\mcl G_m$ obtained after $j$ steps of Wilson's algorithm. That is, $\mcl T_{m,j}^{\op{Wil}}$ is obtained in the same manner as $\mcl T_j^{\op{Wil}}$ but with random walk on $\mcl G_m$ in place of random walk reflected off of $\infty$ on $\mcl G$. By Wilson's algorithm for $\mcl G_m$, we can couple the trees $\mcl T_{m,j}^{\op{Wil}}$ with $\mcl T_m^{\op{FSF}}$ so that $\mcl T_{m,j}^{\op{Wil}} \subset \mcl T_m^{\op{FSF}}$ for each $j=0,\dots,J$. 

By $J$ applications of Lemma~\ref{lem-le-tv}, we get that for each fixed $N\geq n$, the total variation distance between the following two random variables tends to 0 as $m \to\infty$. 
\begin{enumerate}[($A$)] 
\item \label{eqn-wilson-rv} The random variable with equals $\mcl T_{m,J}^{\op{Wil}}$ if $ \mcl T_{m,J}^{\op{Wil}} \subset \mcl G_N$ and $\emptyset$ otherwise; and the random variable which equals $\mcl T_J^{\op{Wil}}$ if $ \mcl T_{ J}^{\op{Wil}} \subset \mcl G_N$ and $\emptyset$ otherwise. 
\end{enumerate}
\medskip

\noindent\textit{Step 3: Total variation convergence.}
Since $\mcl V\mcl G_n \subset \{x_0 ,\dots, x_J\}$, we have 
\eqb  \label{eqn-wilson-agree}
\mcl T_m^{\op{FSF}} \cap \mcl G_n = \mcl T_{m,J}^{\op{Wil}} \cap \mcl G_n  \quad \text{and} \quad
\mcl T^{\op{Wil}} \cap \mcl G_n = \mcl T_J^{\op{Wil}} \cap \mcl G_n   .
\eqe   
Furthermore, by construction, the tree $\mcl T_{m,J}^{\op{Wil}}$ is the union of the simple paths in $\mcl T_m^{\op{FSF}}$ between pairs of vertices in $\{x_0,\dots,x_J\}$. That is, the event $\{\mcl T_{m,J}^{\op{Wil}} \subset \mcl G_N\}$ is defined in the same way as the event $F_n^N$ of~\eqref{eqn-wilson-event-N}, but with $\mcl T_m^{\op{FSF}}$ instead of $\mcl T^{\op{FSF}}$. 

We have $\mcl T_m^{\op{FSF}} \cap \mcl G_n \to \mcl T^{\op{FSF}} \cap \mcl G_n$ in the total variation sense as $m\to\infty$ (Definition~\ref{def-fsf}) and the total variation distance between the random variables~\eqref{eqn-wilson-rv} goes to zero as $m\to\infty$. By combining this with the preceding paragraph, we get that the following two random variables agree in law. 
The random variable with equals $\mcl T^{\op{FSF}} \cap \mcl G_n$ if $F_n^N$ occurs and $\emptyset$ otherwise; and the random variable which equals $\mcl T^{\op{Wil}} \cap \mcl G_n = \mcl T_J^{\op{Wil}} \cap \mcl G_n$ if $ \mcl T_{ J}^{\op{Wil}} \subset \mcl G_N$ and $\emptyset$ otherwise. 

By sending $N\to\infty$, we get that the following two random variables agree in law. 
The random variable with equals $\mcl T^{\op{FSF}} \cap \mcl G_n$ if the event $E_n$ of~\eqref{eqn-wilson-event} occurs and $\emptyset$ otherwise; and the random variable which equals $\mcl T^{\op{Wil}} \cap \mcl G_n$ if $ \mcl T_{ J}^{\op{Wil}}$ is finite and $\emptyset$ otherwise. 
Sending $n\to\infty$ and recalling Lemma~\ref{lem-wilson-connected} now concludes the proof. 
\end{proof}

\section{Additional results}
\label{sec-additional}

\subsection{Ends}
\label{sec-ends}

Assume that we are in the setting of Theorem~\ref{thm-cont-time-walk}. In this section, we show that if $\mcl G$ is locally finite, then we can canonically associate each $t \geq 0$ such that $X_t = \infty$ with a unique end of $\mcl G$. Hence, we can view the random walk on $\mcl G$ reflected off of $\infty$ as a function from $[0,\infty)$ to $\mcl V\mcl G \cup \{\text{ends of $\mcl G$}\}$.  
We start by giving a precise definition of the notion of an end. 

\begin{defn}[End of $\mcl G$] \label{def-end}
Let $\mcl P$ be the set of one-sided infinite paths $P : \BB N_0 \to \mcl V\mcl G$ which are simple (i.e., they visit each vertex of $\mcl V\mcl G$ at most once). We define an equivalence relation $\sim$ on $\mcl P$ by $P\sim Q$ if and only if the following is true. For each finite set $A\subset \mcl V\mcl G$, there is exactly one connected component of $\mcl V\mcl G\setminus A$ which contains both infinitely many vertices of $P$ and infinitely many vertices of $Q$. An equivalence class $\omega$ under this equivalence relation is called an \textbf{end} of $\mcl G$. 
\end{defn}

The following lemma is easily verified from the definition of an end. 

\begin{lem} \label{lem-end-comp}
Let $\{V_n\}_{n\geq 1}$ be an increasing family of finite subsets of $\mcl V\mcl G$ whose union is all of $\mcl V\mcl G$. 
There is a bijection $\omega \mapsto \{C_n(\omega)\}_{n\geq 1}$ between ends of $\mcl G$ and sequences $\{C_n\}_{n \geq 1} $ of connected components of $\mcl V\mcl G\setminus V_n$ such that $C_n \subset C_{n-1}$ for each $n\geq 1$. 
Under this bijection, for each end $\omega$ and each equivalence class representative $P\in\omega$, the set $C_n(\omega)$ is the unique connected component of $\mcl V\mcl G\setminus V_n$ which contains infinitely many vertices of $P$. 
\end{lem}

The main result of this subsection is the following proposition.
 
\begin{prop} \label{prop-rw-end}
Let $X : [0,\infty) \to \mcl V\mcl G\cup \{\infty\}$ be the continuous time random walk on $\mcl G$ reflected off of $\infty$ as in Theorem~\ref{thm-cont-time-walk}. Almost surely, for each $t\in [0,\infty)$ such that $X_t =\infty$, there is a unique end $\omega_t$ of $\mcl G$ such that the following is true. 
Let $\{V_n\}_{n\geq 1}$ be an increasing family of finite, non-empty subsets of $\mcl V\mcl G$ whose union is all of $\mcl V\mcl G$ and let $\{C_n(\omega_t)\}_{n\geq 1}$ be the sequence of connected components of $\mcl V\mcl G_n\setminus V_n$ as in Lemma~\ref{lem-end-comp}. Then for each $n\geq 1$, there exists $\ep > 0$ such that 
\eqbn
X_s \in C_n(\omega_t) \cup \{\infty\}  ,\quad \forall s \in (t-\ep , t+\ep ) .
\eqen
\end{prop}

It is natural to re-define $X_t := \omega_t$ whenever $X_t =\infty$, so that $X$ becomes a function as in~\eqref{eqn-end-function}.

\begin{proof}[Proof of Proposition~\ref{prop-rw-end}]
Let $\{V_n\}_{n\geq 1}$ be as in the lemma statement and let $t \in [0,\infty)$ such that $X_t = \infty$.
Since we are assuming that $\mcl G$ is locally finite, the set $B_1 V_1$ is finite. 
By Lemma~\ref{lem-to-infty}, for each $n\geq 1$ there exists $\ep_n > 0$ such that $X[t-\ep_n , t] \cap B_1 V_n = \emptyset$. 
In particular, there exists a sequence of times $\{t_j\}_{j\in\BB N}$ increasing to $t$ such that $X_{t_j} \in \mcl V\mcl G\setminus V_n$ for each $j$. 

Let $T_n := \inf\{s \geq t : X_s \in V_n\}$. By Lemma~\ref{lem-no-jump}, for each $j\in\BB N$, there exists a time $s_j \in[t_j, T_n]$ such that $X_{s_j} \in B_1 V_n\setminus V_n$. This time $s_j$ cannot be in $[t-\ep_n ,t]$, so must be in $[t,T_n]$. In particular, $T_n > t$.  
Hence, by possibly shrinking $\ep_n$, we can arrange that $X[t-\ep_n , t+\ep_n] \cap V_n = \emptyset$.  
 
By Lemma~\ref{lem-no-jump}, there exists a single connected component $C_n$ of $\mcl V\mcl G\setminus V_n$ which contains $X[t-\ep_n , t+\ep_n]  \setminus \{\infty\} $. 
Since $\{V_n\}_{n\geq 1}$ is increasing, the sequence of connected components $\{C_n\}_{n\geq 1}$ is nested. Hence Lemma~\ref{lem-end-comp} shows that this sequence of connected components corresponds to an end $\omega_t$ as in the lemma statement.  
\end{proof}

\subsection{Tutte embedding}
\label{sec-tutte}

Let $\mcl G$ be a planar map equipped with a conductance function $\frk c$, a boundary $\bdy\mcl G$, and marked vertices $\BB x\in\mcl G\setminus \bdy\mcl G$ and $\BB y \in \bdy\mcl G$, as in Section~\ref{sec-tutte0}. As explained in that section, we can define the Tutte embedding $H : \mcl V\mcl G\to \ol{\BB D}$.  

We can extend the definition of $H$ beyond the vertex set as follows: 
\begin{itemize}
\item For each edge $e = \{x,y\} \in \mcl E\mcl G$, we define $H(e)$ to be the line segment from $H(x)$ to $H(y)$ (if $H(x) = H(y)$, this line segment consists of a single point). 
\item For each face $f $ of $\mcl G$, we define $H(f)$ to be the closed region which is the union of the line segments $H(e)$ for edges $e$ on the boundary of $f$ and the region disconnected from $\infty$ by these line segments. 
\item For each end $\omega$ of $\mcl G$ (Definition~\ref{def-end}), we define $H(\omega)$ as follows. Let $\{V_n\}_{n\geq 1}$ be an increasing sequence of finite subsets of $\mcl V\mcl G$ whose union is all of $\mcl V\mcl G$. For $n \in \BB N$, let $C_n(\omega)$ be the connected component of $\mcl V\mcl G\setminus V_n$ corresponding to $\omega$, as in Lemma~\ref{lem-end-comp}. Let $F_n$ be the set of faces of $\mcl G$ whose vertices all belong to $C_n(\omega)$. We define
\eqb \label{eqn-tutte-end}
H(\omega) := \bigcap_{n=1}^\infty \bigcup_{f\in F_n} H(f) \subset \ol{\BB D} .
\eqe 
It is easy to see that this definition does not depend on the choice of $\{V_n\}_{n\geq 1}$. 
\end{itemize}

We now prove some further results about the Tutte embedding. 
We say $\mcl H$ is a \textbf{submap} of $\mcl G$ if $\mcl H$ is a subgraph of $\mcl G$, with the planar map structure inherited from $\mcl G$. Note that each face of $\mcl G$ is contained in a face of $\mcl H$, but faces of $\mcl H$ need not be faces of $\mcl G$.

\begin{lem} \label{lem-tutte-conv}
Let $\{\mcl G_n\}_{n\geq 1}$ be an increasing family of finite connected submaps of $\mcl G$ whose union is all of $\mcl G$. Assume that $\bdy\mcl G\cup\{\BB x\} \subset\mcl G_n$ for every $n\geq 1$ and set $\bdy \mcl G_n := \bdy\mcl G$. 
Then the Tutte embeddings of $(\mcl G_n , \BB x , \BB y)$ converge pointwise on $\mcl V\mcl G$ to the Tutte embedding of $(\mcl G,\BB x , \BB y)$ as $n\to\infty$.
\end{lem}
\begin{proof}
Let $y_1,\dots,y_m$ be the vertices of $\bdy\mcl G$ in counterclockwise cyclic order, as in the definition of the Tutte embedding (Section~\ref{sec-tutte0}). For $k=1,\dots,m$ and $n\geq 1$, let $h_n^k$ be the function on $\mcl G_n$ which equals $\BB 1_{y_k}$ on $\bdy\mcl G$ and is $\mcl G_n$-discrete harmonic on $\mcl V\mcl G_n\setminus\bdy\mcl G$. Also let $h^k$ be the unique energy-minimizing function on $\mcl V\mcl G$ which equals $\BB 1_{y_k}$ on $\bdy \mcl G$ (Proposition~\ref{prop-min-harmonic}). By the definition of the Tutte embedding and Lemma~\ref{lem-hm-linear}, the Tutte embedding of $(\mcl G,\BB x,\BB y)$ is given by
\eqb \label{eqn-tutte-linear}
H(x) = \sum_{k=1}^m \exp\left( 2\pi i \sum_{j=1}^k h^j(\BB x) \right)  h^k(x) .
\eqe
Furthermore, for $n\in\BB N$, the Tutte embedding of $(\mcl G_n , \BB x , \BB y)$ is given by~\eqref{eqn-tutte-linear} with $h_n^j$ and $h_n^k$ in place of $h^j$ and $h^k$. By Proposition~\ref{prop-harmonic-conv}, for each $k=1,\dots,m$ we have $h_n^k(x) \to h^k(x)$ as $n\to\infty$. The lemma statement therefore follows from~\eqref{eqn-tutte-linear} and its analog for $\mcl G_n$.
\end{proof}

\begin{lem} \label{lem-tutte-convex}
Let $H$ be the Tutte embedding of $(\mcl G,\BB x,\BB y)$. 
For each non-external face $f$ of $\mcl G$, the region $H(f)$ is convex.  
Moreover, the regions $H(f)$ for different non-external faces $f$ have disjoint interiors. 
\end{lem}
\begin{proof}
Let $\{\mcl G_n\}_{n\geq 1}$ be as in Lemma~\ref{lem-tutte-conv}.
For $n\in\BB N$, let $H_n$ be the Tutte embedding of $(\mcl G,\BB x,\BB y)$.  It was proven by Tutte~\cite{tutte-embedding} that the image under $H_n$ of each non-external face of $\mcl G_n$ is a convex set and the images of different non-external faces have disjoint interiors. 
For each face $f$ of $\mcl G$, it holds for large enough $n\in\BB N$ that $f$ is also a face of $\mcl G_n$.
Sending $n\to\infty$ and using Lemma~\ref{lem-tutte-conv} therefore gives the lemma statement. 
\end{proof}

\subsection{Green's function}
\label{sec-green}

Recall the definition of the Green's function for random walk reflected off of $\infty$ and killed upon hitting $\mcl A$ (Definition~\ref{def-green}). In this section, we establish some basic properties of this function.

\begin{lem} \label{lem-green-finite}
For each $x,y\in \mcl V\mcl G$, we have
\eqb  \label{eqn-green-finite}
G_{\mcl A}(x,y)   < \infty .
\eqe 
Furthermore, for each fixed $y\in\mcl V\mcl G$, the function $x\mapsto G_{\mcl A}(x,y)$ is discrete harmonic on $\mcl V\mcl G\setminus (\mcl A \cup \{y\})$ and if $y\notin \mcl A$, then
\eqb \label{eqn-green-laplace}
 \sum_{x \sim y} \frk c(y,x) \left( G_{\mcl A}(x,y)  - G_{\mcl A}(y,y) \right)   =  - \pi(y) .
\eqe
\end{lem}
\begin{proof}
Let $\{\mcl G_n\}_{n \geq 1}$ be an increasing family of finite subgraphs of $\mcl G$ whose union is all of $\mcl G$, as in Section~\ref{sec-discrete-markov}. 
Also let $\{Y^n\}_{n\geq 0}$ be the Markov chain on $B_1\mcl G_n$ as in~\eqref{eqn-transition-walk} and~\eqref{eqn-transition-comp}, coupled together as in Lemma~\ref{lem-rw-coupling}. 
Assume that $n \in \BB N$ is sufficiently large so that $\mcl A\subset\mcl V\mcl G_n$ and let
\eqbn
T^n_{\mcl A} := \min\left\{ j \in \BB N_0 : Y^n_j \in \mcl A \right\} .
\eqen

Recall the discrete time random walk reflected off of $\infty$, $\{Y_\xi\}_{x\in\Xi}$, from~\eqref{eqn-process-on-xi}. 
Let $N_{\mcl A}(y)$ be the number of visits of $X$ to $y$ before it hits $\mcl A$, as in~\eqref{eqn-number-of-visits} and let $\eta_{\mcl A}$ be the smallest $\xi\in\Xi$ such that $Y_\xi \in \mcl A$. Then
\eqb  \label{eqn-number-of-visits-xi}
N_{\mcl A}(y)  = \#\left\{\xi\in\Xi \,:\, \xi < \eta_{\mcl A} ,\,  Y_\xi = y\right\} .
\eqe  
If $x  , y \in\mcl V\mcl G_n$ and the processes all start at $x$, then for each equivalence class $\xi\in\Xi$ for which $Y_\xi = y$, there exists $(n,j) \in \xi$ such that $Y^n_j = y$. Hence, if $x,y\in\mcl V\mcl G_n$, then~\eqref{eqn-number-of-visits-xi} shows that under $\BB P_x$, 
\eqbn
N_{\mcl A}(y) = \#\left\{ j \in \BB N_0 \,:\, j < T_{\mcl A}^n ,\, Y^n_j = y\right\} .
\eqen
Each $Y^n$ is an irreducible Markov chain on a finite state space, so it follows from standard Markov chain theory that $\BB E_x[N_{\mcl A}(y)] < \infty$. That is,~\eqref{eqn-green-finite} holds.

By the Markov property of $Y^n$ and the transition probability formula~\eqref{eqn-transition-walk}, if $x\in\mcl V\mcl G_n\setminus \{y\}$, then 
\eqbn
G_{\mcl A}(x,y) = \sum_{z \sim x} \frac{\frk c(x,z)}{\pi(x)} G_{\mcl A}(z,y)   
\eqen
which says precisely that $z\mapsto G_{\mcl A}(z,y)$ is discrete harmonic at $x$. If $y\notin \mcl A$, we also have
\eqb  \label{eqn-green-diagonal0}
G_{\mcl A}(y,y) = 1 +  \sum_{x \sim y} \frac{\frk c(y,x)}{\pi(y)} G_{\mcl A}(x,y)   
\eqe 
which re-arranges to give~\eqref{eqn-green-laplace}. 
\end{proof}

As is the case for most objects considered in this paper, the Green's function $G_{\mcl A}$ arises as the limit of Green's functions on finite subgraphs of $\mcl G$.

\begin{lem} \label{lem-green-conv}  
Let $\mcl A\subset\mcl V\mcl G$ be finite and non-empty. Let $\{\mcl G_n\}_{n\geq 1}$ be an increasing family of connected subgraphs of $\mcl G$ whose union is all of $\mcl G$ such that $\mcl A\subset\mcl V\mcl G_n$ for every $n\in\BB N$. 
For $n\in\BB N$, let $G_{\mcl A}^n$ be the Green's function for random walk on $\mcl G_n$ killed when it first hits $\mcl A$. 
Then for each $x,y\in\mcl V\mcl G$, 
\eqb  \label{eqn-green-conv}
\lim_{n\to\infty} \mcl G_{\mcl A}^n(x,y) = \mcl G_{\mcl A}(x,y) .
\eqe 
\end{lem}
\begin{proof}
For $n\in\BB N$ and $y\in\mcl V\mcl G_n$, let $h^y_n$ be the function on $\mcl G_n$ which is equal to 0 on $\mcl A$, is equal to 1 at $y$, and is $\mcl G_n$-discrete harmonic on $\mcl V\mcl G_n\setminus (\mcl A \cup \{y\})$. Also let $h^y$ be the energy-minimizing function on $\mcl V\mcl G$ which is equal to 0 on $\mcl A$ and is equal to 1 at $y$, as in Proposition~\ref{prop-min-harmonic}. By Proposition~\ref{prop-harmonic-conv}, for each $x\in\mcl V\mcl G$, we have 
\eqb \label{eqn-hit-prob-conv}
\lim_{n\to\infty} h_y^n(x) = h^y(x)  .
\eqe

To deduce the lemma statement from~\eqref{eqn-hit-prob-conv}, we need to relate $G_{\mcl A}(\cdot,y)$ to $h^y$. For random walk on finite graphs, the analogous relation is standard (see, e.g.,~\cite[Proposition 2.1]{lyons-peres}). One possible proof of this relation (but not every possible proof) works in the case of random walk reflected at $\infty$. We give this proof now. 
 
For $x\in \mcl V\mcl G_n$, the strong Markov property (Lemma~\ref{lem-wt-strong-markov}) and the relationship between random walk and harmonic functions (Property~\eqref{item-harmonic} of Theorem~\ref{thm-cont-time-walk}) show that, with $N_{\mcl A}(y)$ as in~\eqref{eqn-number-of-visits},
\eqb \label{eqn-green-relation}
 G_{\mcl A} (x,y) 
 = \BB E_x\left[ N_{\mcl A} (y) \right]   
 = h^y(x)  \BB E_y\left[ N_{\mcl A} (y)  \right]  
 = h^y(x) G_{\mcl A}(y,y) .
\eqe 
By plugging into~\eqref{eqn-green-relation} the expression for $ G_{\mcl A}(y,y)$ from~\eqref{eqn-green-diagonal0}, we get 
\eqb \label{eqn-green-sum}
  G_{\mcl A}(x,y) 
= h^y(x) \left(  1 + \frac{1}{\pi(y)}  \sum_{z \sim y} \frk c(y,z) G_{\mcl A}(z,y)  \right)  .
\eqe
Multiplying both sides of~\eqref{eqn-green-sum} by $\frk c(y,x) / \pi(y)$ and summing over $x \sim y$ gives
\eqbn
\frac{1}{\pi(y)}  \sum_{x\sim y} \frk c(y,x) G_{\mcl A}(x,y) 
=  \left(  1 +  \frac{1}{\pi(y)}  \sum_{z \sim y} \frk c(y,z) G_{\mcl A}(z,y)  \right) \frac{1}{\pi(y)}  \sum_{x\sim y} \frk c(y,x) h^y(x) .
\eqen
Solving gives
\eqb \label{eqn-green-soln}
\frac{1}{\pi(y)}  \sum_{z \sim y} \frk c(y,z) G_{\mcl A}(z,y) 
=   \frac{\frac{1}{\pi(y)}  \sum_{z \sim y} \frk c(y,z) h^y(z) }{ 1 -  \frac{1}{\pi(y)}  \sum_{z \sim y} \frk c(y,z) h^y(z) } . 
\eqe
We note that the denominator is not equal to zero since $\mcl G$ is connected, so there is at least one $z\sim y$ for which random walk started from $z$ and reflected from $\infty$ has a positive chance to hit $\mcl A$ before $y$. 
Plugging~\eqref{eqn-green-soln} into~\eqref{eqn-green-relation} gives that for each $x\in\mcl V\mcl G_n$, 
\eqb \label{eqn-green-harmonic}
 G_{\mcl A}(x,y)   = 
 h^y(x) \times \left( 1 + \frac{\frac{1}{\pi(y)}  \sum_{z \sim y} \frk c(y,z) h^y(z) }{ 1 -  \frac{1}{\pi(y)}  \sum_{z \sim y} \frk c(y,z) h^y(z) } \right) . 
\eqe

By exactly the same argument, we also have the analog of~\eqref{eqn-green-harmonic} with $G_{\mcl A}^n(\cdot,y)$ and $h^y_n(\cdot)$ in place of $G_{\mcl A} (\cdot,y)$ and $h^y (\cdot)$. 
We have $h_n^y(z) \in [0,1]$ for each $n\in\BB N$ and $z \in \mcl V\mcl G$.  
By~\eqref{eqn-hit-prob-conv} and the dominated convergence theorem (to deal with the possibility that $y$ could have infinite degree), we therefore have
\eqbn
\lim_{n\to\infty} \sum_{\substack{ z\in\mcl V\mcl G_n \\ z \sim y }    } \frk c(y,z) h_n^y(z) =  \sum_{z \sim y} \frk c(y,z) h^y(z) .
\eqen
By combining this with~\eqref{eqn-green-harmonic} and its analog for $G_{\mcl A}^n(\cdot,y)$, we get~\eqref{eqn-green-conv}. 
\end{proof}

The following lemma justifies that $G_{\mcl A}(x,y)  \pi(y)$ is a valid covariance kernel for a Gaussian process, so Definition~\ref{def-gff} makes sense. 

\begin{lem} \label{lem-green-reversible} 
For each $x,y\in\mcl V\mcl G $, 
\eqb  \label{eqn-green-reversible}
\frac{  G_{\mcl A} (x,y) }{\pi(y)}  =  \frac{ G_{\mcl A}(y,x) }{\pi(x)}  .
\eqe 
Furthermore, for any function $f:\mcl V\mcl G \to \BB R$ which vanishes outside of a finite set,
\eqb \label{eqn-green-pos} 
\sum_{x,y\in\mcl V\mcl G} f(x) f(y) \frac{G_{\mcl A}(x,y)}{\pi(y)} \geq 0.
\eqe
\end{lem}
\begin{proof}
As in Lemma~\ref{lem-green-conv}, let $\{\mcl G_n\}_{n\geq 1}$ be an increasing family of subgraphs of $\mcl G$ whose union is all of $\mcl G$ such that $\mcl A\subset\mcl V\mcl G_n$ for every $n\in\BB N$, and let $G_{\mcl A}^n$ be the Green's function for random walk on $\mcl G_n$ killed when it first hits $\mcl A$. 
The analog of~\eqref{eqn-green-reversible} for $  G_{\mcl A}^n$ is standard, see, e.g.,~\cite[Exercise 2.1]{lyons-peres}.
The analog of~\eqref{eqn-green-pos} for $G_{\mcl A}^n$ when $n$ is large enough so that the support of $f$ is contained in $\mcl V\mcl G_n$ follows from a discrete integration by parts calculation.
The lemma statement therefore follows from Lemma~\ref{lem-green-conv}.
\end{proof}

Finally, we give the proof of Kirkhoff's formula for the $\frk c$-FSF. 

\begin{proof}[Proof of Proposition~\ref{prop-fsf-edge}] 
Let $X$ be the random walk reflected off of $\infty$ on $\mcl G$ started at $X_0 = x$. 
Let $\tau_0 = 0$ and for $k\geq 1$, inductively let $\tau_k$ be the first time after $\tau_{k-1}$ at which $X$ jumps to $x$ from a vertex other than $x$. Let $K  $ be the smallest $k$ such that $y \in X([\tau_{k-1}, \tau_k])$. That is, $K = N_y(x)$ in the notation of~\eqref{eqn-number-of-visits}. By the strong Markov property (Lemma~\ref{lem-wt-strong-markov}), the increments $X|_{[\tau_{k-1} , \tau_k]}$ are i.i.d., so $K$ has a geometric distribution with success probability 
\eqbn
p := \BB P_x\left[ \text{$X$ visits $y$ before time $\tau_1$} \right] .
\eqen
By Definition~\ref{def-green}, we have 
\eqb \label{eqn-edge-green}
1/p = \BB E[K] = G_y(x,x) . 
\eqe

If we generate $\mcl T^{\op{FSF}}$ using the Aldous--Broder algorithm started from $x$ (Definition~\ref{def-rw-forest} and Theorem~\ref{thm-fsf}), we have $e \in \mcl T^{\op{FSF}}$ if and only if $a(y) =x$, equivalently, $y$ is the first vertex other than $x$ which $X$ visits after time $\tau_{K-1}$. Since the increments $X|_{[\tau_{k-1} , \tau_k]}$ are i.i.d., the probability that this is the case is equal to the conditional probability that $y$ is the first vertex other than $x$ which $X$ visits after time 0, given that $X$ visits $y$ before time $\tau_1$. By Property~\eqref{item-rw} of Theorem~\ref{thm-cont-time-walk}, the probability that $y$ is the first vertex other than $x$ visited after time 0 is $\frk c(x,y) / \pi(x)$. 
By combining this with~\eqref{eqn-edge-green}, we get that the aforementioned conditional probability is equal to the right side of~\eqref{eqn-edge-contain}.  
\end{proof}

\section{Conjectures for random planar maps in the supercritical LQG universality class}
\label{sec-supercritical-rpm}

The goal of this section is to state a number of  conjectures for random planar maps which are believed to converge to Liouville quantum gravity (LQG) with central charge $\cc \in (1,25)$, all of which are based on the theory of random walk reflected off of $\infty$. We provide some background on LQG (with references for further reading) in Section~\ref{sec-lqg}. We then define the precise family of random planar maps we will consider, following~\cite{ag-supercritical-cle4,bgs-supercritical-crt} in Section~\ref{sec-rpm}. In Section~\ref{sec-tutte-conj}, we state a scaling limit conjecture for these random planar maps under the Tutte embedding. In Section~\ref{sec-phase-transition}, we state conjectures about phase transitions (in the parameter $\cc$) for various stochastic processes on the maps. 

\subsection{Background on Liouville quantum gravity} 
\label{sec-lqg} 

Liouville quantum gravity (LQG) is a canonical one-parameter family of random fractal surfaces. Such surfaces were originally introduced by Polyakov~\cite{polyakov-qg1} in the context of bosonic string theory. They have deep connections to a variety of other topics in math and physics, including conformal field theory, random planar maps, Schramm--Loewner evolution, and random permutations, and (more speculatively) Yang--Mills theory. We refer to~\cite{gwynne-ams-survey,sheffield-icm} for overviews of the mathematical theory of LQG and~\cite{bp-lqg-notes} for a more comprehensive introduction. 

One can define LQG surfaces with the topology of any underlying Riemann surface, but in this paper we will only need to consider LQG surfaces with the topology of the unit disk $\BB D\subset \BB C$. Heuristically speaking, an LQG surface with the disk topology, with \textbf{central charge} $\cc > 1$, is the random 2d Riemannian manifold $(\BB D,g)$, where $g$ is sampled from the ``uniform measure on Riemmanian metric tensors on $\BB D$, weighted by $(\det \Delta_g)^{-(26-\cc)/2}$'', where $\Delta_g$ is the Laplace-Beltrami operator. 

Instead of the parameter $\cc$, it is common to instead use the equivalent parameters $Q > 0$ or $\gamma \in (0,2] \cup \{z\in\BB C  : |z| = 2\}$ which are related by
\eqb  \label{eqn-c-Q-gamma}
\cc = 1+6Q^2 ,\quad Q = \frac{2}{\gamma}  + \frac{\gamma}{2} .
\eqe 
The behavior of LQG undergoes a phase transition at $\cc = 25$ ($Q =\gamma=2$). We call the case when $\cc  > 25$ the \textbf{supercritical phase} and the case when $\cc \in (1,25)$ the \textbf{subcritical phase}.
Most works on LQG consider only the subcritical and critical cases, $\cc \geq 25$ (equivalently $\gamma\in (0,2]$). But, the supercritical phase is potentially more interesting from the perspective of string theory and Yang--Mills theory (see~\cite[Remark 1.4]{ag-supercritical-cle4} for some discussion of this). Furthermore, some recent mathematical works have begun to investigate supercritical LQG~\cite{ghpr-central-charge,dg-supercritical-lfpp,pfeffer-supercritical-lqg,dg-uniqueness,dg-confluence,ag-supercritical-cle4,apps-central-charge,bgs-supercritical-crt}. 

The above definition of LQG does not make literal sense, but it is possible to make rigorous sense of LQG as the random geometry obtained by exponentiating a generalized function $\Phi$ on $\BB D$ which locally looks like the Gaussian free field (GFF). In particular, for $\cc \geq 25$ one can define the volume form (area measure) $\mu_\Phi$ associated with LQG as a limit of regularized versions of $e^{\gamma \Phi} \,dx\,dy$~\cite{kahane,shef-kpz,rhodes-vargas-log-kpz,shef-deriv-mart,shef-renormalization}. There is no volume form for $\cc \in (1,25)$~\cite[Theorem 1.5]{bgs-supercritical-crt}. 

Furthermore, for all $\cc > 1$, one can define the distance function (metric) associated with LQG as a limit of regularized versions of the metric $D_\Phi(z,w) =  \inf_{P:z\to w} \int_0^1 e^{\xi P(t)} |P'(t)| \,dt$, where the inf is over all piecewise $C^1$ paths in $\BB D$ from $z$ to $w$~\cite{dddf-lfpp,gm-uniqueness,dg-supercritical-lfpp,dg-uniqueness}. Here, $\xi > 0$ is a $\cc$-dependent parameter which is not known explicitly. 

For $\cc \geq 25$, the LQG metric induces the same topology on $\BB D$ as the Euclidean metric. In contrast, for $\cc \in (1,25)$, the LQG metric does not induce the Euclidean topology. Rather, there is an uncountable, Euclidean-dense, zero-Lebesgue measure set of \textbf{singular points} $z \in \BB D$ which satisfy 
\eqb \label{eqn-singular-pt}
D_\Phi(z,w) = \infty,\quad \forall w\in\BB D\setminus\{z\} . 
\eqe
Roughly speaking, singular points correspond to $\alpha$-thick points of $\Phi$ with $\alpha  \in (Q,2]$~\cite[Proposition 1.11]{pfeffer-supercritical-lqg} (see Proposition~\ref{prop-singular-thick} below). 

For all $\cc  > 1$, LQG is expected (and in some cases proven) to describe the scaling limit of various types of random planar maps. For example, uniform random planar maps (including triangulations, quadrangulations, etc.) converge --- as the number of edges goes to $\infty$ and areas and distances are re-scaled appropriately --- to LQG with $\cc = 26$ ($\gamma=\sqrt{8/3}$)~\cite{legall-uniqueness,miermont-brownian-map,lqg-tbm1,lqg-tbm2,hs-cardy-embedding}. Random planar maps which are believed converge to LQG with central charge $\cc$ are said to belong to the \textbf{central charge-$\cc$ LQG universality class}. See Section~\ref{sec-tutte-conj} for some discussion about various possible topologies of convergence.

\subsection{Supercritical random planar maps} 
\label{sec-rpm}

In this paper, we will be interested in random planar maps in the central charge-$\cc$ LQG universality class for $\cc \in (1,25)$. Such planar maps are infinite (at least with high probability), with infinitely many ends (Definition~\ref{def-end}). Roughly speaking, the ends of the random planar map correspond to the singular points of the LQG metric (as defined in~\eqref{eqn-singular-pt}). Indeed, the ends of the map are precisely the ``points'' which lie at infinite graph distance from every other point. 

The first paper to construct combinatorially natural random planar maps in the central charge-$\cc$ LQG universality class was~\cite{ag-supercritical-cle4}. A slightly different class of random planar maps which are more combinatorially tractable was defined in~\cite{bgs-supercritical-crt}. Roughly speaking, these random planar maps are a variant of maps decorated by the $O(2)$ loop model, with the loops modified to force the ratio of their inner and outer boundary lengths to be approximately $\exp\left( \pm \pi \frac{\sqrt{ 4-Q^2} }{Q} \right)$. 

For concreteness, we recall the construction from~\cite{bgs-supercritical-crt} here. But, the details are not important for the rest of the paper. All of the discussion in this section should apply for any reasonable random planar map models which converge to supercritical LQG (see Remark~\ref{remark-other-rpm}). Consequently, the reader may wish to skip this section on a first read and take the definition of the maps $\mcl M_k$ as  black box. 

The construction of the random planar maps from~\cite{bgs-supercritical-crt} is is inspired by the gasket decomposition for random planar maps decorated by the $O(n)$ loop model~\cite{bbg-recursive-approach}, and is a discrete analog of the coupling of supercritical LQG and CLE$_4$ from~\cite{ag-supercritical-cle4}. The construction is based on iteratively gluing together two types of planar maps: \textbf{non-generic critical Boltzmann maps of type 2}, and \textbf{rings}. We now review the definitions of these two types of random planar maps. The following definitions are standard in the random planar maps literature, see, e.g.,~\cite{curien-peeling-notes}.

\begin{defn} \label{def-boltzmann}
For $k  \in \BB N$, let $\BB M^{(k)}$ be the set of finite bipartite planar maps $\frk m$ with a distinguished face (called the \textbf{external face}) of degree $2k$. 
For $\frk m \in \BB M^{(k)}$, we write $\mcl F \frk m$ for its set of non-external faces. 
Let $\mathbf{q} = \{q_i\}_{i \in \BB N}$ be a sequence of non-negative real numbers. We say that $\bf{q}$ is \textbf{admissible} if 
\eqbn
W_k(\mathbf{q}) := \sum_{\frk m \in \BB M^{(k)}} \prod_{f \in \mcl F \frk m}  q_{\deg(f) /2} < \infty , \quad \forall k \in \BB N .
\eqen
In this case, we define the \textbf{Boltzmann map} of perimeter $2k$ with weights $\bf q$ to be the random planar map $M $ such that for each $\frk m \in \BB M^{(k)}$, 
\eqbn
\BB P\left[ M = \frk m \right] =  \frac{1}{W_k(\bf{q})}  \prod_{f \in \mcl F \frk m}  q_{\deg(f) /2} .
\eqen
\end{defn}

Note that we do not require the boundary of a Boltzmann map to be simple. In particular, a Boltzmann map has a positive chance to have no faces except for the external face, in which case it is a tree with $k$ edges. 

\begin{defn} \label{def-critical} 
For an admissible weight sequence $\bf q$, define 
\eqbn
f_{\bf{q}}(x) :=\sum_{k=1}^\infty q_k \binom{2k-1}{k-1} x^{k-1} .
\eqen
If $\bf q$ is admissible, then there exists a smallest $Z_{\bf q} > 0$ such that $f_{\bf q}(Z_{\bf q}) = 1- 1/ Z_{\bf q}$. 
We say that $\bf q$ is \textbf{non-generic critical of type 2} if 
\eqbn
	q_k \binom{2k-1}{k-1} (Z_{\mathbf{q}})^{k-1} \sim k^{-5/2} \quad \text{as $k \to \infty$}.
\eqen
\end{defn}

Equivalent definitions of critical and non-generic critical weight sequences can be found in~\cite[Theorem 5.4 and Proposition 5.10]{curien-peeling-notes}. An example of a Boltzmann map with a non-generic critical weight sequence of type 2 is the gasket of  a certain type of a random planar map decorated by the $O(2)$ loop model (in particular, the model of~\cite{adh-volume} with $g = h/2$, see~\cite[Appendix B]{adh-volume}).\footnote{For some other types of random planar maps decorated by the $O(2)$ loop model, the gasket satisfies a more general version of non-generic criticality where we allow a slowly varying correction~\cite{adh-volume,kammerer-O2-gasket}.}

If $\bf q$ is a non-generic critical weight sequence of type 2, then the Boltzmann map with perimeter $2k$ and weights $\bf q$ should converge in the scaling limit to a version of the so-called \textbf{$3/2$-stable map} with respect to the Gromov--Hausdorff distance. See~\cite{cmr-large-faces} for a proof of this convergence for a slightly different definition of Boltzmann maps than the ones considered here (where we specify the number of vertices rather than the perimeter).

\begin{defn} \label{def-ring}
For $p \geq 1$ and $q\geq 0$, a \textbf{ring} with inner perimeter $p$ and outer perimeter $q$ is a planar map $R$ with two distinguished faces: an outer face of perimeter $2p$ and and an inner face of perimeter $2q$, with the following properties. The two distinguished faces do not share any vertices, and every face of $R$ other than the two distinguished faces is a quadrilateral which shares a vertex with each of the inner face and the outer face. 
\end{defn}

Rings will play a role analogous to $O(n)$ loops in the gasket decomposition of~\cite{bbg-recursive-approach}. 
We can now describe the construction of random planar maps in the supercritical LQG universality class. 
See Figure~\ref{fig-model} for an illustration. 
 
\begin{figure}
\centering
\includegraphics[width=0.8\linewidth]{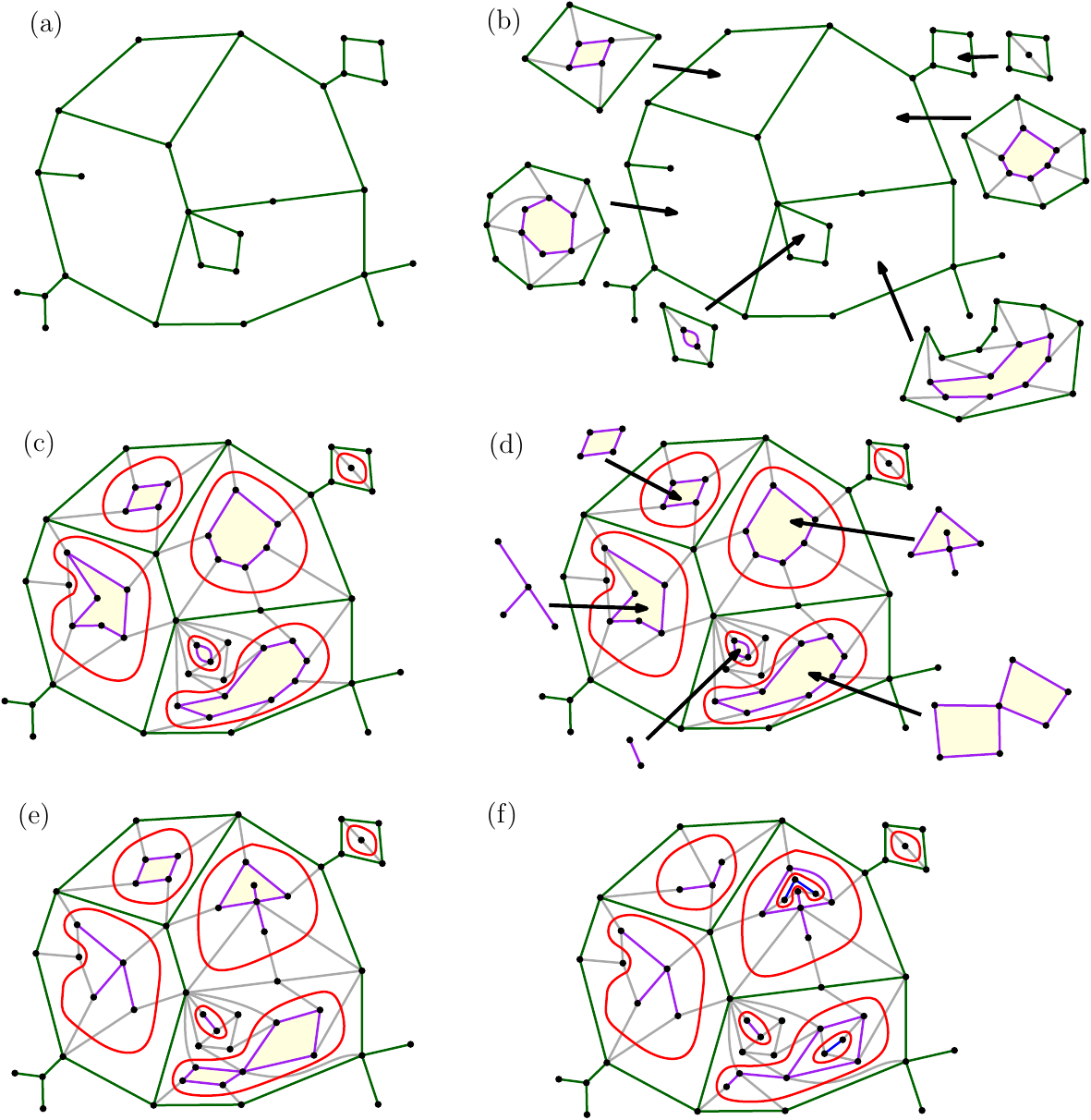}
\caption{\label{fig-model} An illustration of the iterative construction in Definition~\ref{def-supercritical-map} with $k=13$, taken from~\cite{bgs-supercritical-crt}. (a) We start with a non-generic critical Boltzmann map $\frk M_0$ of type 2 and perimeter $k$. (b) Given $\frk M_0$, we sample a ring $R_h$ for each hole (non-external face) of $\frk M_0$. The outer/inner boundaries of the rings are colored in green/purple, respectively. The ring in the top-right of the figure has inner perimeter zero. (c) The rings are attached to the corresponding faces, with the possible rotations chosen uniformly. The rings are identified with red loops that separate inner and outer boundaries. These loops are the discrete analogs of CLE$_4$ loops in the coupling of supercritical LQG disk with CLE$_4$ from~\cite{ag-supercritical-cle4}. (d) Given the previous figure, for each ring $R(f)$,  sample conditionally independent Boltzmann maps $M_h$ with perimeter equal to the inner perimeter of the ring. (e) The Boltzmann maps are glued to the inner boundaries of the rings, with the possible rotations chosen uniformly. This gives the map $\frk M_1$ The holes $h\in \mcl H_1$ are colored in yellow. (f) The map after another iteration. In this case, $\mcl H_2 = \emptyset$, so the construction terminates at this stage, giving us a finite map. However, when $k$ is large the construction has a high probability to go on forever (Proposition~\ref{prop-rpm-infinite}).
}
\end{figure}

\begin{defn}  \label{def-supercritical-map}
Fix $\cc \in (1,25)$, let $Q = \sqrt{\frac{\cc-1}{6}} \in (0,2)$, and let $\bf q$ be a non-generic critical weight sequence of type 2 (Definition~\ref{def-critical}). 
Also let $k\in\BB N$. 
We inductively define random planar maps $\{\frk M_j\}_{j\in\BB N}$, each of which is equipped with a set of \textbf{holes} $\mcl H_j \subset \mcl F \frk M_j$, as follows. 
\begin{itemize}
\item
Let $\frk M_0$ be a Boltzmann map with perimeter $2k$ and weights $\bf q$ (Definition~\ref{def-critical}) and let $\mcl H_0$ be the set of faces of $\frk M_0$ other than the external face. 
\item 
Inductively, assume that $j \geq 1$ and $(\frk M_{j-1}, \mcl H_{j-1})$ has been defined. 
Conditional on $(\frk M_{j-1} , \mcl H_{j-1})$, let $\{Y_h\}_{h \in \mcl H_{j-1}}$ be i.i.d.\ random variables with the uniform distribution on $\{-1,1\}$, indexed by the holes $ h\in \mcl H_{j-1}$. Conditional on $(\frk M_{j-1} , \mcl H_{j-1}, \{Y_h\}_{h\in\mcl H_{j-1}})$, for each $h\in\mcl H_{j-1}$, sample a ring $R_h$ from the uniform distribution on rings of outer perimeter $\deg(h)$ and inner perimeter\footnote{
The reason for the strange looking quantity $\exp\left( \pm \frac{\pi\sqrt{4-Q^2}}{Q}  \right)$ in~\eqref{eqn-perimeter-ratio} is that this quantity coincides with the ratio of the LQG lengths of a CLE$_4$ loop as measured from inside and outside the loop in the coupling of~\cite{ag-supercritical-cle4}, see~\cite[Proposition 2.16]{ag-supercritical-cle4}. 
}
\eqb  \label{eqn-perimeter-ratio}
2 \left\lfloor \exp\left( Y_h  \frac{\pi\sqrt{4-Q^2}}{Q}  \right) \frac{\deg(h)}{2} \right\rfloor  .
\eqe  
We take the rings $R_h$ to be conditionally independent given $(\frk M_{j-1} , \mcl H_{j-1}, \{Y_h\}_{h\in\mcl H_{j-1}})$. 
For each $h\in\mcl H_{j-1}$, we identify the edges of the outer boundary of $R_h$ to the edges of the boundary of $h$ in cyclic order. There are $\deg(h)$ ways to do this identification. We choose one way uniformly at random. 
\item 
Conditional on $(\frk M_{j-1} , \mcl H_{j-1}, \{R_h\}_{h\in\mcl H_{j-1}})$, for each $h\in\mcl H_{j-1}$ such that $R_h$ has non-zero inner perimeter, we sample a Boltzmann map $M_h$ of perimeter equal to the inner perimeter~\eqref{eqn-perimeter-ratio} of $R_h$ (we take the maps $M_h$ to be conditionally independent). We then identify the edges of the boundary of $M_h$ to the edges of the inner boundary of $R_h$ in cyclic order (as above, we choose one of the possible ways to do the identification uniformly at random). 
\item 
We define $\frk M_j$ to be the map which is the union of $\frk M_{j-1}$, $\{R_h\}_{h\in\mcl H_{j-1}}$, and $\{M_h\}_{h\in\mcl H_{j-1}}$, subject to the above identifications. We define $\mcl H_j$ to be the set of faces of $\frk M_j$ which are also faces of $M_h$ for some $h\in\mcl H_{j-1}$. 
\item We define $\mcl M = \mcl M_k$ to be the planar map which is the increasing union of the maps $\frk M_j$, and call it the \textbf{supercritical random planar map of central charge $\cc$ and perimeter $2k$}. We define $\mcl L_k$ to be the collection of simple loops on the dual of $\mcl M_k$ such that each loop traverses the faces of one of the rings in the above construction (these loops are shown in red in Figure~\ref{fig-model}). 
\end{itemize}
\end{defn}

We note that the only dependence on $\cc$ in Definition~\ref{def-supercritical-map} is the ratio between the inner and outer perimeters of the rings, as specified in~\eqref{eqn-perimeter-ratio}. 

The map $\mcl M_k$ of Definition~\ref{def-supercritical-map} can be finite or infinite, depending on whether the construction terminates (i.e., $\mcl H_j =\emptyset$) for some finite value of $j$. The following proposition follows from~\cite[Proposition 1.3]{bgs-supercritical-crt} and the Markovian nature of the construction in Definition~\ref{def-supercritical-map}. We recall the definition of ends from Definition~\ref{def-end}.  

\begin{prop}[\!\!{\cite[Proposition 1.3]{bgs-supercritical-crt}}] \label{prop-rpm-infinite}
There exists $\alpha = \alpha(\cc,\bf q) > 0$ such that as $k\to\infty$,
\eqbn
\BB P\left[ \text{$\mcl M_k$ is infinite, with infinitely many ends} \right] = 1 - e^{-\alpha  k + o(k) } .
\eqen
\end{prop}

We will prove in future work that on the event that $\mcl M_k$ is infinite, a.s.\ the random walk on $\mcl M_k$ (with unit conductance) is transient. Hence, Theorem~\ref{thm-cont-time-walk} has non-trivial content for $\mcl M_k$. 

As explained in~\cite[Section 3.2]{ag-supercritical-cle4}, the supercritical random planar map of central charge $\cc$ is an exact discrete analog of the coupling of the central charge-$\cc$ LQG disk and CLE$_4$ introduced in~\cite{ag-supercritical-cle4}, with the loops $\mcl L_k$ playing the role of the CLE$_4$ loops. From this perspective, is is natural to expect that $(\mcl M_k,\mcl L_k)$ converges in the scaling limit to central charge-$\cc$ LQG decorated by CLE$_4$. The conjectural topology of convergence was not specified in~\cite{ag-supercritical-cle4}, since, as we discuss in Section~\ref{sec-tutte-conj}, it is not obvious how to define an appropriate topology. We will give a precise scaling limit conjecture for $(\mcl M_k,\mcl L_k)$ toward supercritical LQG under the Tutte embedding in Conjecture~\ref{conj-scaling-limit} just below. 

\begin{remark}[Other random planar map models] \label{remark-other-rpm}
In addition to the random planar maps from Definition~\ref{def-supercritical-map}, there are other random planar map models which are believed to converge to supercritical LQG. These include random planar maps given by variants of the construction in Definition~\ref{def-supercritical-map} (see~\cite{ag-supercritical-cle4,bgs-supercritical-crt}) and random planar maps constructed directly from the Gaussian free field. One class of random planar maps of the latter type are the dyadic tilings of $\BB C$ considered in~\cite{ghpr-central-charge}. These tilings consist of dyadic squares which all have approximately the same ``central charge-$\cc$ LQG size''. For $\cc \in (1,25)$, there are infinitely many points in $\BB C$ where the squares accumulate. For the sake of concreteness, we focus on the model of Definition~\ref{def-supercritical-map}, but we expect that the conjectures given in this section are all valid for these other random planar map models as well. 
\end{remark}

\subsection{Convergence under the Tutte embedding}
\label{sec-tutte-conj}

As noted in Section~\ref{sec-lqg}, for $\cc \geq 25$, there are a variety of random planar maps which are conjectured (and in a few cases proven) to converge in the scaling limit to LQG with central charge $\cc$. Moreover, various discrete models on such random planar maps are conjectured to converge to their continuum counterparts, often described in terms of SLE. One of the most natural topologies is convergence of the (possibly decorated) random planar map under a certain embedding into the plane.  

In the supercritical case $\cc \in (1,25)$, the random planar maps which are conjectured to converge to LQG are infinite, with infinitely many ends, with high probability (see, e.g., Proposition~\ref{prop-rpm-infinite}). Hence, it is not obvious how to define embeddings of such random planar maps. Additionally, supercritical LQG surfaces, equipped with the supercritical LQG metric, are not locally compact~\cite[Proposition 1.14]{pfeffer-supercritical-lqg}, so one cannot look at Gromov--Hausdorff convergence. Consequently, it is not a priori obvious how to define a topology under which random planar maps should convergence to LQG. 

As explained in Section~\ref{sec-tutte0}, the results of this paper allow us to define Tutte embeddings of random planar maps with infinitely many ends. In this section, we use such embeddings to formulate a precise scaling limit conjecture for random planar maps toward supercritical LQG. This answers~\cite[Problem 4.4]{ag-supercritical-cle4} in the case of the Tutte embedding. It is still of interest to find versions of other embeddings (e.g., circle packing) which make sense for random planar maps in the supercritical LQG universality class.

Fix $\cc \in (1,25)$ and let $Q = \sqrt{\frac{\cc - 1}{6}}$, as in~\eqref{eqn-c-Q-gamma}. For $k \in\BB N$, let $(\mcl M_k , \mcl L_k)$ be the loop-decorated supercritical random planar map with central charge $\cc$ of perimeter $2k$, as in Definition~\ref{def-supercritical-map}. 

Let $(\BB D , \Phi , \Gamma)$ be the unit boundary length LQG disk decorated by CLE$_4$, as in~\cite[Definitions 2.7 and 2.12]{ag-supercritical-cle4}. We emphasize that $\Gamma$ is neither independent from nor determined by $\Phi$~\cite[Theorem 2.9]{ag-supercritical-cle4}. We choose the embedding $(\BB D,\Phi,\Gamma)$ so that $1$ corresponds to a point sampled from the LQG length measure on $\bdy\BB D$ and 0 corresponds to a point sampled from the LQG measure on the outermost gasket of $\Gamma$, i.e., the closure of the union of the outermost CLE$_4$ loops. Given an arbitrary embedding $(\Phi,\Gamma)$, such an embedding can be obtained as follows. Conditional on $\Phi$, let $y$ be sampled from the central charge-$\cc$ LQG length measure $\nu_\Phi$ on $\bdy\BB D$~\cite[Equation (2.5)]{ag-supercritical-cle4}. Also let $\frk m_\Phi$ be the central charge-$\cc$ LQG measure on the outermost gasket of $\Gamma$,\footnote{A rigorous construction of this measure has not yet been written down, but it should arise, e.g., as a Gaussian multiplicative chaos with respect to the canonical Euclidean measure on $\op{gask}(\Gamma)$~\cite[Proposition 4.5]{ms-cle-measure}.}
 and, conditional on $(\Phi,\Gamma)$, let $x$ be sampled from this measure. We take $x$ and $y$ to be conditionally independent given $(\Phi,\Gamma)$. Let $f : \BB D \to \BB D$ be the unique conformal map which takes 0 to $x$ and 1 to $y$. Then consider the embedding
\eqbn
(\BB D , \Phi\circ f + Q \log|f'| , f^{-1}(\Gamma) ) .
\eqen

\begin{conj} \label{conj-scaling-limit}
Let $(\mcl M_k , \mcl L_k)$ and $(\Phi,\Gamma)$ be as above. 
For $k\in \BB N$, let $\BB y_k$ be sampled uniformly from $\bdy\mcl M_k$ and let $\BB x_k$ be a point sampled uniformly from the outermost gasket of $(\mcl M_k , \mcl L_k)$. Let $H_k : \mcl V\mcl M_k \to \ol{\BB D}$ be the Tutte embedding of $(\mcl M_k , \BB x_k , \BB y_k)$, as in Section~\ref{sec-tutte}. We have the following joint convergence in law as $k\to\infty$. 
\begin{itemize}
\item The counting measure on vertices of $\bdy\mcl M_k$, re-scaled by $1/(2k)$, converges weakly to the LQG length measure $\nu_\Phi$ on $\bdy\BB D$. 
\item The loop ensembles $H_k(\mcl L_k)$ converge to the CLE$_4$ $\Gamma$ with respect to, e.g., the Hausdorff distance on countable collections of loops in $\BB D$ viewed modulo time parametrization (see, e.g.,~\cite[Section 4.1]{shef-cle}). 
\item  There are deterministic scaling factors $\{\frk b_k\}_{k\geq 1}$ with $\frk b_k = k^{\xi Q / 2 + o(1)}$ such that the graph distance on $H(\mcl M_k)$, re-scaled by $\frk b_k^{-1}$, converges to the LQG metric $D_\Phi$ with respect to the topology on lower semicontinuous functions\footnote{In order to view the graph distance on $d_k : H_k(\mcl M_k)$ as a lower semincontinuous function $d_k : \ol{\BB D} \times \ol{\BB D} \to [0,\infty]$, instead of as a function $H_k(\mcl V\mcl G_k) \times H_k(\mcl V\mcl G_k) \to [0,\infty)$, we can, e.g., proceed as follows. 
Recall the definition of the set $H_k(f) \subset \ol{\BB D}$ for a face $f$ of $\mcl M_k$ from Section~\ref{sec-tutte}.
For $z,w\in \ol{\BB D}$, we define 
\[
d_k(z,w) = \inf_{f , f' : z \in H_k(f) , w \in H_k(f')} \inf_{x \in \bdy f , y \in \bdy f'} d_{\mcl M_k}(x,y) , 
\]
where $d_{\mcl M_k}$ denotes graph distance and the infimum is over all faces $f,f'$ of $\mcl M_k$ such that $z \in H_k(f)$ and $w \in H_k(f')$ and all vertices $x$ and $y$ lying on the boundaries of $f$ and $f'$, respectively. Note that the infimum is infinite if one of $z$ or $w$ does not lie in $H_k(f)$ for any face $f$ of $\mcl M_k$. 
} 
$\ol{\BB D} \times \ol{\BB D} \to [0,\infty]$ (see, e.g.,~\cite{beer-usc} or~\cite[Section 1.2]{dg-supercritical-lfpp}).  
\item Let $X^k$ be the continuous time random walk on $\mcl M_k$ reflected off of $\infty$ started from $\BB x_k$ and stopped when it first hits $\bdy \mcl M_k$ (with any choice of rate function satisfying the condition of Theorem~\ref{thm-cont-time-walk}). Then $H_k(X^k)$ converges with respect to the topology on curves viewed modulo time parametrization~\cite{ab-random-curves} to standard planar Brownian motion started from 0 sampled independently from $(\Phi,\Gamma)$ and stopped when it first hits $\bdy\BB D$.  
\item Let $\Psi_k$ be the discrete GFF on $\mcl M_k$ with zero boundary condition on $\bdy\mcl M_k$ and free boundary condition at $\infty$ (Definition~\ref{def-gff}). Then $\sum_{x\in\mcl V\mcl M_k} \Psi_k(x) \BB 1_{H_k(x)}(\cdot)$ converges in the distributional sense to a zero-boundary GFF on $\BB D$ sampled independently from $(\Phi, \Gamma)$. 
\end{itemize}
\end{conj}

Recall that the singular points~\eqref{eqn-singular-pt} of the LQG metric $D_\Phi$ are the continuum analog of the ends of $\mcl M_k$. The set of singular points of $D_\Phi$ is totally disconnected: indeed, this follows from the local finiteness property of $D_\Phi$, see~\cite[Axiom V]{dg-uniqueness}, and the fact that paths of finite $D_\Phi$-length cannot hit singular points. Consequently, it is natural to make the following conjecture.

\begin{conj} \label{conj-tutte-end}
For each end $\omega$ of $\mcl M_k$, define its image $H_k(\omega) \subset \ol{\BB D}$ under the Tutte embedding as in~\eqref{eqn-tutte-end}. Then a.s.\ $H_k(\omega)$ is a singleton for each end of $\mcl M_k$.
\end{conj}

\subsection{Phase transition conjectures}
\label{sec-phase-transition}

Let $(\mcl M_k, \mcl L_k)$ and $(\Phi,\Gamma)$ be as in Section~\ref{sec-tutte-conj}. 
In this section we make several conjectures for the existence of phase transitions for the qualitative behavior of various processes on the random planar map $\mcl M_k$, occurring at particular values of $\cc \in (1,25)$. Such processes include the random walk reflected off $\infty$, the free uniform spanning forest, and critical percolation. 

The basic idea behind the conjectures in this section is that existing results from the LQG literature (which we discuss just below) allow us to determine whether or not a random Borel set sampled independently from $\Phi$ intersects the set of singular points~\eqref{eqn-singular-pt} of $D_\Phi$. From this and the (heuristic) analogy between singular points and ends of $\mcl M_k$, we can make conjectures about whether or not various random subsets of the map $\mcl M_k$ accumulate at the ends of $\mcl M_k$ (i.e., whether or not these random subsets are infinite). This leads to conjectures about the qualitative properties of such sets.
 
It is shown in~\cite[Proposition 1.11]{pfeffer-supercritical-lqg} that the singular points can be described in terms of the so-called thick points of $\Phi$, whose definition we now recall. 
For $z\in\BB D$ and $r\in (0,1-|z|)$, let $\Phi_r(z)$ be the average of $\Phi$ on the circle of radius $r$ centered at $z$. See~\cite[Section 3.1]{shef-kpz} or~\cite[Section 1.12]{bp-lqg-notes} for the basic properties of circle averages. For $\alpha \in \BB R$, we define the set $T_{> \alpha}$ of \textbf{points of thickness greater than $\alpha$} and the set $T_{< \alpha}$ of \textbf{points of thickness less than $\alpha$}, respectively, by
\allb  \label{eqn-thick-pt-def}
T_{> \alpha} := \left\{ z\in \BB D : \limsup_{r \to 0} \frac{\Phi_r(z)}{\log r^{-1}} > \alpha \right\}  
\quad \text{and} \quad
T_{< \alpha} := \left\{ z\in \BB D : \limsup_{r \to 0} \frac{\Phi_r(z)}{\log r^{-1}} < \alpha \right\}  
\alle
We note that some literature use a limit or a liminf instead of a limsup in the definition of the $\alpha$-thick points. Then we have the following statement.

\begin{prop}[\!\!{\cite[Proposition 1.11]{pfeffer-supercritical-lqg}}] \label{prop-singular-thick}
Let $Q = \sqrt{\frac{\cc - 1}{6}} \in (0,2)$. Define $T_{>\alpha}$ and $T_{<\alpha}$ as in~\eqref{eqn-thick-pt-def}. Then each point in $T_{>Q}$ is a singular point for $D_\Phi$, and each point in $T_{<Q}$ is not a singular point for $D_\Phi$.
\end{prop}

If $\limsup_{r \to 0} \frac{\Phi_r(z)}{\log r^{-1}} = Q$, then $z$ could be a singular point or a non-singular point, depending on the second-order behavior of $\Phi_r(z)$ as $r\to 0$. 

Let $A\subset \BB D$ be a random Borel set sampled independently from $\Phi$ and let $\dim_{\mcl H} A$ denote its Hausdorff dimension with respect to the Euclidean metric.   
By combining Proposition~\ref{prop-singular-thick} with a result for the dimension of the intersection of $A$ with the $\alpha$-thick points~\cite[Theorem 4.1]{ghm-kpz}, we get that
\eqb \label{eqn-singular-dim-lower}
\text{On the event $\{\dim_{\mcl H} A > Q^2/2\}$, a.s.\ $A$ contains uncountably many singular points of $D_\Phi$.}
\eqe
We expect that also
\eqb \label{eqn-singular-dim-upper}
\text{On the event $\{\dim_{\mcl H} A < Q^2/2\}$, a.s.\ $A$ is disjoint from the set of singular points of $D_\Phi$.}
\eqe  
The statement~\eqref{eqn-singular-dim-upper} has not exactly been proven in the literature, but several closely related statements have been proven.\footnote{\cite[Theorem 4.1]{ghm-kpz} proves an analogous statement to~\eqref{eqn-singular-dim-upper} with singular points replaced by thick points, but with either a $\lim$ or a $\liminf$ instead of a $\limsup$ in the definition of thick points. The proof of~\cite[Theorem 1.15]{pfeffer-supercritical-lqg} gives the analog of~\eqref{eqn-singular-dim-upper} with packing dimension in place of Hausdorff dimension. We expect that~\eqref{eqn-singular-dim-upper} can be established by adapting the ideas in these proofs, but we will not carry this out here.}
 
In the case when $\dim_{\mcl H} A = Q^2/2$, more refined information about $A$ is needed to determine whether $A$ contains singular points of $D_\Phi$ or not. We will not address the case when $\dim_{\mcl H} A = Q^2/2$ in this paper.

\subsubsection{Range of the random walk reflected off of $\infty$}

As a first example, let $B$ be a standard planar Brownian motion started from 0, sampled independently from $\Phi$ and stopped at the first time $\tau$ when it hits $\bdy \BB D$. We say that $z\in B([0,\tau])$ is a \textbf{cut point} of $B$ if $B([0,\tau])\setminus \{z\}$ is not connected. It is shown in~\cite{lsw-bm-exponents2} that the Hausdorff dimension of the cut points of $B$ (with respect to the Euclidean metric) is a.s.\ equal to $3/4$. By~\eqref{eqn-singular-dim-lower} and~\eqref{eqn-singular-dim-upper}, the set of cut points of $B$ a.s.\ contains singular points of $D_\Phi$ if $Q <  \sqrt{3/2}$, and should a.s.\ not contain singular points if $Q > \sqrt{3/2}$ (the situation for $Q =\sqrt{3/2}$ is unclear). Equivalently, there is a path in $B([0,\tau])$  between any two points of $B([0,\tau])$ which avoids the singular points of $D_\Phi$ if $Q > \sqrt{3/2}$, but not if $Q < \sqrt{3/2}$.

Now let $X^k$ be the continuous time random walk reflected off of $\infty$ on $X^k$ stopped at the first time $\tau_k$ when it hits $\bdy \mcl M_k$. The range $X^k([0,\tau_k])\subset\mcl V\mcl M_k$ can potentially fail to be connected since $X^k$ can pass through $\infty$ on its way to $\bdy \mcl M_k$. 
From Conjecture~\ref{conj-scaling-limit}, we expect that under the Tutte embedding, $X^k|_{[0,\tau_k]}$ converges to $B|_{[0,\tau]}$ modulo time parametrization. In light of the preceding paragraph and the conjectural correspondence between singular points of $D_\Phi$ and ends of $\mcl M_k$, it is therefore natural to conjecture the following. 

\begin{conj} \label{conj-cut-pt}
The range of $X^k([0,\tau_k]) \subset \mcl V\mcl M_k$ of the continuous time random walk stopped when it hits $\bdy\mcl M_k$ is a.s.\ connected if $Q  > \sqrt{3/2}$ (equivalently, $\cc > 10$), and has positive probability to be non-connected if $Q < \sqrt{3/2}$ (equivalently, $\cc < 10$). 
\end{conj}

\begin{remark}
In the setting of Conjecture~\ref{conj-cut-pt}, we do not have a conjecture for what happens at the critical parameter value $Q = \sqrt{3/2}$. Analogous statements apply for Conjectures~\ref{conj-fsf} and~\ref{conj-perc} below.
\end{remark}

\subsubsection{Connectedness of the free uniform spanning forest}

Consider the free uniform spanning forest (FUSF) $\mcl T_k^{\op{FSF}}$ on $\mcl M_k$ (Definition~\ref{def-fsf} with $\frk c \equiv 1$). By Lemma~\ref{lem-fsf-lerw}, the paths in $\mcl T_k^{\op{FSF}}$ between points of the same connected component of $\mcl T^{\op{FSF}}$ can be obtained by loop-erasing segments of the continuous time random walk $X^k$ reflected off of $\infty$. Since we expect that $X^k$ should converge to Brownian motion, modulo time parametrization, under the Tutte embedding (Conjecture~\ref{conj-scaling-limit}), it is natural to guess that loop erasures of segments of $X^k$ converge to SLE$_2$-type curves. One reason why this is natural is that it was shown in~\cite{yy-lerw} that for a large class of planar graphs, the convergence of random walk to Brownian motion implies that convergence of loop erased random walk to SLE$_2$. 

The Hausdorff dimension of SLE$_2$ with respect to the Euclidean metric is $5/4$~\cite{beffara-dim}. Hence,~\eqref{eqn-singular-dim-upper} implies that an SLE$_2$ curve sampled independently from $\Phi$ a.s.\ passes through uncountably many singular points of $D_\Phi$ if $Q < \sqrt{5/2}$, and~\eqref{eqn-singular-dim-lower} suggests that an SLE$_2$ curve should not pass through any singular points of $D_\Phi$ if $Q > \sqrt{5/2}$. By combining this with the preceding paragraph, it is natural to guess that one can find paths in the FUSF $\mcl T_k^{\op{FSF}}$ joining any two vertices of $\mcl V\mcl M_k$ if $Q > \sqrt{5/2}$, but not if $Q < \sqrt{5/2}$. 

\begin{conj} \label{conj-fsf}
 If $Q > \sqrt{5/2}$ ($\cc > 16$), a.s.\ the FUSF on $\mcl M_k$ is connected (i.e., it consists of a single tree).
 If $Q < \sqrt{5/2}$ ($\cc < 16$), then on the event $\{\#\mcl V\mcl M_k = \infty\}$, a.s.\ the FUSF on $\mcl M_k$ has infinitely many connected components. 
\end{conj}

By way of comparison, we note that Pemantle~\cite{pemantle-ust} showed that the FUSF on $\BB Z^d$ a.s.\ consists of a single tree if $d\leq 4$, but a.s.\ has infinitely many connected components if $d \geq 5$. 

In contrast, we expect that the wired uniform spanning forest (WUSF) on $\mcl M_k$ has infinitely many connected components regardless of the value of $\cc$. The reason is that the WUSF should have components rooted at infinitely many of the uncountably many ends of $\mcl M_k$.

\begin{conj} \label{conj-wsf}
For all $\cc \in (1,25)$, the WUSF on $\mcl M_k$ has infinitely many connected components a.s.\ on the event $\{\#\mcl V\mcl M_k = \infty\}$. 
\end{conj}

\subsubsection{Percolation at criticality}

We next consider site percolation on $\mcl M_k$ with open boundary condition (we can make similar conjectures for bond percolation). We declare every vertex in $\bdy M_k$ to be open. For $p \in (0,1)$, we independently declare each vertex in $\mcl M_k\setminus \bdy\mcl M_k$ to be open with probability $p$ and closed with probability $1-p$. We call the connected components of the set of open vertices \textbf{open clusters}. We conjecture that there is a non-trivial phase transition.

\begin{conj} \label{conj-critical-prob}
For each $Q \in (0,2)$, there exists a critical probability $p_c = p_c(Q) \in (0,1)$ such that for each $p < p_c$ and each $k\in\BB N$, a.s.\ there are no infinite open clusters. 
On the other hand, for $p > p_c$, it holds for each $k\in\BB N$ that with positive probability, the open cluster containing $\bdy\mcl M_k$ is infinite.
\end{conj}

For site percolation on $\BB Z^d$, $d\geq 2$, it is notoriously difficult to determine what happens at the critical parameter value. For $d = 2$, at criticality a.s.\ there are no infinite open clusters. It is believed that the scaling limit of the interfaces between open and closed vertices is described by the conformal loop ensemble (CLE) with parameter $\kappa=6$. For the triangular lattice, this was proven by Smirnov~\cite{smirnov-cardy} and Camia and Newman~\cite{camia-newman-cle6}, but it remains open for the square lattice. In the higher-dimensional setting, it was proven by Fitzner and van der Hofstad~\cite{fv-d11} that a.s.\ there are no infinite clusters at criticality for $d\geq 11$, building on work by Hara and Slade~\cite{hs-high-dim-perc,hs-lace}. It is a famous open problem to determine whether there is percolation at criticality for $3\leq d\leq 10$. 

For percolation on the planar maps $\mcl M_k$, it is natural to guess (based on the results for the triangular lattice) that if we take $p = p_c(Q)$ to be the critical value as in Conjecture~\ref{conj-critical-prob}, then under the Tutte embedding the scaling limit of the percolation interfaces is a CLE$_6$ independent from the LQG field $\Phi$. In particular, the scaling limit of the open cluster containing $\bdy\mcl M_k$ should be the CLE$_6$ gasket, i.e., the fractal subset of $\ol{\BB D}$ obtained by removing the regions surrounded by the outermost CLE$_6$ loops. It is shown in~\cite{msw-gasket} that the Hausdorff dimension of the CLE$_6$ gasket is a.s.\ equal to $91/48$. Therefore, \eqref{eqn-singular-dim-lower} implies that if we sample a CLE$_6$ independently from $\Phi$, then for $Q < \sqrt{91/24}$, a.s.\ the CLE$_6$ gasket contains uncountably many singular points of $D_\Phi$. Furthermore~\eqref{eqn-singular-dim-upper} suggests that for $Q > \sqrt{91/24}$, the CLE$_6$ gasket should not contain any singular points of $D_\Phi$. This leads us to guess that the critical percolation cluster on $\mcl M_k$ which contains $\bdy\mcl M_k$ can reach $\infty$ if $Q < \sqrt{91/24}$, but not if $Q > \sqrt{91/24}$. That is, we should have a phase transition in the parameter $Q$ for whether or not there is percolation at criticality on $\mcl M_k$. 
 
\begin{conj} \label{conj-perc}
Consider critical ($p = p_c(Q)$) percolation on $\mcl M_k$ with open boundary condition. 
If $Q > \sqrt{91/24}$ ($\cc > 95/4$), then a.s.\ there is no infinite open cluster. 
If $Q < \sqrt{91/24}$ ($\cc < 95/4$), then with positive probability the open cluster containing $\bdy\mcl M_k$ is infinite and has uncountably many ends.   
\end{conj}

\subsubsection{Other models}

Using similar arguments as above, we can also make similar conjectures for other models on $\mcl M_k$. For example, critical percolation interfaces should have finite length if $Q > \sqrt{7/2}$ ($\cc > 22$), but not if $Q < \sqrt{7/2}$ ($\cc < 22$). Level lines of the discrete Gaussian free field (Definition~\ref{def-gff}) should be finite a.s.\ if $Q > \sqrt 3$ ($\cc > 19$), but not if $Q < \sqrt 3$ ($\cc < 19$). Level sets of the Green's function (Definition~\ref{def-green}) at fixed levels should be finite if $Q < \sqrt 2$ ($\cc < 13$), but not if $Q > \sqrt 2$ ($\cc > 13$).

\bibliography{cibib}
\bibliographystyle{alphaarxiv}

\end{document}